\documentclass[12pt]{article}
\usepackage{amsmath}
\usepackage{amsfonts,amssymb,mathrsfs, mathtools}
\usepackage{tikz}
\usetikzlibrary{matrix}
\usetikzlibrary{arrows} 

\usepackage{theorem}
\usepackage[all]{xy}

\numberwithin{equation}{section}

\numberwithin{figure}{section}

\newcommand{\coord}{\rm coord}

\newcommand{\aff}{\rm aff}

\newcommand{\Gr}{{\rm G r}}

\newcommand{\bfzero}{{\bf 0}}

\newcommand{\Lap}{\underline{\Delta}}

\newcommand{\Spec}{{\mathrm{Spec}}}
\newcommand{\mult}{{\mathrm{mult}}}

\renewcommand{\c}{{\circ}}

\newcommand{\Omb}{\Omega^{\bullet}}

\newcommand{\Cobar}{{\rm Cobar}}
\renewcommand{\Bar}{{\rm Bar}}

\newcommand{\Gra}{{\mathsf{Gra}}}
\newcommand{\GC}{{\mathsf{GC}}}

\newcommand{\gra}{{\mathsf{gra}}}

\newcommand{\fGC}{{\mathsf{fGC}}}

\newcommand{\fgraphs}{{\mathsf{fgraphs}}}
\newcommand{\graphs}{{\mathsf{graphs}}}

\newcommand{\GRT}{{\bf GRT}}

\newcommand{\Cbu}{C^{\bullet}}

\newcommand{\cCb}{\check{C}^{\bullet}}
\newcommand{\cC}{\check{C}}
\newcommand{\cpa}{\check{\partial}}
\newcommand{\cmB}{\check{{\mathfrak{B}}}}

\newcommand{\Def}{{\mathbf{Def}}}

\newcommand{\bDel}{{\mathbf{\Delta}}}

\newcommand{\sfDel}{{\mathsf{\Delta}}}

\newcommand{\sfD}{{\mathsf{D}}}
\newcommand{\sfA}{{\mathsf{A}}}

\newcommand{\Der}{{\mathrm{Der}}}
\newcommand{\coDer}{\mathsf{coDer}}

\newcommand{\Hom}{{\mathrm{Hom}}}
\newcommand{\Ih}{{\mathrm{Ih}}}
\newcommand{\ad}{{\mathrm{ad}}}

\newcommand{\AW}{{\mathrm{AW}}}

\newcommand{\id}{{\mathrm{id}}}

\newcommand{\Td}{{\mathrm{Td}}}

\newcommand{\Gal}{{\mathrm{Gal}}}

\newcommand{\hs}{{\heartsuit}}
\newcommand{\ds}{{\diamondsuit}}

\newcommand{\conn}{{\mathrm{conn}}}
\newcommand{\Conv}{{\mathrm{Conv}}}

\newcommand{\Sh}{{\mathrm{Sh}}}

\newcommand{\Mat}{{\rm Mat}}

\newcommand{\GL}{{\mathrm{GL}}}

\newcommand{\sgn}{{\mathrm{sgn}}}

\renewcommand{\L}{\langle}
\newcommand{\R}{\rangle}

\newcommand{\bs}{{\bf s}}
\newcommand{\bsi}{{\bf s}^{-1}\,}

\newcommand{\TS}{{\mathrm{TS}}}

\newcommand{\Ger}{{\mathsf{Ger}}}
\newcommand{\Lie}{{\mathsf{Lie}}}
\newcommand{\coLie}{{\mathsf{coLie}}}
\newcommand{\Com}{{\mathsf{Com}}}
\newcommand{\coCom}{{\mathsf{coCom}}}

\newcommand{\bul}{{\bullet}}

\newcommand{\al}{{\alpha}}
\newcommand{\la}{{\lambda}}
\newcommand{\io}{{\iota}}

\newcommand{\om}{{\omega}}
\newcommand{\Om}{{\Omega}}
\newcommand{\Te}{{\Theta}}

\newcommand{\si}{{\sigma}}
\newcommand{\ga}{{\gamma}}
\newcommand{\vf}{{\varphi}}
\newcommand{\ve}{{\varepsilon}}
\newcommand{\ka}{{\kappa}}
\newcommand{\vr}{{\varrho}}
\newcommand{\G}{{\Gamma}}

\newcommand{\bd}{{\bf d}}

\newcommand{\grt}{{\mathfrak{grt}}}
\newcommand{\lie}{{\mathfrak{lie}}}
\newcommand{\mt}{{\mathfrak{t}}}

\newcommand{\mgl}{{\mathfrak{gl}}}

\newcommand{\mmu}{{\mathfrak{u}}}
\newcommand{\mv}{{\mathfrak{v}}}

\newcommand{\ma}{{\mathfrak{a}}}

\newcommand{\md}{{\mathfrak{d}}}

\newcommand{\mI}{{\mathfrak{I}}}
\newcommand{\mT}{{\mathfrak{T}}}

\newcommand{\mS}{{\mathfrak{S}}}

\newcommand{\mR}{{\mathfrak{R}}}

\newcommand{\mB}{{\mathfrak{B}}}

\newcommand{\msA}{{\mathscr{A}}}
\newcommand{\msQ}{{\mathscr{Q}}}

\newcommand{\bv}{\overline{\mv}}

\newcommand{\pa}{{\partial}}
\newcommand{\tr}{{\mathrm{tr}}}

\newcommand{\cF}{{\mathcal{F}}}
\newcommand{\cD}{{\mathcal{D}}}
\newcommand{\cQ}{{\mathcal{Q}}}

\newcommand{\cA}{{\mathcal {A}}}
\newcommand{\cI}{{\mathcal {I}}}

\newcommand{\cS}{{\mathcal{S}}}

\newcommand{\cB}{{\mathcal{B}}}

\newcommand{\cT}{{\mathcal{T}}}
\newcommand{\cK}{{\mathcal{K}}}

\newcommand{\FR}{{\mathcal{FR}}}
\newcommand{\cV}{{\mathcal{V}}}
\newcommand{\cO}{{\mathcal O}}
\newcommand{\cN}{{\mathcal N}}

\newcommand{\bbA}{{\mathbb A}}
\newcommand{\bbC}{{\mathbb C}}
\newcommand{\bbR}{{\mathbb R}}
\newcommand{\bbZ}{{\mathbb Z}}
\newcommand{\bbK}{{\mathbb K}}
\newcommand{\bbQ}{{\mathbb Q}}

\newcommand{\ui}{\underline{i}}
\newcommand{\uj}{\underline{j}}
\newcommand{\uk}{\underline{k}}

\newcommand{\La}{{\Lambda}}

\newcommand{\te}{\theta}

\newcommand{\de}{{\delta}}
\newcommand{\D}{{\Delta}}

\newcommand{\wt}[1]{{\,\widetilde{#1}\,}}
\newcommand{\wh}[1]{{\,\widehat{#1}\,}}
\newcommand{\und}[1]{{\underline{#1}}}

\newcommand{\tR}{{\widetilde{R}}}

\newcommand{\tf}{{\widetilde{f}}}
\newcommand{\tg}{{\widetilde{g}}}

\newcommand{\dimens}{{\mathrm{dim}}}

\newcommand{\End}{{\mathrm{End}}}
\newcommand{\Tpoly}{\cT_\poly}
\newcommand{\poly}{{\mathrm{poly}}}

\date{}

\newcommand{\ed}{{\bullet\hspace{-0.05cm}-\hspace{-0.05cm}\bullet}}
\newcommand{\bb}{{\bullet\, \bullet}}

\newtheorem{defi}{Definition}[section]
\newtheorem{defp}[defi]{Proposition-Definition}

\newtheorem{thm}[defi]{Theorem}
\newtheorem{cor}[defi]{Corollary}

\newtheorem{prop}[defi]{Proposition}
\newtheorem{claim}[defi]{Claim}

\newtheorem{pty}[defi]{Property}

\theorembodyfont{\rm}
\newtheorem{example}[defi]{Example}
\newtheorem{remark}[defi]{Remark}
\newtheorem{warning}[defi]{Warning}

\theorembodyfont{\sc}

\makeatletter
\newcommand\qedsymbol{\hbox{$\Box$}}
\newcommand\qed{\relax\ifmmode\Box\else
  {\unskip\nobreak\hfil\penalty50\hskip1em\null\nobreak\hfil\qedsymbol
  \parfillskip=\z@\finalhyphendemerits=0\endgraf}\fi}

\newcommand\subqedsymbol{\hbox{$\triangledown$}}
\newcommand\subqed{\relax\ifmmode\triangledown\else
  {\unskip\nobreak\hfil\penalty50\hskip1em\null\nobreak\hfil\subqedsymbol
  \parfillskip=\z@\finalhyphendemerits=0\endgraf}\fi}
\makeatother

\newenvironment{proof}[1][{}]{\par\noindent Proof{#1}. }{\qed}

\title{Kontsevich's graph complex, GRT, and the deformation complex of the sheaf of polyvector fields}

\author{V.A. Dolgushev, C.L. Rogers, and T.H. Willwacher}

\begin{document}

\maketitle

\begin{flushright}
{\it To the memory of Boris Vasilievich Fedosov}
\end{flushright}
~\\[0.26cm]

\begin{abstract}
We generalize Kontsevich's construction \cite{K-conj} of 
$L_{\infty}$-derivations of polyvector fields from the affine space 
to an arbitrary smooth algebraic variety. More precisely, 
we construct a map (in the homotopy category) from Kontsevich's 
graph complex to the deformation complex of the sheaf of polyvector 
fields on a smooth algebraic variety. 
We show that the action of Deligne-Drinfeld elements of the Grothendieck-Teichm\"uller
Lie algebra on the cohomology of the sheaf of 
polyvector fields coincides with the action of 
odd components of the Chern character.  
Using this result, we deduce that the $\hat{A}$-genus in the
Calaque-Van den Bergh formula \cite{Damien} for the isomorphism between 
harmonic and Hochschild structures can be replaced by 
a generalized  $\hat{A}$-genus. 
\end{abstract}

~and graph complexes.\\[0.3cm]

\tableofcontents

\section{Introduction}
Inspired by Grothendieck's lego-game from \cite{Esq},  V. Drinfeld introduced 
in \cite{Drinfeld} a pro-unipotent algebraic group which he called 
the Grothendieck-Teichm\"uller group $\GRT$. This group is closely connected with 
the  absolute Galois group $\Gal(\overline{\bbQ}\,/\, \bbQ)$, 
it appears naturally in the study of moduli of algebraic curves, 
solutions of the Kashiwara-Vergne problem \cite{AT}, theory of motives \cite{Brown}, 
\cite{DG}, \cite{Furusho} and formal quantization 
procedures \cite{exhausting}, \cite{stable1}, \cite{char-classes}, \cite{Thomas}. 
The Lie algebra $\grt$ of $\GRT$ carries a natural grading by positive integers. 
Furthermore, according to \cite[Proposition 6.3]{Drinfeld}, $\grt$ has a non-zero 
vector $\si_{n}$ for every odd degree $n\ge 3$. We call $\si_n$'s Deligne-Drinfeld 
elements of $\grt$. 

In paper \cite{Thomas}, the third author established a link between the graph 
complex $\GC$ introduced in \cite{K-conj} by M. Kontsevich and the Lie 
algebra $\grt$ of the Grothendieck-Teichm\"uller group. More precisely, in 
\cite{Thomas}, it was shown that 
\begin{equation}
\label{H-GC-grt}
H^0(\GC) \cong \grt\,.
\end{equation}
 
In paper \cite{K-motives}, M. Kontsevich conjectured that the Grothendieck-Teichm\"uller 
group reveals itself in the extended moduli \cite{BK} of deformations of
an algebraic variety $X$ via the action of odd components  of 
the Chern character of $X$ on the cohomology of the sheaf of polyvector 
fields
\begin{equation}
\label{HTpoly-intro}
H^{\bul}(X, \cT_{\poly})\,.
\end{equation}
In this paper, we use the isomorphism \eqref{H-GC-grt} to
establish this fact for an arbitrary smooth algebraic variety $X$ 
over an algebraically closed field $\bbK$ of characteristic zero.    

More precisely, we define a map 
(in the homotopy category of dg Lie algebras) from $\GC$ to the deformation 
complex of the sheaf of polyvector fields $\cT_{\poly}$ on an arbitrary smooth 
algebraic variety $X$. This result generalizes Kontsevich's construction 
\cite[Section 5]{K-conj} from the case of affine space to the case of an 
arbitrary smooth algebraic variety.

Using a link \cite{Thomas} between the graph complex $\GC$ and 
the deformation complex of the operad $\Ger$, we prove that 
every cocycle  $\ga \in \GC$ gives us a derivation of the 
Gerstenhaber algebra \eqref{HTpoly-intro}.

Combining these results with the isomorphism \eqref{H-GC-grt}, 
we get a natural action of the Lie algebra $\grt$ on the cohomology \eqref{HTpoly-intro} 
of the sheaf of polyvector fields. In addition, we deduce that this action is compatible with
the Gerstenhaber algebra structure on \eqref{HTpoly-intro}.

We show that the action of Deligne-Drinfeld element $\si_n$ ($n$ odd $\ge 3$)  
of $\grt$ on  \eqref{HTpoly-intro} is given by a non-zero multiple of the 
contraction with the $n$-th component of the Chern character of $X$. 
This result confirms that the Grothendieck-Teichm\"uller
group indeed reveals itself in the extended moduli of deformations of $X$ in the 
way predicted by M. Kontsevich in \cite{K-motives}. Our results imply 
that the contraction of polyvector fields with any odd component of the Chern 
character induces a derivation of  \eqref{HTpoly-intro} with respect to 
the cup-product. This statement was formulated in \cite[Theorem 9]{K-motives} 
without a proof. 

We prove that the $\hat{A}$-genus in the
Calaque-Van den Bergh formula \cite{Damien} for the isomorphism between 
harmonic and Hochschild structures can be replaced by 
a generalized  $\hat{A}$-genus.
 
We give examples of algebraic varieties for which 
odd components of the Chern character act non-trivially on 
\eqref{HTpoly-intro}.  In particular, using  Theorem \ref{thm:main}, we 
show that smooth Calabi-Yau complete intersections in projective spaces 
provide us with a large supply of non-trivial representations of 
the Grothendieck-Teichm\"uller Lie algebra $\grt$.
This situation is strikingly different from what we have in the classical Duflo theory and in the 
classical Poisson geometry. Indeed, as remarked by M. Duflo (see \cite[Section 4.6]{K-motives}), the action of $\grt$ on Duflo isomorphisms is trivial for all (non-graded) Lie algebras. Furthermore, the authors are still unaware of any instance of a (non-graded) Poisson structure on which $\grt$ acts non-trivially. 
    
Finally, we show how Corollary \ref{cor:main} allows us to get 
some information about the Gerstenhaber algebra structure on 
\eqref{HTpoly-intro} when $X$ is a complete intersection in a projective 
space.

\paragraph*{Recent related results} We would like to mention two papers \cite{Alm-SM} 
and \cite{Jost} in which similar results were obtained.   

In paper \cite{Alm-SM}, J. Alm and S. Merkulov proved that, for an 
arbitrary formal Poisson structure $\pi$ on a smooth real manifold $M$
and an arbitrary element $g$ of the group $\GRT$, the Poisson 
cohomology $H^{\bul}(M, g(\pi))$ of $g(\pi)$ is isomorphic to the 
Poisson cohomology $H^{\bul}(M, \pi)$ of $\pi$ as a graded 
associative algebra.  
 
In paper \cite{Jost}, C. Jost described a large class of $L_{\infty}$-automorphisms of 
the Schouten algebra of polyvector fields on $\bbR^d$ which can be ``extended'' to 
$L_{\infty}$-automorphisms of the Schouten algebra of polyvector fields on an 
arbitrary smooth real manifold. Combining this result with the isomorphism \eqref{H-GC-grt}, 
C. Jost constructed an action the group $\GRT$ by $L_{\infty}$-automorphisms  
on the Schouten algebra of polyvector fields on an 
arbitrary smooth real manifold.

\paragraph*{Structure of the paper} 
In the remaining subsections of the Introduction, we fix notation 
and conventions.

Sections \ref{sec:cO-coord-aff} and \ref{sec:Fed-tensor} are 
devoted to the Fedosov resolution of the sheaf of tensor fields 
on a smooth algebraic variety. 

The key idea of this construction \cite{FTHC}, \cite{CEFT} has various 
incarnations and it is often referred to as the Gelfand-Fuchs 
trick \cite{GelFuchs} or Gelfand-Kazhdan formal geometry \cite{GelKazh} 
or mixed resolutions \cite{Ye}.  

The version given here is a modification of the construction 
proposed in \cite{VdB} by M. Van den Bergh.
The important advantage of our version is that we managed
to streamline  Van den Bergh's approach by avoiding completely 
the use of formal schemes and the use of jets. 
We believe that our modification of Van den Bergh's construction 
will be useful far beyond the scope of our paper. 

In Section \ref{sec:Atiyah}, we describe a convenient explicit representative 
of the Atiyah class of $X$ in the Fedosov resolution of the tensor algebra. 
In this section, we also observe that the Fedosov resolution allows us 
to represent this class by a global section of some sheaf unlike 
the conventional representative which is given by a 1-cochain 
in the \v{C}ech complex. 
  
In Section \ref{sec:fGC}, we recall the operad $\Gra$, the full graph 
complex $\fGC$, and Kontsevich's graph complex \cite[Section 5]{K-conj} $\GC$. 
We state the results of the third author from \cite{Thomas} which are used later in the text and 
introduce a couple of dg Lie algebras related to the full graph complex $\fGC$. 

Section \ref{sec:GC-DefTpoly} is devoted to the construction of 
a map $\Te$ of dg Lie algebras from Kontsevich's graph complex $\GC$ to 
the deformation complex of the dg sheaf $\FR$ which is quasi-isomorphic 
to the sheaf of polyvector fields on $X$. In this section, we consider 
the sheaf $\FR$ primarily with the Schouten-Nijenhuis bracket forgetting 
the cup product structure. However, in technical Section \ref{sec:aux}, we
extend the map $\Te$ to a map from an auxiliary dg Lie algebra linked to $\fGC$
to the deformation complex of $\FR$, where $\FR$ is viewed as a sheaf of 
dg Gerstenhaber algebras.   
 
In Section \ref{sec:cD-ga-deriv}, we prove that for every cocycle $\ga \in \GC$ the 
cocycle $\Te(\ga)$ induces a derivation of the Gerstenhaber algebra 
$H^{\bul}(X, \cT_{\poly})$. 

In Section \ref{sec:DD-action}, we give a geometric description of 
the action of Deligne-Drinfeld elements of $\grt$ on the Gerstenhaber algebra 
$H^{\bul}(X, \cT_{\poly})$. In this section, we also prove that the contraction with 
odd components of the Chern character induces derivations of the Gerstenhaber 
algebra $H^{\bul}(X, \cT_{\poly})$. 

In Section \ref{sec:harm-Hoch}, we generalize the result \cite{Damien} 
of D. Calaque and M. Van den Bergh on harmonic and Hochschild structures 
of a smooth algebraic variety. 

In Section \ref{sec:examples}, we give several examples which show 
that Theorem \ref{thm:main} and Theorem \ref{thm:A-hat} are non-trivial. 
Many of these examples can be found among complete 
intersections in a projective space. 

At the end of the paper we have several appendices. 

In Appendix \ref{app:Def-comp}, we briefly recall the notion of 
a homotopy $O$-algebra and the notion of the deformation complex of 
an $O$-algebra.

In Appendix \ref{app:sheaves}, we recall necessary constructions 
related to sheaves of algebras over an operad. More precisely, 
we review the Thom-Sullivan normalization and use it to define derived 
global sections for a dg sheaf $\cA$ of operadic algebras and the deformation 
complex of $\cA$. Although the Thom-Sullivan normalization is extremely convenient 
for proving general facts about derived global sections and the deformation
complex, in the bulk of our paper, we use the conventional Cech resolution.
This use is justified by Propositions \ref{prop:BG} and \ref{prop:O-deriv}.

In Appendix \ref{app:twisting},  we briefly recall twisting of shifted
Lie algebras and Gerstenhaber algebras 
by Maurer-Cartan elements. In this appendix, 
we also extend the twisting operation to a subspace 
of cochains in the deformation complexes of such algebras. 

Most of the material given in the appendices is standard and 
well known to specialists. However, various statements are 
hard to find in the literature in the desired generality. So we added these
appendices for convenience of the reader.

\subsection{Notation and conventions}
Throughout the paper $\bbK$ is an algebraically 
closed field of characteristic zero.

The notation $S_{n}$ is reserved for the symmetric group 
on $n$ letters and  $\Sh_{p_1, \dots, p_k}$ denotes 
the subset of $(p_1, \dots, p_k)$-shuffles 
in $S_n$, i.e.  $\Sh_{p_1, \dots, p_k}$ consists of 
elements $\si \in S_n$, $n= p_1 +p_2 + \dots + p_k$ such that 
$$
\begin{array}{c}
\si(1) <  \si(2) < \dots < \si(p_1),  \\[0.3cm]
\si(p_1+1) <  \si(p_1+2) < \dots < \si(p_1+p_2), \\[0.3cm]
\dots   \\[0.3cm]
\si(n-p_k+1) <  \si(n-p_k+2) < \dots < \si(n)\,.
\end{array}
$$

For algebraic structures considered in this paper, we use 
the following symmetric monoidal categories for which we 
tacitly assume the usual Koszul rule of signs: 
\begin{itemize}

\item  the category of $\bbZ$-graded $\bbK$-vector spaces,

\item the category of (possibly) unbounded cochain 
complexes of $\bbK$-vector spaces, 

\item the category of sheaves of $\bbZ$-graded $\bbK$-vector spaces, 

\item the category of sheaves of (possibly) unbounded cochain 
complexes of $\bbK$-vector spaces. 

\end{itemize}
In particular, we frequently use the ubiquitous
combination ``dg'' (differential graded) to refer to algebraic objects 
in the category of cochain complexes or the category of sheaves of cochain 
complexes. We often consider a graded vector space (resp. a sheaf  of graded vector spaces)
as the cochain complex (resp. the sheaf of cochain complexes) with the zero differential. 

For a homogeneous vector $v$ in 
a cochain complex $\cV$, $|v|$ denotes the degree of $v$\,.
Furthermore, we denote 
by $\bs$ (resp. $\bs^{-1}$) the operation of suspension (resp. desuspension), i.e.
$$
(\bs\, \cV)^{\bul} = \cV^{\bul - 1}\,, \qquad 
(\bs^{-1} \, \cV)^{\bul} = \cV^{\bul + 1}\,.
$$

We reserve the  notation $S(\cV)$ (resp. $\und{S}(\cV)$) for the symmetric algebra 
(resp. the truncated symmetric algebra) in $\cV$: 
\begin{equation}
\label{S-cV}
S(\cV) = \bbK ~\oplus~ \bigoplus_{n \ge 1} \big( \cV^{\otimes \, n} \big)_{S_n}\,,
\end{equation}
\begin{equation}
\label{und-S-cV}
\und{S}(\cV) = \bigoplus_{n \ge 1} \big( \cV^{\otimes \, n} \big)_{S_n}\,.
\end{equation}
 
The notation $\Lie$ (resp. $\Com$, $\Ger$) is reserved for the 
operad  governing Lie algebras (resp. commutative 
(and associative) algebras without unit, Gerstenhaber algebras without unit).
Dually, the notation $\coLie$ (resp. $\coCom$) is reserved for the 
cooperad governing Lie coalgebras (resp. cocommutative (and coassociative)
coalgebras without counit).

For an operad $O$ (resp. a cooperad $C$) and a cochain complex (or a sheaf 
of cochain complexes) $\cV$, the notation 
$O(\cV)$ (resp. $C(\cV)$) is reserved for the free $O$-algebra (resp. cofree $C$-coalgebra). 
Namely, 
\begin{equation}
\label{O-cV}
O(\cV) : = \bigoplus_{n \ge 0} \Big( O(n) \otimes \cV^{\otimes \, n} \Big)_{S_n}\,,
\end{equation}
\begin{equation}
\label{C-cV}
C(\cV) : = \bigoplus_{n \ge 0} \Big( C(n) \otimes \cV^{\otimes \, n} \Big)^{S_n}\,.
\end{equation}
For example\footnote{In our paper, we often identify invariants and coinvariants using
the fact that the underlying field $\bbK$ has characteristic zero.}
\begin{equation}
\label{coCom-cV}
\coCom(\cV) = \und{S}(\cV)\,. 
\end{equation}

For an augmented operad $O$ (resp. a coaugmented cooperad $C$) the 
notation $O_{\c}$ (resp. $C_{\c}$) is reserved for the kernel 
(resp. the cokernel) of the augmentation (resp. the coaugmentation).
For example, 
$$
\coCom_{\c}(n) = \begin{cases}
  \bbK \qquad {\rm if} ~~ n \ge 2  \,, \\
  \bfzero \qquad {\rm otherwise}\,.
\end{cases}
$$

We denote by $\La$ the endomorphism operad of the $1$-dimensional
vector space $\bs^{-1} \bbK$ placed in degree $-1$
\begin{equation}
\label{La}
\La = \End_{\bs^{-1} \bbK}\,.
\end{equation}
In other words,
$$
\La(n) = \bs^{1-n} \sgn_n\,,
$$
where $\sgn_n$ is the sign representation for the symmetric group $S_n$
We observe that the collection $\La$ is also naturally a cooperad. 

For a dg operad (resp. a dg cooperad) $P$ in we denote by $\La P$ the 
dg operad (resp. the dg cooperad) which is obtained from $P$ via tensoring  
with $\La$, i.e.
\begin{equation}
\label{La-P}
\La P (n) = \bs^{1-n} P(n) \otimes \sgn_n\,.
\end{equation}
For example, an algebra over $\La\Lie$ is
a graded vector space $V$ equipped with the binary operation: 
$$
\{\,,\, \} : V \otimes  V \to V 
$$
of degree $-1$ satisfying the identities:  
$$
\{v_1,v_2\} = (-1)^{|v_1| |v_2|} \{v_2, v_1\}
$$
$$
\{\{v_1, v_2\} , v_3\} +  
(-1)^{|v_1|(|v_2|+|v_3|)} \{\{v_2, v_3\} , v_1\} +
(-1)^{|v_3|(|v_1|+|v_2|)} \{\{v_3, v_1\} , v_2\}  = 0\,,
$$
where $v_1, v_2, v_3 $ are homogeneous vectors in $V$\,.

$\Ger^{\vee}$ denotes the Koszul dual cooperad \cite{Fresse}, 
\cite{GJ}, \cite{GK} for $\Ger$\,. It is known \cite{Hinich} that 
\begin{equation}
\label{Ger-vee}
\Ger^{\vee} = \La^2 \Ger^{*}\,,
\end{equation}
where $\Ger^{*}$ is obtained from the operad 
$\Ger$ by taking the linear dual.
In other words, algebras over the linear dual 
$(\Ger^{\vee})^*$ are very much like Gerstenhaber algebras 
except that the bracket carries degree $1$ and the multiplication 
carries degree $2$\,.  

The notation $\Cobar$ is reserved for the cobar construction
(see \cite[Section 3.7]{notes}).

A graph $\G$ is a pair $(V(\G), E(\G))$, 
where $V(\G)$ is a finite non-empty 
set and $E(\G)$ is a set of unordered pairs 
of elements of $V(\G)$. Elements of $V(\G)$ are called vertices 
and elements of $E(\G)$ are called edges. 
We say that a  graph $\G$ is {\it labeled} 
if it is equipped with a bijection between the set $V(\G)$ and 
the set of numbers $\{1,2, \dots, |V(\G)|\}$\,. We allow a graph 
with the empty set of edges. 
An {\it orientation} of $\G$ is a choice of directions on all edges of $\G$.

In this paper, $X$ denotes a smooth algebraic variety over $\bbK$. 
We denote by $\cO_X$ the 
structure sheaf on $X$, by $\cT_X$ (resp. $\cT^*_X$) the tangent 
(resp. cotangent) sheaf and by  
\begin{equation}
\label{T-poly-intro}
\cT_{\poly} = S_{\cO_X}(\bs \cT_X)
\end{equation}
the sheaf of polyvector fields.

\subsection{Trimming operadic algebras}
\label{sec:trimming}
Here we present a special construction which is used 
throughout the text. 

Let $\cV$ be an algebra over a (possibly colored) dg operad $O$ with the 
underlying graded operad $\wt{O}$\,. Furthermore, 
let $\{i_{\mv}\}_{\mv\in \cS}$ be a set of degree $-1$ derivations 
of the $\wt{O}$-algebra $\cV$\,. 

Let $\cV^{[\cS]}$ be the subcomplex of {\it basic elements} of $\cV$ 
with respect to $\cS$, i.e.
\begin{equation}
\label{cV-trimmed}
\cV^{[\cS]} : = \{w \in \cV ~ | ~ \forall~ \mv \in \cS \quad  i_\mv (w) = (\pa \circ i_{\mv} + i_{\mv} \circ \pa )(w) = 0\}\,,
\end{equation}
where $\pa$ is the differential on $\cV$\,.
It is easy to see the $O$-algebra structure on $\cV$ descends to the subcomplex 
of basic elements.

We call the construction of the $O$-algebra $\cV^{[\cS]}$ from an $O$-algebra 
$\cV$ and a set of degree $-1$ derivations  $\{i_{\mv}\}_{\mv\in \cS}$ {\it trimming}.

~\\

\noindent
\textbf{Memorial note:} Unfortunately, none of the authors met  
Boris Vasilievich Fedosov personally. However, we were 
influenced greatly by his works. We devote this paper in his memory.   

~\\

\noindent
\textbf{Acknowledgements:} V.A.D. acknowledges NSF grants 
DMS-0856196, DMS-1161867, the UC Regent's Faculty Fellowship, 
and the grant FASI RF 14.740.11.0347. C.L.R.'s work was partially supported 
by the NSF grant DMS-0856196, and by the German Research Foundation
(Deutsche Forschungsgemeinschaft (DFG)) through the Institutional Strategy 
of the University of G\"{o}ttingen.
T.W. was supported by a Fellowship of the Harvard Society of Fellows and by the Swiss National Science Foundation, grant PDAMP2\_137151. 
We would like to thank Damien Calaque 
for reminding us a well known fact about the Schouten-Nijenhuis bracket 
on the sheaf of polyvector fields on a Calabi-Yau variety.  We also thank the 
anonymous referees for carefully reading our manuscript and for many
useful suggestions.  V.A.D. would like 
to thank Oren Ben-Bassat and Tony Pantev for discussions.  
V.A.D.'s part of the work on this paper was
done when his parents-in-law Anna and Leonid Vainer were 
baby sitting Gary and Nathan.
V.A.D. would like to thank his parents-in-law for their help. 
V.A.D. is also thankful to the Free Library of Philadelphia 
in which a part of this paper was typed.

\section{The sheaves $\cO^{\coord}_X$ and  $\cO^{\aff}_X$ }
\label{sec:cO-coord-aff}
Let $X$ be a smooth algebraic variety of dimension $d$ over 
an algebraically closed field $\bbK$ of characteristic zero.
 
In this section, we recall the constructions of the sheaves $\cO^{\coord}_X$
and $\cO^{\aff}_X$ associated to the structure sheaf $\cO_X$ 
on $X$. We use these sheaves in Section \ref{sec:Fed-tensor} to 
construct the Fedosov resolution of the sheaf of tensor fields 
on $X$.

\subsection{The sheaf of $\cO_X$-algebras $\cO^{\bd}_X$}

Let $R$ be a commutative $\bbK$-algebra with identity. 

Let $\ui, \uj, \dots$ denote multi-indices 
$\ui= (i_1, i_2, \dots, i_d)$, $\uj= (j_1, j_2, \dots, j_d)$,
with $i_s, j_t$ being non-negative integers.  
For every multi-index $\ui$ the notation $|\,\ui\,|$ is reserved 
for the length of the multi-index $\ui$\,, namely 
\begin{equation}
\label{length}
|\, \ui\, | = i_1+ i_2 + \dots + i_d\,.
\end{equation}

\begin{defi}
\label{dfn:R-bd} 
Let us consider pairs
\begin{equation}
\label{pairs}
(f, \ui)
\end{equation}
with $f \in R$ and $\ui$ being a multi-index\,. 
We set $R^{\bd}$ to be the quotient of the 
free commutative algebra over $R$ 
generated by pairs (\ref{pairs}) with the respect 
to the ideal generated by relations\footnote{Here we use the obvious 
structure of Abelian semigroup on the set of multi-indices.}
\begin{equation}
\label{relations-unfolded}
\begin{array}{cc}
 ((f+g), \ui) = (f, \ui) + (g, \ui)\,,  &  
 \displaystyle (fg, {\ui}) =  \sum_{\uj+ \uk = \ui} (f, {\uj}) (g, {\uk})\,,    \\[0.5cm]
(f, (0,0, \dots, 0)) = f\,,  &    (\la, \ui) = 0 ~{\rm whenever}~ | \,\ui\, | \ge 1  
\end{array}
\end{equation}
$$
\forall \quad  f,g \in R,  \quad  \la \in \bbK\,.
$$
\end{defi}

It is convenient to use the notation $f_{\ui}$ for 
the pair $(f, \ui)$\,. So, from now on, we switch to 
this notation. 

It is also convenient to assign to each $f\in R$
the  following formal Taylor power series  
\begin{equation}
\label{f-tilde}
\tf =  \sum_{\ui} f_{\ui} t^{\ui} \in R^{\bd} [[t^1, t^2, \dots, t^d]]\,,
\end{equation}
where $t^{\ui}= (t^1)^{i_1} (t^2)^{i_2} \dots (t^d)^{i_d}$ and the 
summation goes over all multi-indices $\ui$\,. 

Using the notation $\wt{f}$ it is possible to 
rewrite the relations (\ref{relations-unfolded})
in the following condensed form 
\begin{equation}
\label{relations-R-d}
\wt{(f+g)} = \tf + \tg\,, \qquad 
\wt{fg} = \tf\, \tg\,, \qquad 
\wt{f}\, \Big|_{t=0} = f\,,
\qquad 
\wt{\la} = \la, 
\end{equation}
with $f,g \in R$ and $\la\in \bbK$\,.

Relations (\ref{relations-R-d}) imply that 
the formula
\begin{equation}
\label{homo:I}
I(f) =  \wt{f}
\end{equation}
defines an injective homomorphism of $\bbK$-algebras
from $R$ to $R^{\bd}[[t^1, t^2, \dots, t^d]]$\,.

The construction of the $R$-algebra $R^{\bd}$ 
from a $\bbK$-algebra $R$ is functorial   
in the following sense:  for every map of $\bbK$-algebras $\vf: R \to \tR$ we have 
a map (of $\bbK$-algebras)
$$
\vf^{\bd} : R^{\bd} \to \tR^{\bd}
$$ 
which makes the following diagram commutative. 
\begin{equation}
\label{bd-functorial}
\xymatrix@M=0.4pc{
R \ar[d] \ar[r]^{\vf} & \tR\ar[d] \\
R^{\bd} \ar[r]^{\vf^{\bd}} & \tR^{\bd}
}
\end{equation}

Let $f$ be a non-zero element of $R$\,. 
Applying the functoriality (\ref{bd-functorial}) to the 
natural map from $R$ to its localization $R_f$ we get the commutative diagram 
\begin{equation}
\label{R-R-f}
\xymatrix@M=0.4pc{
R \ar[d] \ar[r] & R_f \ar[d]\\
R^{\bd} \ar[r] & (R_f)^{\bd}
}
\end{equation}
of maps of $\bbK$-algebras.

Using the lower horizontal arrow in (\ref{R-R-f}) we produce 
the obvious map of $R_f$-modules: 
\begin{equation}
\label{psi-f}
\psi_{f} :  \left(R^{\bd}\right)_f  \to (R_f)^{\bd}\,.
\end{equation}
We claim that 
\begin{prop}
\label{prop:localize}
For every commutative $\bbK$-algebra $R$ and any non-zero element $f\in R$
the map $\psi_f$ (\ref{psi-f}) is an isomorphism of 
$R_f$-modules. 
\end{prop}
\begin{proof}
Since $f$ is invertible in $\left(R^{\bd}\right)_f$\,,
the element $\tf$ is invertible in the algebra $\big(R^{\bd} \big)_f [[t^1, \dots, t^d]]$\,.
Let us denote by $f^*_{\ui} \in \big(R^{\bd} \big)_f$ the coefficient in front of $t^{\ui}$ in 
the series $(\,\tf\,)^{-1} \in \big(R^{\bd} \big)_f[[t^1, \dots, t^d]]$\,.  For example,
$$
f^*_{(0,0, \dots, 0)} = \frac{1}{f}\,,
$$
and 
$$
f^*_{(1,0,\dots, 0)} = -\frac{f_{(1,0,\dots, 0)}}{f^2} \,. 
$$

Next, we claim that the formulas 
\begin{equation}
\label{nu-f}
\nu_f (a_{\ui}) := a_{\ui}\,, \qquad 
\nu_f \big( (f^{-1})_{\ui}  \big) : = f^*_{\ui}\,, \qquad a \in R 
\end{equation}
define a homomorphism of $R_f$-algebras 
$$
\nu_f :  \big(R_f \big)^{\bd}  \to \left(R^{\bd}\right)_f \,.
$$ 

Indeed, formulas (\ref{nu-f}) define $\nu_f$ on generators 
of $\big( R_f \big)^{\bd}$ and it is not hard to check $\nu_f$ respects all the 
defining relations. 
 
Furthermore it is not hard to see that $\nu_f$ is the inverse of 
$\psi_f$ (\ref{psi-f})\,.
\end{proof}

Proposition \ref{prop:localize} implies the following. 
\begin{cor}
\label{cor:cO-bd-sheaf}
Let $X$ be an algebraic variety over $\bbK$ and $U$ be an 
affine open subset of $X$ with $R_U= \cO_X(U)$\,. Furthermore, 
let $(R_U^{\bd})\,\tilde{}$ be the 
quasicoherent sheaf on $U$ corresponding to the $R_U$-module 
$R_U^{\bd}$\,. Then the formula 
$$
\cO^{\bd}_X \, \Big|_{U} : = (R_U^{\bd})\,\tilde{}
$$
defines a quasicoherent sheaf of $\cO_X$-algebras.  ~~$\Box$
\end{cor}

Let us remark that the definition of the sheaf $\cO^{\bd}_X$
makes perfect sense for an arbitrary (not necessarily smooth)
algebraic variety $X$\,.

For a smooth algebraic variety $X$ we have the following statement.
\begin{prop}
\label{prop:cO-bd-free}
Let $X$ be a smooth algebraic variety over $\bbK$
of dimension $d$\,. Furthermore, let  $U$ be an affine subset of $X$
which admits\footnote{Equivalently, we could say that $U$ admits an \'etale map to the 
affine space $\bbA^d_{\bbK}$\,.} 
a global system of parameters: 
\begin{equation}
\label{glob-param-here}
x^1, x^2, \dots, x^d \in \cO_X(U)\,.
\end{equation}
Then $\cO^{\bd}_X (U)$ is isomorphic to the polynomial algebra
over $\cO_X(U)$ in the generators 
\begin{equation}
\label{R-d-generators}
\big\{ x^a_{\ui} \big\}_{|\ui| \ge 1, ~ 1 \le a \le d}\,.
\end{equation}
\end{prop}
\begin{proof}
Let us set 
$$
R : = \cO_X(U)\,.
$$
Our goal is to show that the obvious map of commutative $R$-algebras
\begin{equation}
\label{obv-map}
\vr :  R \big[ \{ x^a_{\ui}  \}_{|\ui| \ge 1, ~ 1 \le a \le d} \big]
\to R^{\bd}  
\end{equation}
is an isomorphism. 

For this purpose we introduce increasing filtrations 
on both algebras $R \big[ \{ x^a_{\ui}  \}_{|\ui| \ge 1, ~ 1 \le a \le d} \big]$
and $R^{\bd}$\,.
$$
R = F^0  \subset F^1 \subset F^2 \subset F^3 \subset \dots \quad \subset 
R \big[ \{ x^a_{\ui}  \}_{|\ui| \ge 1, ~ 1 \le a \le d} \big]\,,
$$
$$
R = F^0 R^{\bd} \subset  F^1 R^{\bd} \subset F^2 R^{\bd} \subset \dots
\quad \subset R^{\bd}\,,
$$
where 
$$ 
F^m =  R \big[ \{ x^a_{\ui}  \}_{1\le |\ui| \le m, ~ 1 \le a \le d} \big] 
$$
and $F^m R^{\bd}$ is the quotient of the polynomial algebra
$$
R\big[ \{f_{\ui} \}_{f\in R, ~ |\ui| \le m} \big]
$$
by the relations of $R^{\bd}$ (\ref{relations-unfolded}) involving 
only the elements $f_{\ui}$ with $|\ui| \le m$\,.

Furthermore, for each $m \ge 0$ we introduce the ideal 
$\wt{I}^m$ (resp. $I^m$) of the $R$-algebra $F^m R^{\bd}$
(resp. $F^m$). The ideal $\wt{I}^m \subset F^m R^{\bd}$ is 
generated by elements $f_{\ui}$ where $f\in R$ and $1 \le |\ui| \le m$. 
The ideal $I^m \subset F^m$ is generated by elements $x^a_{\ui}$ 
for $1 \le a \le d$ and $|\ui| \le m$\,. For $m = 0$ we set 
$$
I^0 = \wt{I}^0 = \bfzero\,.
$$
 
It is clear that the map $\vr$ (\ref{obv-map}) is compatible with 
the filtrations and moreover
$$
\vr( I^m )  \subset  \wt{I}^m\,.
$$ 

Let us prove, by induction, that for each $m \ge 0$ the map $\vr$ (\ref{obv-map}) 
gives us an isomorphism from $F^m$ to $F^m R^{\bd}$ and 
$$
\vr(I^m) = \wt{I}^m\,.
$$

For $m=0$ the statement is obvious. 
Let assume that the desired statement holds for $m-1$\,.  

The $R$-algebra $F^m R^{\bd}$ is obtained from $F^{m-1} R^{\bd}$
via adjoining elements $f_{\ui}$ with $|\ui|=m$ and imposing the 
relations 
\begin{equation}
\label{relations-at-m}
\begin{array}{c}
(f + g)_{\ui} = f_{\ui} + g_{\ui} \\[0.3cm]
\la_{\ui} =0  \qquad \forall \quad \la \in \bbK\\[0.3cm]
(fg)_{\ui} = f_{\ui} g + f g_{\ui} + \dots 
\end{array}
\end{equation}
where $\dots$ stands for a sum of elements in the ideal $\wt{I}^{m-1}$\,.

Therefore, the quotient $R$-algebra 
\begin{equation}
\label{F-m-slash-I}
F^m R^{\bd} \big/ \wt{I}^{m-1}
\end{equation}
is isomorphic to the symmetric algebra
(over $R$) on $N(d,m)$-copies of the module 
$\Om^1_{\bbK}(R)$ of K\"ahler differentials
\begin{equation}
\label{Symm-Om-copies}
S_R\big(\, (\Om^1_{\bbK}(R))^{\oplus~ N(d,m)} \, \big)\,,
\end{equation}
where $N(d,m)$ is the total number\footnote{$N(d,m)$ is also 
the number of integer points inside the $(d-1)$-simplex 
$\{ (u_1, u_2, \dots, u_d), ~~ u_i\ge 0, ~~ u_1 + u_2 + \dots + u_d = m\}$\,.} 
of multi-indices $\ui$ of length $m$\,.

Let us consider the following commutative diagram of maps of commutative rings: 
\begin{equation}
\label{I-F-FR-d}
\xymatrix@M=0.4pc{
0 \ar[r] & I^{m-1} \ar[r] \ar[d]^{\vr} & F^m \ar[r] \ar[d]^{\vr}  &  F^m \big/ I^{m-1} \ar[r] \ar[d] & 0 \\ 
0 \ar[r] & \wt{I}^{m-1} \ar[r]  & F^m R^{\bd} \ar[r]   & F^m R^{\bd} \big/ \wt{I}^{m-1} \ar[r] & 0  
}
\end{equation}
Since $R$ has a global system of parameters (\ref{glob-param-here}),
$\Om^1_{\bbK}(R)$ is a free\footnote{This is an easy exercise from 
algebraic geometry.} $R$-module on the $1$-forms $d x^a$\,. 
Therefore, the right most vertical arrow in \eqref{I-F-FR-d} is an 
isomorphism. 

On the other hand, the left most vertical arrow is also an isomorphism 
by the induction assumption.

Thus, by the five lemma, the middle vertical arrow 
in \eqref{I-F-FR-d} is also an isomorphism.   

It remains to show that for every $f \in R$ and for every multi-index $\ui$ of length 
$m$ the element $f_{\ui}$ belongs to $\vr(I^m)$\,.

Since variables $x^1, x^2, \dots, x^d$ form a global system of parameters for $R$,
for each $f\in R$ there exists a unique collection of elements $\{ f_a \}_{1 \le a \le d} \in R$
such that 
$$
d f = \sum_{a=1}^d  f_a \, d x^a\,.
$$
Hence, using the above isomorphism between the quotient \eqref{F-m-slash-I}
and the symmetric algebra \eqref{Symm-Om-copies} we deduce that, for 
every multi-index $\ui$ of length $m$, 
\begin{equation}
\label{f-ui-x-ui}
f_{\ui} =  \sum_{a=1}^d  f_a \, x^a_{\ui} + \dots\,,
\end{equation}
where $\dots$ stands for a sum of elements in the ideal $\wt{I}^{m-1}$\,.

Thus, due to the inductive assumption, $f_{\ui}$ belongs to $\vr(I^m)$\,.

Proposition \ref{prop:cO-bd-free} is proven. 
\end{proof}

From now on, we assume that $X$ is a smooth algebraic variety over $\bbK$
of dimension $d$\,.

\subsection{The sheaf of $\cO_X$-algebras $\cO^{\coord}_X$} 
\label{subsec:cO-coord} 

Let $R$ be the ring of functions $\cO(U)$
on a smooth affine variety $U$ of dimension $d$\,. 
Let us assume that $U$ has a global system of parameters 
\begin{equation}
\label{system-param}
x^1, x^2, \dots, x^d \in R\,.
\end{equation}
For every $x^a$ we rewrite the formal series $\wt{x^a} \in R^{\bd}[[t^1, \dots, t^d]]$ 
as follows
\begin{equation}
\label{wt-x-a}
\wt{x^a} = x^a  + \sum_{b=1}^{d} x^a_{(b)} t^b + \sum_{\ui,~ |\,\ui \,| > 1}
x^a_{\ui} t^{\ui} 
\end{equation}
where $(b)$ denotes the multi-index $(0, \dots, 0, 1, 0, \dots 0)$
with $1$ placed in the $b$-th spot and $|\,\ui\,|$ is the 
length (\ref{length}) of the multi-index $\ui$\,.

\begin{defp}
\label{dfn:R-coord}
We define the $\bbK$-algebra $R^{\coord}$ as 
the localization of $R^{\bd}$ with respect to the element 
$$
\det || x^a_{(b)} ||\,.
$$ 
This definition does not depend on the choice 
of the system of parameters (\ref{system-param}).
\end{defp}

\begin{proof}
Since $x^1, x^2, \dots, x^d \in R$ form a global system of parameters 
on $U$ the module $\Om^1_{\bbK}(R)$ of K\"ahler differentials is freely 
generated by $\{ d x^a \}_{1\le a \le d}$\,. In particular, for every 
$f \in R$ the element $d f$ can be written uniquely in the form 
\begin{equation}
\label{df-decompos}
d  f = f_a \, d x^a\,, \qquad f_a \in R\,. 
\end{equation}

Let $y^1, \dots, y^d \in R$ be another global system of parameters.
The above observation implies that there exist elements 
$\La^a_b \in R$ such that 
\begin{equation}
\label{dy-dx-La}
d y^a = \La^a_b \,  d x^b\,.  
\end{equation}
Furthermore, since  the elements $y^1, \dots, y^d$ also 
form a global system of parameters, the $R$-valued matrix 
$||\La^a_b||$ has to be invertible. 

In order to prove the proposition, we observe that the 
operation 
\begin{equation}
\label{md-b}
\md_b (f) =  f_{(b)} : R \to R^{\bd}
\end{equation}
is a $\bbK$-linear derivation of $R$-modules. 

Therefore, for every $f\in R$ we have  
\begin{equation}
\label{md-f}
\md_b(f) = f_a \, x^a_{(b)}\,, 
\end{equation}
where $f_a \in R$ are the coefficients in the 
decomposition (\ref{df-decompos})\,.

Thus, for $y^a$ we have
\begin{equation}
\label{xab-La-yab}
y^a_{(b)} = \La^a_c \, x^c_{(b)} 
\end{equation}
and hence   
$$
\det || y^a_{(b)} || = \det ||\La^a_b|| \, \det || x^a_{(b)} ||\,.
$$ 

The desired statement follows immediately from the 
fact that the matrix $||\La^a_b||$ is invertible.
\end{proof}

Let $X$ be an arbitrary smooth algebraic variety over $\bbK$ and $U$ be an affine 
open subset of $X$ equipped with a global system of parameters. 
Furthermore, let $R_U = \cO_X(U)$ and $\left(R_U^{\coord}\right)\tilde{}\,$ be the quasi-coherent sheaf on $U$ corresponding to the $R_U$-module 
$R_U^{\coord}$\,. Combining Corollary \ref{cor:cO-bd-sheaf} with 
Proposition-Definition \ref{dfn:R-coord} we see that the formula 
$$
\cO^{\coord}_X \Big |_{U} := \left(R_U^{\coord}\right)\,\tilde{} 
$$ 
defines a  quasi-coherent sheaf $\cO^{\coord}_X$ of $\cO_X$-algebras over $X$\,.

Proposition \ref{prop:cO-bd-free} implies the following statement: 
\begin{cor}
\label{cor:cO-coord}
If $X$ is a smooth algebraic variety over $\bbK$
of dimension $d$ and  $U$ is an affine subset of $X$
which admits a global system of parameters 
\begin{equation}
\label{glob-param-cO-coord}
x^1, x^2, \dots, x^d \in \cO_X(U),
\end{equation}
then $ \cO^{\coord}_X(U) $ is isomorphic to the quotient 
\begin{equation}
\label{cO-coord}
\cO_X(U) \big[ \{ x^a_{\ui}  \}_{|\ui| \ge 1, ~ 1 \le a \le d}
 \cup \{ K \} \big] \big/ \L K\det ||x^a_{(b)} ||-1 \R
\end{equation}
of the polynomial algebra over $\cO_X(U) $ in the variables
$$
 \{ x^a_{\ui}  \}_{|\ui| \ge 1, ~ 1 \le a \le d}~ \cup ~ \{ K \}
$$
with respect to the ideal generated by the element $ K\det ||x^a_{(b)} || - 1$\,. $~~~\Box$
\end{cor}

\begin{remark}
\label{rem:cO-coord}
Following \cite[Section 6.1]{VdB}, one can think of $\cO^{\coord}_X(U)$ as
the ring of functions of the (infinite dimensional) affine scheme of formal coordinate 
systems on $U$\,.
\end{remark}

\subsection{The sheaf of $\cO_X$-algebras $\cO^{\aff}_X$} 
\label{subsec:cO-aff}
Let us start by observing that there is an obvious bijection 
between the set of multi-indices 
$$
\{\ui = (i_1, i _2, \dots, i_d) ~|~ i_s \ge 0, ~ |\,\ui \,| \ge 1 \}
$$
and the set of symmetric multi-indices 
\begin{equation}
\label{symm-multi}
\big\{ (a_1, a_2, \dots, a_k) ~|~ 1 \le a_t \le d, ~ k \ge 1 \big\}  
 \Big/    (\dots, a_s, \dots a_t,  \dots) =  (\dots, a_t, \dots a_s,  \dots)\,.
\end{equation}
This bijection assigns to the multi-index $\ui$ the symmetric multi-index 
\begin{equation}
\label{bijection-ui-ua}
(\underbrace{1,1,\dots, 1}_{i_1 ~{\rm times}}, \underbrace{2,2,\dots, 2}_{i_2 ~{\rm times}}, 
\dots, \underbrace{d,d,\dots, d}_{i_d ~{\rm times}} )\,.  
\end{equation}
Notice that $k$ in (\ref{symm-multi}) is exactly the length $|\,\ui\,|$ of 
the multi-index $\ui$\,.
 
For a symmetric multi-index  $(a_1, a_2, \dots, a_k)$ corresponding 
to $\ui = (i_1, i _2, \dots, i_d)$ and $f \in R$ we 
set\footnote{Here $R$ is the algebra of functions on a smooth affine variety $U$\,.}
\begin{equation}
\label{f-a1-dots-ak}
f_{(a_1, a_2, \dots, a_k)} : = i_1! \,i_2! \,\dots\, i_d! \,\, f_{\ui}\,. 
\end{equation}

It is clear that the $R$-algebra $R^{\bd}$ is the 
quotient of the free $R$-algebra in elements  
\begin{equation}
\label{set-f-a1-dots-ak}
\big\{ f_{(a_1, a_2, \dots, a_k)}  \big\}_{ 1 \le a_t \le d, ~ k \ge 1 }
\end{equation}
with respect to the ideal generated by the relations
\begin{equation}
\label{rel-f-a1-dots-ak}
f_{( \dots, a_i, a_{i+1}, \dots )} = 
f_{ (\dots, a_{i+1}, a_{i}, \dots )}\,, 
\qquad 
\la_{(a_1, a_2, \dots, a_k)} = 0\,,
\end{equation}
\begin{equation}
\label{rel1-f-a1-dots-ak}
\wt{(f+g)} = \wt{f} + \wt{g}\,,  \qquad
\wt{fg} = \wt{f} \, \wt{g}\,,
\end{equation}
where $\la \in \bbK$, $f,g \in R$ and 
\begin{equation}
\label{wt-f-fas}
\wt{f} = f +  \sum_{k \ge 1} 
\sum_{1 \le a_t \le d} \frac{1}{k!} f_{(a_1, a_2, \dots, a_k)} 
t^{a_1} t^{a_2} \dots t^{a_k} \in R^{\bd} [[t^1, t^2, \dots, t^d]]\,.
\end{equation}
Equation \eqref{f-a1-dots-ak} allows us to switch back and forth 
between the sets of generators $\{f_{\ui} \}_{|\ui| \ge 1}$ and 
\eqref{set-f-a1-dots-ak}.

We claim that 
\begin{prop}
\label{prop:GL-d-acts}
The formula
\begin{equation}
\label{GL-d-acts}
h (f_{(a_1, a_2, \dots, a_k)}) = 
\sum_{1 \le b_1, \dots, b_k \le d } h^{b_1}_{a_1} h^{b_2}_{a_2}
\dots   h^{b_k}_{a_k}  f_{(b_1, b_2, \dots, b_k)}\,,
\end{equation}
$$
h = || h^b_a ||\in \GL_d(\bbK)
$$
defines a left action of the affine algebraic group $\GL_d(\bbK)$
on the $R$-algebra $R^{\bd}$
This action extends in the obvious way to $R^{\coord}$ and to the sheaf 
of $\cO_X$-algebras $\cO^{\coord}_X$ for any smooth algebraic variety 
$X$ over $\bbK$\,.
\end{prop}
\begin{proof}
A direct computation shows that for every pair $h, h' \in \GL_d(\bbK)$
$$
h \big( h'(f_{(a_1, a_2, \dots, a_k)}) \big) = h h' (f_{(a_1, a_2, \dots, a_k)})\,. 
$$

It is also clear that the ideal generated by relations \eqref{rel-f-a1-dots-ak} and
\eqref{rel1-f-a1-dots-ak} is closed with the respect to the action of  $\GL_d(\bbK)$\,.

Thus, formula \eqref{GL-d-acts} indeed defines a left action of $\GL_d(\bbK)$
on the $R$-algebra $R^{\bd}$\,.

Let us recall that $R^{\coord}$ is obtained from $R^{\bd}$ by localizing 
with respect to the element 
$$
\det ||x^a_{(b)}||\,, 
$$
where $x^1, \dots, x^d$ is the global system of parameters 
for $\Spec(R)$\,.

Since for every $h \in \GL_d(\bbK)$
$$
h (x^a_{(b)}) = \sum_{c=1}^d h^c_b \, x^a_{(c)}
$$
the action  \eqref{GL-d-acts} of $\GL_d(\bbK)$ extends to 
the localization $R^{\coord}$ of $R^{\bd}$\,.

This action also obviously extends to the sheaf $\cO^{\coord}_X$
of $\cO_X$-algebras on any smooth algebraic variety $X$ over $\bbK$.
\end{proof}
\begin{remark}
\label{rem:extending-scalars}
We would like to mention that formula (\ref{GL-d-acts})
makes sense if $h \in \GL_d(R^{\coord})$ and $f_{(a_1, a_2, \dots, a_k)}$
are considered as elements of $R^{\coord}$. 
\end{remark}

Using the action \eqref{GL-d-acts} of  $\GL_d(\bbK)$ on $\cO^{\coord}_X$, we define 
yet another sheaf of $\cO_X$-algebras $\cO^{\aff}_X$:
\begin{defi}
\label{dfn:cO-aff}
Let $X$ be a smooth algebraic variety over $\bbK$\,. 
The sheaf $\cO_X^{\aff}$ is the subsheaf of $\GL_d(\bbK)$-invariant  sections of 
$\cO_X^{\coord}$\,.
\end{defi}
Since any section of the structure sheaf $\cO_X$ is 
obviously invariant under the $\GL_d(\bbK)$-action, the sheaf 
$\cO_X^{\aff}$ is naturally a sheaf of $\cO_X$-algebras.

\begin{remark}
\label{rem:cO-aff}
Let $U$ be an open affine subset of $X$.
Following \cite[Section 6.3]{VdB}, one can think of $\cO^{\aff}_X(U)$ as
the ring of functions on the (infinite dimensional) affine scheme of formal affine
systems on $U$\,.
\end{remark}

Let us observe that the $R$-algebra $R^{\bd}[[t^1, \dots t^d]]$ carries two left  $\GL_d(\bbK)$-actions:
the first one is obtained by extending the action \eqref{GL-d-acts} by 
$R[[t^1, \dots t^d]]$-linearity; the second one is obtained by 
setting  
\begin{equation}
\label{msA-h}
\msA_h (t^a) = \sum_{b=1}^d (h^{-1})^a_b \, t^b\,, \qquad
\msA_h (f_{\ui}) = 0\,, \qquad \forall ~~\ui\,. 
\end{equation}

\begin{remark}
\label{rem:actions-commute}
It is easy to see that the above actions of 
$\GL_d(\bbK)$ on  $R^{\bd}[[t^1, \dots t^d]]$ 
commute and hence the formula
\begin{equation}
\label{combined-action}
F \in R^{\bd}[[t^1, \dots t^d]] ~ \mapsto ~
h \circ \msA_{h} (F) \in R^{\bd}[[t^1, \dots t^d]]
\end{equation}
defines another left action on the $R$-algebra $R^{\bd}[[t^1, \dots t^d]]$\,.
Let us also remark that for every $f \in R$ and $h \in \GL_d(\bbK)$ 
we have 
\begin{equation}
\label{h-msA-h-wt-f}
h \circ \msA_{h}  (\,\wt{f} \,) =\wt{f}\,.  
\end{equation}
\end{remark}

Differentiating the actions \eqref{GL-d-acts}, \eqref{msA-h} of $\GL_d(\bbK)$ 
on $R^{\bd}$ and $R^{\bd}[[t^1, \dots t^d]]$, respectively we obtain the corresponding 
actions of the Lie algebra  $\mgl_d(\bbK)$. The latter action is given by the assignment 
\begin{equation}
\label{gl-d}
\mv = ||\mv^a_b|| ~ \mapsto ~ - \mv^a_b t^b \frac{\pa}{\pa t^a} \in 
\Der_{R^{\bd}} \big( R^{\bd}[[t^1, \dots t^d]] \big)
\end{equation}
and the former 
$$ 
\mv  \mapsto \bv \in  \Der_{R} (R^{\bd}) 
$$ 
is defined by declaring that 
\begin{equation}
\label{gl-d-acts}
\sum_{\ui} \bv(f_{\ui}) t^{\ui} + \mv(\tf)  = 0\,, 
\qquad \forall~~ f \in R\,.
\end{equation}

The  actions \eqref{GL-d-acts}, \eqref{msA-h}, and 
\eqref{combined-action} of $\GL_d(\bbK)$ on $R^{\bd}[[t^1, \dots t^d]]$
extend in the obvious way to left actions on the 
$R$-algebra $R^{\coord}[[t^1, \dots t^d]]$\,.
By abuse of notation, we will denote by $\bv$ the 
$R$-derivation of $R^{\coord}$ corresponding to the action 
\eqref{GL-d-acts} of $\mv  \in \mgl_d(\bbK)$\,.

To give a local description of the sheaf of  $\cO_X$-algebras
$\cO_X^{\aff}$, we consider an affine subset $U \subset X$
which admits a global system of parameters 
\begin{equation}
\label{glob-param-cO-aff}
x^1, x^2, \dots, x^d \in \cO_X(U)\,.
\end{equation}

Let us denote by $u_x$ the invertible $d\times d$-matrix 
with entries $x^a_{(b)}$:
\begin{equation}
\label{u-x}
u_x = ||x^a_{(b)} || \in \GL_d(R^{\coord})\,.
\end{equation}

It is obvious that the elements\footnote{Here we use Remark \ref{rem:extending-scalars}.} 
\begin{equation}
\label{gener-cO-aff}
u^{-1}_x  (x^a_{\ui})\,, \qquad 1 \le a \le d, \qquad |\ui| \ge 2 
\end{equation}
are $\GL_d(\bbK)$-invariant.  In other words, 
$$
u^{-1}_x  (x^a_{\ui}) \subset  \cO^{\aff}_X(U)\,.
$$

For the sheaf $\cO^{\aff}_X$ we have:
\begin{prop}[Proposition 6.3.1, \cite{VdB}]
\label{prop:affine-space}
Let $X$ be a smooth algebraic variety over $\bbK$ of 
dimension $d$ and let $U$ be an affine subset of $X$ which 
admits a global system of parameters 
\eqref{glob-param-cO-aff}.  Then the map 
$$
\vr^{\aff}  : \cO_X(U)\Big[ \, \{ y^a_{\ui} \}_{ |\ui| \ge 2, ~  1 \le a \le d }\, \Big] \to  \cO^{\aff}_X(U) 
$$
\begin{equation}
\label{vr-aff}
\vr^{\aff} (y^a_{\ui})  : =  u^{-1}_x  (x^a_{\ui})
\end{equation}
is an isomorphism of $\cO_X(U)$-algebras. 
The sheaf $\cO^{\aff}_X$ can be equivalently defined 
as the sheaf of $\mgl_d(\bbK)$-invariant sections of 
$\cO^{\coord}_X$\,.
\end{prop}
\begin{proof}
Let $R = \cO_X(U)$\,. Due to Corollary \ref{cor:cO-coord}, the commutative algebra
$R^{\coord}$ is isomorphic to the quotient 
\begin{equation}
\label{R-coord}
R \big[ \{ x^a_{\ui}  \}_{|\ui| \ge 1, ~ 1 \le a \le d}
 \cup \{ K \} \big] \big/ \L K\det ||x^a_{(b)} ||-1 \R
\end{equation}

The group $\GL_d(\bbK)$ acts on generators $x^a_{\ui}$ according 
to formula (\ref{GL-d-acts}) and $K$ transforms as 
$$
K \mapsto K / \det(h)\,, 
$$ 
where $h \in \GL_d(\bbK)$\,.

To describe the algebra $R^{\aff} = (R^{\coord})^{\GL_d(\bbK)}$ we consider the 
following isomorphism of $R$-algebras  
\begin{equation}
\label{R-coord-isom}
 \si: R \Big[ \{ y^a_{\ui}  \}_{|\ui| \ge 2, ~ 1 \le a \le d} \cup 
 \{ x^a_{(b)} \}_{1\le a, b \le d} 
 \cup \{ K \} \Big] \, \big/ \, \langle K\det ||x^a_{(b)} ||-1 \rangle 
 \to  R^{\coord}
\end{equation}
$$
\si ( y^a_{\ui}) = u^{-1}_x  (x^a_{\ui}), \qquad 
\si( x^a_{(b)} ) =  x^a_{(b)}, \qquad 
\si(K) = K \,. 
$$

The group $\GL_d(\bbK)$ acts on the generators $y^a_{\ui}$, $x^a_{(b)}$, 
$K$ in the following way
$$
y^a_{\ui} \mapsto y^a_{\ui}\,, \qquad 
x^a_{(b)} \mapsto  h^c_{b} \, x^a_{(c)}\,, \qquad
K \mapsto K / \det(h)\,,
$$ 
where $h \in \GL_d(\bbK)$\,.

Thus $R^{\coord}$ is isomorphic to 
\begin{equation}
\label{R-coord-new}
\cO(\GL_d(\bbK)) \otimes_{\bbK} 
R \Big[ \{ y^a_{\ui}  \}_{|\ui| \ge 2, ~ 1 \le a \le d} \Big]\,, 
\end{equation}
where $\cO(\GL_d(\bbK))$ is the algebra of regular functions 
on the algebraic group $\GL_d(\bbK)$ and the $\GL_d(\bbK)$-action 
on (\ref{R-coord-new}) is given by right translations on $\GL_d(\bbK)$\,.

Since $\big(\cO(\GL_d(\bbK)) \big)^{\GL_d(\bbK)} = \bbK$, we immediately 
conclude that equation \eqref{vr-aff}, indeed defines an isomorphism 
\begin{equation}
\label{R-aff}
R \Big[ \{ y^a_{\ui}  \}_{|\ui| \ge 2, ~ 1 \le a \le d} \Big] \cong R^{\aff}\,.
\end{equation}

It is also clear that  $\big(\cO(\GL_d(\bbK)) \big)^{\mgl_d(\bbK)} = \bbK$. 
Hence, 
$$
R^{\aff}  = \Big(  R^{\coord} \Big)^{\mgl_d(\bbK)}\,.
$$

Proposition \ref{prop:affine-space} is proven. 
\end{proof}

\subsection{The canonical flat connection on $\Omb(\cO^{\coord}_X)[[t^1, \dots t^d]]$}
Let us consider the algebra 
\begin{equation}
\label{C-coord}
\Omb_{\bbK}(\cO^{\coord}_X)
\end{equation}
of exterior forms of the sheaf $\cO^{\coord}_X$
of $\bbK$-algebras. We denote by $d$ the de Rham 
differential on $\Omb_{\bbK}(\cO^{\coord}_X)$\,.

We claim that 
\begin{thm}
\label{thm:om}
There exists a unique (degree $1$) 
$\Omb_{\bbK}(\cO^{\coord}_X)$-linear
continuous\footnote{$\Omb_{\bbK}(\cO^{\coord}_X)[[t^1, \dots t^d]]$ carries the obvious 
$t$-adic topology.}
derivation
$$
\om : \Omb_{\bbK}(\cO^{\coord}_X)[[t^1, \dots t^d]] \to  
\Omb_{\bbK}(\cO^{\coord}_X)[[t^1, \dots t^d]]
$$ 
such that
\begin{equation}
\label{tilde-f-is-flat}
d \wt{f} + \om(\wt{f}) = 0 
\end{equation}
for all local sections $f$ of the sheaf $\cO_X$\,. 
In addition, we have 
\begin{equation}
\label{om-is-flat}
(d + \om)^2 = 0\,.
\end{equation}
\end{thm}
\begin{proof}
A degree $1$ continuous $\Omb_{\bbK}(\cO^{\coord}_X)$-linear derivation
$$
\om : \Omb_{\bbK}(\cO^{\coord}_X)[[t^1, \dots, t^d]] \to  
\Omb_{\bbK}(\cO^{\coord}_X)[[t^1, \dots, t^d]]
$$ 
is uniquely determined by a collection of 1-forms: 
$$
\om^a \in \G(X, \Om^1(\cO^{\coord}_X))[[t^1, \dots, t^d]]
$$
via the equation 
$$
\om = \sum_{a=1}^d \om^a  \frac{\pa}{\pa t^a}\,.
$$

We will first define $\om^a$ in terms 
a local system of parameters. Next, we will prove 
equation (\ref{om-is-flat}) and equation (\ref{tilde-f-is-flat}). 
Finally, we will deduce that  $\om$ does not depend on 
the choice of the local system.

Let $U \subset X$ be an affine subset of $X$ 
with a global system of parameters $\{x^1, \dots, x^d\}$  and let
$R = \cO_X(U)$\,.

Since the matrix 
$$
||x^a_{(b)}|| 
$$
is invertible in $\Mat_d(R^{\coord})$, the matrix
\begin{equation}
\label{pa-x-pa-t}
J_x =
\Big | \Big| \frac{\pa \wt{x^a}}{\pa t^b} \Big | \Big| 
\end{equation}
is invertible in $\Mat_d(\,R^{\coord}[[t^1, \dots, t^d]]\,)$\,.

Using this observation, it is easy to see that 
\begin{equation}
\label{om-a-x}
\om^a = - (J_x^{-1})^a_b \sum_{\ui} d x^b_{\ui} t^{\ui}
\end{equation}
is the unique solution of the system of equations: 
\begin{equation}
\label{x-tilde-flat}
d \wt{x^a} + \om(\wt{x^a}) = 0\,, \qquad 1 \le a \le d\,. 
\end{equation}

To prove \eqref{tilde-f-is-flat} we observe that the map
\begin{equation}
\label{conn-circ-I}
(d + \om \cdot) \circ I : R \to  \Om^1(R^{\coord})[[t^1, \dots, t^d]]
\end{equation}
is a $\bbK$-linear derivation of $R$-modules where the 
$R$-module structure on the target of \eqref{conn-circ-I}
is defined by the formula
$$
a \cdot v = I(a) \, v \,, \qquad a\in R, \quad v \in \Om^1(R^{\coord})[[t^1, \dots, t^d]]\,.
$$ 

Therefore, for every $f\in \cO_X(U)$ we have 
$$
d \wt{f} + \om(\wt{f}) = I(f_a) \, (d \wt{x^a} + \om(\wt{x^a})) 
$$
where $f_a\in R$ are the coefficients in the 
decomposition (\ref{df-decompos})\,.

Thus equation \eqref{tilde-f-is-flat} follows from \eqref{x-tilde-flat}.

To prove equation (\ref{om-is-flat}), we remark that 
it is equivalent to  
\begin{equation}
\label{om-is-flat1}
d\, \om^a  + \om^b \wedge \frac{\pa  \om^a }{\pa t^b}   = 0 \,.
\end{equation}
The latter equation can be verified by a direct computation
using the obvious identities:
$$
d \, (J_x^{-1})^a_b = 
-  (J_x^{-1})^a_{a_1} \, \frac{\pa}{\pa t^{b_1}} \, d( \wt{x^{a_1}})\, (J_x^{-1})^{b_1}_b\,,
$$
$$
 \frac{\pa}{\pa t^c}  (J_x^{-1})^a_b = 
 -  (J_x^{-1})^a_{a_1} \,  \frac{\pa^2 \wt{x^{a_1}} }{ \pa t^c \pa t^{b_1}}   \,
 ( J_x^{-1})^{b_1}_b
$$
and the symmetry of the expression 
$$ 
\frac{\pa^2 \wt{x^{a}} }{ \pa t^b \pa t^c } 
$$
in the indices $b$ and $c$\,.

It remains to show that $\om$ does not depend on the 
choice of the system of parameters. 
Let $\{y^1, \dots, y^d\}$ be another local system of 
parameters. Equation (\ref{tilde-f-is-flat}) implies that 
\begin{equation}
\label{y-tilde-flat}
d \wt{y^a} + \om(\wt{y^a}) = 0\,, \qquad 1 \le a \le d\,. 
\end{equation}
But this system has the unique solution:
\begin{equation}
\label{om-a-y}
\om^a = - (J^{-1}_y)^a_b \sum_{\ui} d y^b_{\ui} t^{\ui}\,.
\end{equation}
Thus the construction of $\om$ does not depend on 
the choice of the local system of parameters and
the theorem is proven. 
\end{proof}

Let $R$ be a smooth affine $\bbK$-algebra.
Recall that $\mgl_d(\bbK)$ acts by $R$-derivations
on $R^{\bd}$ and hence by $R$-derivations on $R^{\coord}$\,. 
As above, we denote by $\bv$ the $R$-derivation of 
$R^{\coord}$ corresponding to the action \eqref{GL-d-acts} of 
$\mv \in \mgl_d(\bbK)$\,. Furthermore, we denote by $i_{\bv}$ the corresponding
contraction operator on $\Omb_{\bbK}(R^{\coord})$.  

Using Theorem \ref{thm:om} we deduce the following. 
\begin{cor}
\label{cor:gl-d-om}
Let $X$ be a smooth algebraic variety over $\bbK$  
of dimension $d$ and $\om$ be the derivation of 
$\Omb_{\bbK}(\cO^{\coord}_X)[[t^1, \dots, t^d]]$ from Theorem \ref{thm:om}. 
Then for every $\mv \in \mgl_d(\bbK)$ we have: 
\begin{equation}
\label{eq:gl-d-om}
i_{\bv}\, \om = - \mv^a_b t^b \frac{\pa}{\pa t^a}\,. 
\end{equation}
\end{cor}
\begin{proof}
Using equations \eqref{gl-d}, (\ref{gl-d-acts}) and (\ref{om-a-x}) we deduce
\begin{align*}
i_{\bv} (\om^a) &=
- (J^{-1}_x )^a_b \sum_{\ui} \bv(x^b_{\ui}) t^{\ui}  
= -(J^{-1}_x)^a_b \, \mv^c_{c'} t^{c'}  \frac{\pa}{\pa t^c} \,  \sum_{\ui} x^b_{\ui} t^{\ui}
\\
&= 
- \mv^c_{c'} t^{c'} 
(J^{-1}_x)^a_b   \frac{\pa \wt{x^b}}{\pa t^c}  
= - \mv^c_{c'} t^{c'} \de^a_c = - \mv^a_{c'} t^{c'}\,. 
\end{align*}
Hence 
$$
i_{\bv} (\om^a) \frac{\pa}{\pa t^a} =   - \mv^a_{b} t^{b} \frac{\pa}{\pa t^a}
$$
and the corollary is proven.  
\end{proof}

\section{The Fedosov resolution of the tensor algebra of a smooth variety}
\label{sec:Fed-tensor}

In this section we present the Fedosov resolution of the 
sheaf of tensor fields on a smooth algebraic variety over 
an arbitrary algebraically closed field $\bbK$ of characteristic zero. 

Recall that $\cT_X$ (resp. $\cT^*_X$) denotes the tangent 
(resp. cotangent) sheaf on a smooth algebraic variety $X$\,.
We denote by $\cT^{ p,q}_X$ the sheaf of $p$-contravariant and 
$q$-covariant tensor fields on $X$, i.e. 
\begin{equation}
\label{cT-p-q}
\cT^{ p,q}_X : =\underbrace{\cT_X \otimes_{\cO_X} \dots \otimes_{\cO_X}
\cT_X}_{p \textrm{ times}} \otimes_{\cO_X}
\underbrace{  \cT^*_X  \otimes_{\cO_X} \dots  \otimes_{\cO_X} \cT^*_X }_{q \textrm{ times}} \,.
\end{equation}

For example, $\cT^{0,0}_X = \cO_X$; $\cT^{1,0}_X$ (resp. $\cT^{0,1}_X$) is 
the tangent (resp. cotangent) sheaf on $X$; and $\cT^{1,1}_X$ is the 
sheaf of endomorphisms of the tangent sheaf $\cT_X$ on $X$\,.

Under all possible contraction operations, the tensor product, 
and the action of the group $S_p \times S_q$,
the collection of sheaves \eqref{cT-p-q} carry an
algebraic structure. This algebraic structure is 
governed by a colored\footnote{The set of colors for 
$\mT$ is the set of pairs of non-negative integers $(p,q)$\,.} operad 
which we denote by $\mT$\,. We call the $\mT$-algebra 
\begin{equation}
\label{tensor-alg-X}
\Big\{ \cT^{ p,q}_X \Big\}_{p,q \ge 0}
\end{equation}
{\it the tensor algebra} of $X$\,. 

The goal of this section is to construct the Fedosov resolution 
of the tensor algebra for an arbitrary smooth algebraic variety $X$
of dimension $d$\,.

For this purpose, we let $P= \bbK[[t^1, \dots, t^d]]$ be the (topological) algebra 
of formal Taylor power series in auxiliary variables $t^1, \dots, t^d$\,. 

Next, we set 
\begin{equation}
\label{T-p-q}
T^{p,q} : = \underbrace{\Der_{\bbK}(P)  \otimes_{P} \dots \otimes_{P}
\Der_{\bbK}(P) }_{p \textrm{ times}}  \otimes_{P}
\underbrace{  \Om^1_{\bbK}(P)  \otimes_{P} \dots 
\otimes_{P}  \Om^1_{\bbK}(P) }_{q \textrm{ times}}\,,
\end{equation}
where $\Der_{\bbK}(P)$ is the $P$-module of derivations of $P$ 
and 
$$
 \Om^1_{\bbK}(P) =  \Hom_P ( \Der_{\bbK}(P), P)\,.
$$

In other words, elements of $T^{p,q}$ have the form 
$$
v =
\sum_{ 1\le  a_t, b_s \le d} v^{a_1 \dots a_p}_{b_1 \dots b_q} \,
\pa_{t^{a_1}} \otimes \dots \otimes \pa_{t^{a_p}} 
\otimes d t^{b_1} \otimes \dots \otimes d t^{b_q}\,,
$$
where the components $v^{a_1 \dots a_p}_{b_1 \dots b_q} \in P$\,.

For a derivation 
$$
w = w^c(t) \pa_{t^c} \in \Der_{\bbK}(P)
$$ 
and an element $v \in T^{p,q}$ we denote by $L_w (v)$
the Lie derivative of $v$ along $w$\,. Recall that 
$L_w(v)$ has the following components:
\begin{equation}
\label{L-w-v}
L_w (v)^{a_1 \dots a_p}_{b_1 \dots b_q} = 
\sum_{c=1}^d \, w^c(t) \, \pa_{t^c} v^{a_1 \dots a_p}_{b_1 \dots b_q}(t) -
\sum_{c=1}^d \sum_{i=1}^{p} \, (\pa_{t^c} w^{a_i}(t))\,
 v^{a_1 \dots a_{i-1} \,c\, a_{i+1} \dots a_p}_{b_1 \dots b_q}(t)
\end{equation}
$$
+ \sum_{c=1}^d \sum_{j=1}^{q} \, (\pa_{t^{b_j}} w^{c}(t))\,
 v^{a_1 \dots a_p}_{b_1 \dots b_{j-1} \, c\, b_{j+1} \dots b_q}(t)\,.
$$

It is clear that the collection 
\begin{equation}
\label{T-pq-collection}
\Big\{  T^{p,q} \Big\}_{p,q \ge 0}
\end{equation}
forms a $\mT$-algebra and $L_w$ is a derivation of 
the $\mT$-algebra \eqref{T-pq-collection} for all 
$w \in \Der_{\bbK}(P)$.
 
Let $\om$ be the global section of $\Om^1(\cO_X^{\coord}) \otimes  \Der_{\bbK}(P)$
defined in \eqref{om-a-x}. Due to Theorem  \ref{thm:om}, the sum 
\begin{equation}
\label{diff-Ccoord-Tpq}
d + L_{\om}
\end{equation}
is a differential on the sheaf of graded vector spaces
\begin{equation}
\label{C-coord-T-pq}
\Omb(\cO_X^{\coord}) \otimes T^{p,q}
\end{equation}
for every pair $p,q \ge 0$\,. (The  $\bbZ$-grading on \eqref{C-coord-T-pq}
comes from the exterior degree on  $\Omb(\cO_X^{\coord})$\,.) 

Since $L_w$ is a derivation of 
the $\mT$-algebra \eqref{T-pq-collection} for every 
$w \in \Der_{\bbK}(P)$, the collection 
\begin{equation}
\label{C-coord-T-pqc-collec}
\left\{ \Omb(\cO_X^{\coord})  \otimes  T^{p,q} \right\}_{p,q\ge 0}
\end{equation}
together with the differential $d + L_{\om}$ assembles into 
a sheaf of dg algebras over $\mT$\,.

For every $\mv \in \mgl_d(\bbK)$ the contraction $i_{\bv}$ 
defines a degree $-1$ derivation of the sheaf of 
$\mT$-algebras \eqref{C-coord-T-pqc-collec}. Furthermore, 
Corollary \ref{cor:gl-d-om} implies that 
\begin{equation}
\label{L-v}
[(d + L_{\om}), i_{\bv}]  = l_{\bv} - L_ {\mv^a_b t^b \pa_{t^a}}\,, 
\end{equation}
where $l_{\bv}$ denotes the action \eqref{GL-d-acts} of  
$\mv\in \mgl_d(\bbK)$ on $\Omb(\cO^{\coord}_X)$\,. 

Due to Proposition \ref{prop:GL-d-acts} and Remark 
\ref{rem:actions-commute}, the assignment 
$$
\mv =  ||\mv^a_b|| ~\mapsto~ l_{\bv} - L_ {\mv^a_b t^b \pa_{t^a}}
$$
defines an action on $\mgl_d(\bbK)$ on the $\mT$-algebra    
\eqref{C-coord-T-pqc-collec}.  Moreover, due to equation \eqref{L-v}, 
this action is compatible with the differential \eqref{diff-Ccoord-Tpq}.

Let us construct a map of sheaves of $\mT$-algebras
\begin{equation}
\label{tau-goal}
\tau : \cT^{p,q}_X  \to \left( \cO_X^{\coord}  \otimes  T^{p,q} \right)^{\mgl_d(\bbK)}\,.
\end{equation}

For this purpose we consider an affine subset $U\subset X$ which 
admits a system of parameters 
\begin{equation}
\label{system-param-U}
x^1, x^2, \dots, x^d \in R = \cO_X(U)\,.
\end{equation}
Furthermore, we denote by 
\begin{equation}
\label{pa-pa-x^a}
\pa_{x^1}, \pa_{x^2}, \dots, \pa_{x^d}
\end{equation}
the basis of derivations of $R$ which is dual to the basis 
of K\"ahler differentials 
\begin{equation}
\label{dx-a}
d x^1, d x^2, \dots, d x^d\,.
\end{equation}

Every tensor field $v \in \cT^{p,q}_X (U)$ can be uniquely 
written in the form
\begin{equation}
\label{ga-wrt-basis}
v =  \sum_{1 \le a_t, b_s \le d} 
v^{a_1 \dots a_p}_{b_1 \dots b_q} \,  
\pa_{x^{a_1}} \otimes \dots \otimes \pa_{x^{a_p}} 
\otimes d x^{b_1} \otimes \dots \otimes  d x^{b_q}\,, 
\end{equation}
where $v^{a_1 \dots a_p}_{b_1 \dots b_q} \in R$\,.

Next we set\footnote{We assume the summation over repeated indices.}  
\begin{equation}
\label{tau-dfn}
\tau(v) : =   
I(v^{a_1 \dots a_p}_{b_1 \dots b_q}) 
\, (J_x^{-1})^{a'_1}_{a_1} \dots  (J_x^{-1})^{a'_p}_{a_p} \,
(J_x)^{b_1}_{b'_1} \dots (J_x)^{b_q}_{b'_q}
\pa_{t^{a'_1}} \otimes \dots \otimes \pa_{t^{a'_p}} 
\otimes d t^{b'_1} \otimes \dots \otimes  d t^{b'_q} \,,
\end{equation}
where $I$ is defined in (\ref{homo:I}) and $J_x$ is defined in (\ref{pa-x-pa-t}).

We claim that
\begin{claim}
\label{clm:tau}
The right hand side of (\ref{tau-dfn}) does not depend 
on the choice of the system of parameters (\ref{system-param-U}). 
For every local section $v$ of the sheaf $\cT^{p,q}_X$
the image $\tau(v)$ is $\mgl_d(\bbK)$-invariant and 
$(d + L_{\om})$-closed.  
\end{claim}
\begin{proof}

Let 
$$
\{y^1, y^2, \dots, y^d\}
$$
be another system of parameters on $U$ and 
let $\La$ be the invertible $d\times d$ matrix with entries in $R$
from equation \eqref{dy-dx-La}. 

Then components ${}^*v^{a_1 \dots a_p}_{b_1 \dots b_q}$ of $v$ in the 
new basis 
$$
\Big\{ 
\pa_{y^{a_1}}  \otimes  \dots \otimes \pa_{y^{a_p}}
\otimes d y^{b_1} \otimes \dots \otimes d y^{b_q} 
\Big\}_{1 \le a_t, b_s \le d}
$$
are related to the component $v^{a_1 \dots a_p}_{b_1 \dots b_q}$ via the 
formula: 
\begin{equation}
\label{ga-new-ga-old}
{}^*v^{a_1 \dots a_p}_{b_1 \dots b_q} = \La^{a_1}_{a'_1}
\dots   \La^{a_p}_{a'_p}  
(\La^{-1})^{b'_1}_{b_1} \dots (\La^{-1})^{b_q}_{b'_q} ~ v^{a'_1 \dots a'_p}_{b'_1 \dots b'_q}\,.
\end{equation}

On the other hand, since $I$ (\ref{homo:I}) is a $\bbK$-algebra 
homomorphism, we get 
$$
I ({}^*v^{a_1 \dots a_p}_{b_1 \dots b_q} ) = 
I(\La^{a_1}_{a'_1})
\dots   I(\La^{a_p}_{a'_p})  
I\big((\La^{-1}\big)^{b'_1}_{b_1}) \dots  
I\big((\La^{-1}\big)^{b'_q}_{b_q}) ~ I(v^{a'_1 \dots a'_p}_{b'_1 \dots b'_q})\,.
$$ 

So the proof of independence of the right hand side of 
(\ref{tau-dfn}) on the choice of the system of parameters 
boils down to checking the equation
\begin{equation}
\label{I-know}
\frac{\pa \wt{y^a}}{\pa t^c} = I(\La^a_b) \, \frac{\pa \wt{x^b}}{\pa t^c}\,.
\end{equation}

To prove equation (\ref{I-know}) we notice that for 
every $1 \le c \le d$ the map 
\begin{equation}
\label{pa-t-circ-I}
\frac{\pa}{\pa t^c} \circ I  : R \to R^{\coord}[[t^1, t^2, \dots, t^d]] 
\end{equation}
is a $\bbK$-linear derivation of $R$-modules, where the 
$R$-module structure on the target of (\ref{pa-t-circ-I}) is 
define by the formula: 
\begin{equation}
\label{R-mod-Rcoord-ts}
a \cdot v = I(a)\, v\,, \qquad a\in R, \quad v \in R^{\coord}[[t^1, t^2, \dots, t^d]]\,. 
\end{equation}
Hence equation (\ref{dy-dx-La}) and the universality of K\"ahler differentials 
implies equation (\ref{I-know}). Thus the right hand side of \eqref{tau-dfn} 
is indeed independent on the choice of system of parameters. 

To prove that $\tau(v)$ is  $(d+ L_{\om})$-closed for every 
section $v$ of $\cT^{p,q}_X$, we give the following 
obvious identities of the matrix $J_x$
\begin{equation}
\label{pa-J}
\pa_{t^b}  \sum_{\ui} d x^a_{\ui} t^{\ui}  = d \, (J_x)^a_b\,,
\end{equation}
\begin{equation}
\label{pa-J1}
\pa_{t^c}\, (J^{-1}_x)^a_b = - (J^{-1}_x)^a_{a_1} \,  \big(  \pa_{t^c}(J_x)^{a_1}_{b_1} \big) \, (J^{-1}_x)^{b_1}_{b}\,,
\end{equation}
and
\begin{equation}
\label{pa-J11}
\pa_{t^b} (J_x)^a_c = \pa_{t^c} (J_x)^a_b\,.
\end{equation}

Using identities \eqref{pa-J}, \eqref{pa-J1} and \eqref{pa-J11}, it is easy to see that 
\begin{equation}
\label{Jpa-t-is-flat}
(d+ L_{\om}) \, \big( (J_x^{-1})^{a'}_{a}  \pa_{t^{a'}}  \big)  = 0
\end{equation}
and
\begin{equation}
\label{J-dt-is-flat}
(d+ L_{\om}) \, \big( (J_x)^{b}_{b'} d t^{b'}   \big)  = 0\,.
\end{equation}

On the other hand 
$$
(d+ L_{\om}) \big( I(v^{a_1 \dots a_p}_{b_1 \dots b_q}) \big) =
d \big( I(v^{a_1 \dots a_p}_{b_1 \dots b_q}) \big) + 
\om \big( I(v^{a_1 \dots a_p}_{b_1 \dots b_q}) \big)
= 0
$$
due to \eqref{tilde-f-is-flat}.

Thus $\tau (v)$ is indeed $(d+ L_{\om})$-closed for every tensor field $v$.

Finally, the $\mgl_d(\bbK)$-invariance of $\tau(v)$ is obvious from the defining equation. 

\end{proof}

Let us now consider $\mgl_d(\bbK)$ as the set of degree $-1$ derivations 
$i_{\bv}$, $\mv \in \mgl_d(\bbK)$ of the sheaf of dg $\mT$-algebras
\eqref{C-coord-T-pqc-collec} and apply trimming to the sheaf \eqref{C-coord-T-pqc-collec}
following Section \ref{sec:trimming}.

Namely, according to \eqref{cV-trimmed}, local sections of 
the sheaf
\begin{equation}
\label{C-coord-T-pq-invar}
\Big( \Omb(\cO_X^{\coord}) \otimes  T^{p,q} \Big)^{[\mgl_d(\bbK)]}
\end{equation}
are local sections $w$ of $\Omb(\cO_X^{\coord})   \otimes  T^{p,q} $
satisfying the conditions
\begin{equation}
\label{i-bv-w}
i_{\bv} (w) = 0\,,
\end{equation}
and
\begin{equation}
\label{Lie-bv-w}
\big((d + L_{\om}) \circ i_{\bv} + 
i_{\bv} \circ (d + L_{\om}) \big) (w) = 0\,.
\end{equation}

For example, if $w$ is a local section of  
$$
\Om^0(\cO_X^{\coord})  \otimes  T^{p,q}  = 
\cO_X^{\coord}  \otimes T^{p,q} 
$$
then equation \eqref{i-bv-w} holds automatically and hence, equation 
\eqref{L-v} implies that 
\begin{equation}
\label{invar-trimmed}
\Big( \Om^0(\cO_X^{\coord})  \otimes T^{p,q} \Big)^{[\mgl_d(\bbK)]} 
= \Big( \Om^0(\cO_X^{\coord}) \otimes T^{p,q} \Big)^{\mgl_d(\bbK)}\,.
\end{equation}

\begin{warning}
\label{warn:invar-and-more}
We would like to recall that the notation $\cV^{[\mgl_d(\bbK)]}$ is 
reserved for the dg subalgebra of $\mgl_d(\bbK)$-basic elements in $\cV$
(see eq. \eqref{cV-trimmed}).
On the other hand, we still use $\cV^{\mgl_d(\bbK)}$ to denote 
the subalgebra of $\mgl_d(\bbK)$-invariants.
\end{warning}

According to Section \ref{sec:trimming}, the $\mT$-algebra structure 
descends to the collection
\begin{equation}
\label{C-coord-T-pq-invar-collec}
\Big\{ \Big( \Omb(\cO_X^{\coord})  \otimes  T^{p,q} \Big)^{[\mgl_d(\bbK)]}
\Big\}_{p,q\ge 0}\,.
\end{equation}
In other words, the collection \eqref{C-coord-T-pq-invar-collec} is 
a sheaf of dg algebras over the operad $\mT$\,.

Due to Claim \ref{clm:tau} and equation \eqref{invar-trimmed}, 
formula \eqref{tau-dfn} defines a map of sheaves   
\begin{equation}
\label{tau-is}
\tau : \cT^{p,q}_X ~ \mapsto ~  \Big( \Omb(\cO_X^{\coord}) \otimes T^{p,q} \Big)^{[\mgl_d(\bbK)]}
\end{equation}
of dg algebras over $\mT$, where the sheaf $ \cT^{p,q}_X$ carries 
the zero differential.

Our goal is to show that \eqref{tau-is} is a quasi-isomorphism of sheaves of 
of $\mT$-algebras.  
\begin{thm}[Fedosov resolution of $\cT^{p,q}_X$]
\label{thm:Fedosov}
For every smooth algebraic variety $X$ over $\bbK$ 
the map \eqref{tau-is} defined by equation \eqref{tau-dfn} is a 
quasi-isomorphism from the tensor algebra \eqref{tensor-alg-X}
on $X$ to the dg $\mT$-algebra \eqref{C-coord-T-pq-invar-collec}. 
\end{thm}
\begin{proof} It is clear that the map $\tau$ intertwines all operations 
of the tensor algebra. Furthermore, by Claim \ref{clm:tau}, 
$\tau$ is compatible with the differential on
$$ 
\Big( \Omb(\cO_X^{\coord})  \otimes T^{p,q} \Big)^{[\mgl_d(\bbK)]}\,.
$$

Thus, it remains to prove that the complex of sheaves
\begin{equation}
\label{total-comp}
 \begin{tikzpicture}
\matrix (m) [matrix of math nodes, column sep=0.8em]
{  \cT^{p,q}_X &~ & \Big( \Om^0(\cO_X^{\coord}) \otimes T^{p,q} \Big)^{[\mgl_d(\bbK)]} 
&~ & \Big( \Om^1(\cO_X^{\coord}) \otimes T^{p,q} \Big)^{[\mgl_d(\bbK)]}  & ~
& ~\dots~ \\ };
\path[->,font=\scriptsize]
(m-1-1) edge node[auto] {$\tau$} (m-1-3)  (m-1-3) edge
 node[auto] {$d+ L_{\om}$} (m-1-5)  (m-1-5) edge node[auto] {$d+ L_{\om}$}  (m-1-7);
\end{tikzpicture} 
\end{equation}
is acyclic. 

Since acyclicity of a complex of sheaves is a local property, 
it is enough to prove that the complex 
\begin{equation}
\label{total-comp-U}
 \begin{tikzpicture}
\matrix (m) [matrix of math nodes, column sep=0.8em]
{  \cT^{p,q}_X(U) &~ & \Big( \Om^0(\cO_X^{\coord})(U) \otimes T^{p,q} \Big)^{[\mgl_d(\bbK)]} 
&~ & \Big( \Om^1(\cO_X^{\coord})(U)  \otimes T^{p,q} \Big)^{[\mgl_d(\bbK)]}  & ~
& ~\dots~ \\ };
\path[->,font=\scriptsize]
(m-1-1) edge node[auto] {$\tau$} (m-1-3)  (m-1-3) edge
 node[auto] {$d+ L_{\om}$} (m-1-5)  (m-1-5) edge node[auto] {$d+ L_{\om}$}  (m-1-7);
\end{tikzpicture} 
\end{equation}
is acyclic for ``small enough'' open neighborhood $U$ of 
every point of $X$\,.  

We will prove this fact for an arbitrary affine open subset 
$U \subset X$ which is equipped  a global 
system of parameters 
\begin{equation}
\label{system-param-here}
x^1, x^2, \dots, x^d \in \cO_X(U)\,.
\end{equation}

We set $R = \cO_X(U)$ and observe that,
for such an affine subset $U$, the complex in question is 
\begin{equation}
\label{total-comp-R}
 \begin{tikzpicture}
\matrix (m) [matrix of math nodes, column sep=0.8em]
{  T^{p,q}(R) &~ & \Big( \Om^0(R^{\coord})  \otimes  T^{p,q} \Big)^{[\mgl_d(\bbK)]} 
&~ & \Big( \Om^1(R^{\coord})  \otimes  T^{p,q} \Big)^{[\mgl_d(\bbK)]}  & ~
& ~\dots\,,~ \\ };
\path[->,font=\scriptsize]
(m-1-1) edge node[auto] {$\tau$} (m-1-3)  (m-1-3) edge
 node[auto] {$d+ L_{\om}$} (m-1-5)  (m-1-5) edge node[auto] {$d+ L_{\om}$}  (m-1-7);
\end{tikzpicture} 
\end{equation}
where 
\begin{equation}
\label{T-pq-R}
T^{p,q} (R) : = \underbrace{\Der(R)\otimes_R \dots \otimes_R \Der(R)}_{p \textrm{ times}}
\otimes_R
 \underbrace{\Om^1(R) \otimes_R \dots \otimes_R  \Om^1(R)}_{q \textrm{ times}}\,.  
\end{equation}

We remark that the map $\tau$ (\ref{tau-dfn})
gives us the following isomorphism of $\bbK$-vector spaces
\begin{equation}
\label{simpler}
 \Om^k(R^{\coord}) \otimes T^{p,q} \cong
\Om^k(R^{\coord})[[t^1, t^2, \dots, t^d]]  \otimes_{I(R)}  \tau (T^{p,q}(R))\,.
\end{equation}

Due to Proposition \ref{prop:affine-space}
$$
R^{\coord} \cong R^{\aff} \otimes \cO(\GL_d(\bbK)) 
$$
on which $\GL_d(\bbK)$ acts by right translations. 

Hence, 
\begin{equation}
\label{i-v-zero}
\Big\{ \eta \in  
\Om^k(R^{\coord})[[t^1, t^2, \dots, t^d]] ~~ \big|~~ i_{\bv} \eta = 0 
~~ \forall~ \mv \in \mgl_d(\bbK) \Big\} \cong 
\end{equation}
$$
\Hom_{R^{\aff}} \Big(  \wedge^{k}_{R^{\aff}} \Der(R^{\aff}),
R^{\coord}[[t^1, t^2, \dots, t^d]]  
\Big) \cong  
$$
$$
\cong R^{\coord}[[t^1, t^2, \dots, t^d]]  \otimes_{R^{\aff}}
\Om^k(R^{\aff}) \,.
$$

Thus, using \eqref{L-v}, we get
$$
\Big( \Om^k(R^{\coord}) \otimes T^{p,q}   \Big)^{[\mgl_d(\bbK)]} \cong
$$
\begin{equation}
\label{simpler1}
\big(R^{\coord}[[t^1, t^2, \dots, t^d]] \big)^{\mgl_d(\bbK)} \otimes_{R^{\aff}}
 \Om^k(R^{\aff})
 \otimes_{I(R)}  \tau (T^{p,q}(R))\,.
\end{equation}

Therefore, since $\tau(T^{p,q}(R))$ is a free $I(R)$-module, 
it suffices to prove the acyclicity of the complex 
\begin{equation}
\label{R-Xi}
\xymatrix@M=1.5pc{ 
R  \ar[r]^{I} & \Xi^{0} 
\ar[r]^{(d + \om) }
& \Xi^1  \ar[r]^{(d + \om)} &  \Xi^2  \ar[r]^{(d + \om)} &  \dots 
}
\end{equation}
with 
\begin{equation}
\label{Xi-k}
\Xi^k =  \big(R^{\coord}[[t^1, t^2, \dots, t^d]] \big)^{\mgl_d(\bbK)} \otimes_{R^{\aff}}
 \Om^k(R^{\aff})
\end{equation}

For this purpose we consider the continuous homomorphism 
of commutative $R^{\coord}$-algebras  
\begin{equation}
\label{R-coord-theta-t}
\psi : R^{\coord}[[\te^1, \te^2, \dots, \te^d]] \to 
 R^{\coord}[[t^1, t^2, \dots, t^d]]
\end{equation}
defined by the equation
\begin{equation}
\label{psi-dfn}
\psi (\te^a) = I(x^a)- x^a\,.
\end{equation}

Recall that  
$$
I(x^a)- x^a  = x^a_{(b)} t^b + \sum_{\ui, ~|\ui| \ge 2} x^a_{\ui} t^{\ui}
$$
and the matrix $|| x^a_{(b)}||$ is invertible in $\Mat_{d} (R^{\coord})$\,.
Hence, the homomorphism $\psi$ is an isomorphism. 

Furthermore, since $I(x^a)- x^a$ is $\mgl_d(\bbK)$-invariant for every $a$, 
the map $\psi$ induces an isomorphism between the 
$\bbK$-algebras: 
\begin{equation}
\label{isom}
\Big( R^{\coord} [[t^1, t^2, \dots, t^d]]  \Big)^{\mgl_d(\bbK)} 
\cong 
R^{\aff} [[\te^1, \te^2, \dots, \te^d]]\,.
\end{equation}

On the other hand, by Theorem \ref{thm:om}, 
$$
(d + \om) I(x^a)  = 0
$$
and hence
$$
(d + \om)  \psi(\te^a) = - d x^a\,.
$$
Thus, the cochain complex \eqref{R-Xi} is isomorphic to 
\begin{equation}
\label{R-Xi-simpler}
 \begin{tikzpicture}
\matrix (m) [matrix of math nodes, column sep=0.8em]
{R & ~ &   \Om^0(R^{\aff}) [[\te^1, \te^2, \dots, \te^d]] &~&
 \Om^1 (R^{\aff}) [[\te^1, \te^2, \dots, \te^d]] & ~ &~ \\};
\path[->,font=\scriptsize]
(m-1-1) edge node[auto] {$\psi^{-1} \circ I$} (m-1-3)  (m-1-3) edge
 node[auto] {$D$} (m-1-5)  (m-1-5) edge node[auto] {$D$}  (m-1-7);
\end{tikzpicture} 
\end{equation}
where 
\begin{equation}
\label{D-dfn}
D = d - \sum_{a=1}^d  d x^a \frac{\pa}{\pa \te^a}\,. 
\end{equation}

To deduce the acyclicity of the cochain complex \eqref{R-Xi-simpler},
we need the following technical claim which is proven in 
Subsection \ref{sec:psi-I} below. 
\begin{claim}
\label{cl:psi-I}
Let $U= \Spec(R)$ be a smooth affine variety over $\bbK$
of dimension $d$ with a global system of parameters
$$
x^1, x^2, \dots, x^d \in R\,.
$$
Let $\{ \pa_{x^1}, \pa_{x^2}, \dots, \pa_{x^d} \}$ be the basis 
of $\Der(R)$ dual to $\{d x^1, d x^2, \dots, d x^d\}$\,.
If 
$$
\psi : R^{\coord}[[\te^1, \te^2, \dots, \te^d]] \to  R^{\coord}[[t^1, t^2, \dots, t^d]]
$$
is the isomorphism defined by \eqref{psi-dfn}, then for every 
$f \in R$ 
\begin{equation}
\label{psi-I}
\psi^{-1} \circ I (f) = f + \sum_{k \ge 1} \frac{1}{k!} 
   \pa_{x^{a_1}} \pa_{x^{a_2}} \dots \pa_{x^{a_k}} (f) \te^{a_1} \te^{a_2} \dots \te^{a_k}\,,
\end{equation}
where indices $a_1, a_2, \dots, a_k$ run from $1$ to $d$\,.
\end{claim}

Due to Proposition \ref{prop:affine-space}
$$
R^{\aff} \cong R [\{y^a_{\ui}\}_{1\le a \le d,\, |\ui| \ge 2}]\,.  
$$
Hence,
\begin{equation}
\label{Xi-k-simpler1}
\Xi^k  \cong  \bigoplus_{p=0}^{k}  \Om^p(\bbK[\{y^a_{\ui}\}_{1\le a \le d,\, |\ui| \ge 2}] )
\otimes
\Om^{k-p}(R) 
[[\te^1, \te^2, \dots, \te^d]]
\end{equation}
and the complex  \eqref{R-Xi-simpler} is isomorphic to the tensor product of 
the de Rham complex 
\begin{equation}
\label{de-Rham-ys}
\left( \Om^{\bul}(\bbK[\{y^a_{\ui}\}_{1\le a \le d,\, |\ui| \ge 2}] ,\, d \right)
\end{equation}
and the cochain complex 
\begin{equation}
\label{Om-R-te-s}
 \begin{tikzpicture}
\matrix (m) [matrix of math nodes, column sep=0.8em]
{R & ~ &   \Om^0(R) [[\te^1, \dots, \te^d]] &~&
 \Om^1 (R) [[\te^1, \dots, \te^d]] & ~ &~ \\};
\path[->,font=\scriptsize]
(m-1-1) edge node[auto] {$\psi'$} (m-1-3)  (m-1-3) edge
 node[auto] {$D'$} (m-1-5)  (m-1-5) edge node[auto] {$D'$}  (m-1-7);
\end{tikzpicture} 
\end{equation}
where 
\begin{equation}
\label{psi-pr}
\psi' (f) = f + \sum_{k \ge 1} \frac{1}{k!} 
   \pa_{x^{a_1}} \pa_{x^{a_2}} \dots \pa_{x^{a_k}} (f) \te^{a_1} \te^{a_2} \dots \te^{a_k}\,,
\end{equation}
and 
\begin{equation}
\label{D-pr}
D' = d - \sum_{a=1}^d \, d x^a \frac{\pa}{\pa \te^a}\,. 
\end{equation}

Let us prove that the cochain complex \eqref{Om-R-te-s} is acyclic.
 
For this purpose we denote \eqref{Om-R-te-s} by $\cK$
and observe that it carries the descending filtration
\begin{equation}
\label{filtr-cK}
\cK = F_0 \cK \supset F_1 \cK \supset F_2 \cK  \supset \dots
\end{equation}
where $F_m \cK$ consists of series 
$$
\sum_{p+q \ge m} f_{a_1, \dots a_p; b_1 \dots b_q}
d x^{a_1} \dots d x^{a_p}  \te^{b_1} \dots  \te^{b_q}\,,
\qquad  f_{a_1, \dots a_p; b_1 \dots b_q} \in R\,. 
$$

The associated graded complex $\Gr(\cK)$ is isomorphic to 
\begin{equation}
\label{Gr-cK}
 \begin{tikzpicture}
\matrix (m) [matrix of math nodes, column sep=0.8em]
{R & ~ &   \Om^0(R) [\te^1, \dots, \te^d] &~&
 \Om^1 (R) [\te^1, \dots, \te^d] & ~ &~ \\};
\path[->,font=\scriptsize]
(m-1-1) edge node[auto] {$\Gr(\psi')$} (m-1-3)  (m-1-3) edge
 node[auto] {$\Gr(D')$} (m-1-5)  (m-1-5) edge node[auto] {$\Gr(D')$}  (m-1-7);
\end{tikzpicture} 
\end{equation}
where 
$$
\Gr(\psi') (f) = f\,,
$$
and 
$$
\Gr(D') = - \sum_{a=1}^d \, d x^a \frac{\pa}{\pa \te^a}\,.
$$
In other words, $\Gr(\cK$) is isomorphic to the Koszul complex for 
the polynomial algebra $R [\te^1, \dots, \te^d]$ over $R$.
Therefore $\Gr(\cK)$ is acyclic. 

Combining this observation with the fact that $\cK$ is complete 
with respect to filtration \eqref{filtr-cK}, we conclude that $\cK$ is 
also acyclic.

Theorem \ref{thm:Fedosov} is proven. 
\end{proof}

\subsection{Proof of Claim \ref{cl:psi-I}}
\label{sec:psi-I}
The proof of Theorem \ref{thm:Fedosov} depends on Claim  \ref{cl:psi-I}. 

To prove this claim we observe that, since $\psi$ intertwines the differentials 
$d + \om$ and $D$ \eqref{D-dfn}, we conclude that 
\begin{equation}
\label{D-closed}
D (\psi^{-1} \circ I(f)) = 0\,, \qquad \forall ~ f \in R\,.
\end{equation}

On the other hand, 
$$
D \left(
 f + \sum_{k \ge 1} \frac{1}{k!} 
   \pa_{x^{a_1}} \pa_{x^{a_2}} \dots \pa_{x^{a_k}} (f) \te^{a_1} \te^{a_2} \dots \te^{a_k}
\right) = 0 
$$ 
by construction.

Thus we need to prove that, if a sum 
\begin{equation}
\label{sum}
\sum_{k \ge 1} \frac{1}{k!} 
C_{a_1 a_2 \dots a_k} \te^{a_1} \te^{a_2} \dots \te^{a_k}\,,
\qquad C_{a_1 a_2 \dots a_k}  \in R^{\aff} 
\end{equation}
satisfies the equation 
\begin{equation}
\label{sum-D-closed}
D \left(
\sum_{k \ge 1} \frac{1}{k!} 
C_{a_1 a_2 \dots a_k} \te^{a_1} \te^{a_2} \dots \te^{a_k}
\right) = 0
\end{equation}
then the sum \eqref{sum} is zero. 

If  \eqref{sum} is non-zero then there exists a positive integer $r$ such that 
at least one coefficient $C_{a_1 a_2 \dots a_r}$ is non-zero and 
all coefficients  $C_{a_1 a_2 \dots a_k} = 0$ is $k < r$\,.  
Then, 
$$
D \left(
\sum_{k \ge 1} \frac{1}{k!} 
C_{a_1 a_2 \dots a_k} \te^{a_1} \te^{a_2} \dots \te^{a_k}
\right) = \frac{1}{(r-1)!} C_{a_1 a_2 \dots a_r} dx^{a_1} 
 \te^{a_2} \dots \te^{a_k} + \dots
$$
where $\dots$ is a sum of terms of degree in $\te$ greater than $r-1$\,.

Hence, if \eqref{sum} is non-zero, then  
$$
D \left(
\sum_{k \ge 1} \frac{1}{k!} 
C_{a_1 a_2 \dots a_k} \te^{a_1} \te^{a_2} \dots \te^{a_k}
\right) \neq 0\,.
$$

Thus $\psi^{-1} \circ I(f)$ must equal 
$$
f + \sum_{k \ge 1} \frac{1}{k!} 
   \pa_{x^{a_1}} \pa_{x^{a_2}} \dots \pa_{x^{a_k}} (f) \te^{a_1} \te^{a_2} \dots \te^{a_k}
$$
for every $f \in R$\,.

Claim  \ref{cl:psi-I} is proven. ~~~$\Box$

\subsection{Fedosov resolution of the Gerstenhaber algebra of polyvector fields}
\label{sec:Fed-polyvect}

Let us recall that antisymmetric contravariant tensor fields on $X$ are 
called {\it polyvector fields}. In other words, polyvector fields are sections
of the sheaf 
\begin{equation}
\label{T-poly}
\cT_{\poly} : = S_{\cO_{X}} (\bs \cT_X)\,. 
\end{equation}

Let us also recall that the Schouten-Nijenhuis bracket $\{~,~\}_{SN}$
and the obvious commutative multiplication equips $\cT_\poly$ with the structure 
of a sheaf of Gerstenhaber algebras. We denote by $\Ger$ the Gerstenhaber operad. 

Let us denote by $T_{\poly}(P)$ the Gerstenhaber algebra 
of poly-derivations of $P = \bbK[[t^1, t^2, \dots, t^d]]$, i.e.
\begin{equation}
\label{Tpoly-formal}
T_{\poly}(P) : = P ~ \oplus ~  \bigoplus_{m \ge 1}  \Big(
 \underbrace{\bs \Der(P)  \otimes_{P} \dots   \otimes_{P}
\bs \Der(P)}_{m \textrm{ times}} \Big)_{S_m}
\end{equation}

Next consider the subsheaf 
\begin{equation}
\label{OmcO-coord-T-poly}
\Omb(\cO_X^{\coord})  \otimes T_{\poly}(P)
\subset
\bigoplus_{p \ge 0} \bs^{p} \Omb(\cO_X^{\coord}) \otimes T^{p,0}.
\end{equation}

It is easy to see that the subsheaf \eqref{OmcO-coord-T-poly}
is closed with respect to the differential \eqref{diff-Ccoord-Tpq}. 
Furthermore, the restriction of the differential \eqref{diff-Ccoord-Tpq}
to  \eqref{OmcO-coord-T-poly} coincides with the differential 
\begin{equation}
\label{diff-resol-Tpoly}
d + \{ \om, ~ \}_{SN} \,.
\end{equation}
The sheaf \eqref{OmcO-coord-T-poly} with the differential \eqref{diff-resol-Tpoly}
is naturally a sheaf of dg Gerstenhaber algebras.

Since for every $\mv \in  \mgl_d(\bbK)$, the operation $i_{\bv}$ is 
a derivation of the  Gerstenhaber algebra structure on the sheaf  \eqref{OmcO-coord-T-poly}, 
we may form\footnote{See Section \ref{sec:trimming}.} the following sheaf of dg 
Gerstenhaber algebras
\begin{equation}
\label{Fed-T-poly}
\Big( \Omb(\cO_X^{\coord}) \otimes T_{\poly}(P) \Big)^{[\mgl_d(\bbK)]}\,.
\end{equation}

\begin{thm}[Fedosov resolution for polyvector fields]
\label{thm:Fed-Tpoly}
Let $X$ be a smooth algebraic variety of dimension $d$ over 
an algebraically closed field $\bbK$ of characteristic zero. 
Let us consider the sheaf $\cT_{\poly}$  \eqref{T-poly}  as a sheaf 
of dg  Gerstenhaber algebras with the zero differential. Then
the restriction of the map $\tau$ \eqref{tau-dfn} to $\cT_{\poly}$
\begin{equation}
\label{tau-polyvect}
\tau \Big|_{\cT_{\poly}} ~:~ \cT_{\poly} ~\to~ 
\Big( \Omb(\cO_X^{\coord}) \otimes T_{\poly}(P) \Big)^{[\mgl_d(\bbK)]}
\end{equation}
is a quasi-isomorphism of the sheaves of dg Gerstenhaber algebras.
\end{thm}
\begin{proof}
Due to Theorem \ref{thm:Fedosov}, the map 
\begin{equation}
\label{tau-contra-m}
\tau :  \cT^{m,0}_X  \to
\Big( \Omb(\cO_X^{\coord}) \otimes T^{m,0}(P) \Big)^{[\mgl_d(\bbK)]}
\end{equation}
is a quasi-isomorphism of sheaves for every $m \ge 0$. Hence so is the map 
\begin{equation}
\label{sus-tau-contra-m}
\tau :  \bs^m \, \cT^{m,0}_X  \to
\Big( \Omb(\cO_X^{\coord})  \otimes \bs^m \, T^{m,0}(P) \Big)^{[\mgl_d(\bbK)]}
\end{equation}

Let us denote by $\sgn_m$ the sign representation of $S_m$ and 
observe that
$$
\cT^m_{\poly} = \Big( \sgn_m\, \bs^m \, \cT^{m,0}_X \Big)_{S_m}
$$
and 
$$
\Big( \Omb(\cO_X^{\coord})  \otimes T^m_{\poly}(P) \Big)^{[\mgl_d(\bbK)]} ~ = ~
\left( \sgn_m \Big( \Omb(\cO_X^{\coord}) \otimes \bs^m \, T^{m,0}(P) \Big)^{[\mgl_d(\bbK)]} 
\right)_{S_m}\,.
$$

Thus the map 
$$
\tau : \cT^m_{\poly} \to 
\Big( \Omb(\cO_X^{\coord}) \otimes T^m_{\poly}(P) \Big)^{[\mgl_d(\bbK)]}
$$
is a quasi-isomorphism for every $m$, since, in characteristic zero, 
the cohomology commutes with taking invariants.

It remains to prove that the map \eqref{tau-polyvect} is compatible with
the Gerstenhaber algebra structures. 

To prove this property, we consider an affine open subset $U \subset X$, 
set $R = \cO_X(U)$ and assume that $U$ has a global system of 
parameters \eqref{system-param-U}. 

It is obvious from the definition of $\tau$ \eqref{tau-dfn} that 
for every pair of sections $v_1, v_2 \in \cT_{\poly}(U)$
$$
\tau( v_1 \cdot v_2) =  \tau(v_1) \cdot  \tau(v_2)\,.
$$

To prove the compatibility of $\tau$ with the bracket $\{~,~\}_{SN}$, we 
observe that the  Gerstenhaber algebra $\cT_{\poly}(U)$ is generated by 
$f \in R$ and the derivations \eqref{pa-pa-x^a}. Thus, it suffices to prove that 
\begin{equation}
\label{f1f2}
\{ \tau(f_1),  \tau(f_2) \}_{SN} =0\,,
\end{equation}
\begin{equation}
\label{pa-xa-f}
\{ \tau(\pa_{x^a}),  \tau(f) \}_{SN} = \tau (\pa_{x^a}(f) ) \,,
\end{equation}
and 
\begin{equation}
\label{pa-xa-xb}
\{ \tau(\pa_{x^a}),  \tau(\pa_{x^b}) \}_{SN} = 0 
\end{equation}
for all $f, f_1, f_2 \in R$, and $1 \le a, b \le d$\,.

Since 
$$
\tau(f) = \wt{f}= f + \sum_{|\ui| \ge 1} f_{\ui} t^{\ui} 
$$
equation \eqref{f1f2} holds obviously. 

To prove equation \eqref{pa-xa-f}, we observe that the 
operation 
$$
f \mapsto \pa_{t^b} \wt{f} : R \to R^{\coord}[[t^1, \dots, t^d]] 
$$
is a $\bbK$-linear derivation of $R$-modules with $R^{\coord}[[t^1, \dots, t^d]] $
carrying the $R$-module structure defined in \eqref{R-mod-Rcoord-ts}.

Hence, 
\begin{equation}
\label{pa-t-wt-f}
\pa_{t^b} \wt{f} = I(\pa_{x^a}(f)) \frac{\pa \wt{x^a}}{\pa t^b} \,.
\end{equation}

Using equation \eqref{pa-t-wt-f} we deduce 
$$
\{ \tau(\pa_{x^a}),  \tau(f) \}_{SN} = 
(J^{-1}_x)_a^{b} \pa_{t^{b}} \wt{f} = (J^{-1}_x)_a^{b} 
 I(\pa_{x^c}(f)) \frac{\pa \wt{x^c}}{\pa t^b} =  I(\pa_{x^a}(f))\,. 
$$
Thus equation \eqref{pa-xa-f} holds.

Using equations \eqref{pa-J1} and \eqref{pa-J11}, it is easy to see that 
$$
\{ \tau(\pa_{x^a}),  \tau(\pa_{x^b}) \}_{SN} = 
 \Big( (J^{-1}_x)_a^{c}\, \pa_{t^{c}} (J^{-1}_x)_b^{c'}  -
 (J^{-1}_x)_b^{c}\, \pa_{t^{c}} (J^{-1}_x)_a^{c'} 
\Big) \pa_{t^{c'}} = 0\,. 
$$
Thus equation \eqref{pa-xa-xb} also holds.

Theorem \ref{thm:Fed-Tpoly} is proven.  
\end{proof}

\section{Atiyah class via Fedosov resolution}
\label{sec:Atiyah}

Let us cover our variety $X$ by affine open subsets $\{ U_{\al} \}_{\al \in \cI}$
each of which has a global system of parameters: 
\begin{equation}
\label{system-param-alpha}
x^1_{\al}, x^2_{\al}, \dots, x^d_{\al} \in \cO_X(U_{\al})\,.
\end{equation}

For each $\al \in \cI$, the module $\Om^1(\cO_X(U_{\al}))$ of 
K\"ahler differentials is freely generated by the forms
\begin{equation}
\label{basis-U-alpha}
d \, x^1_{\al},~ d \, x^2_{\al}, ~\dots,~ d \, x^d_{\al}\,. 
\end{equation}

Therefore, for every non-empty intersection\footnote{Let us recall that a variety is necessarily 
a separated scheme. Hence intersection of two affine subsets is again an affine 
subset.}  $U_{\al \beta} = U_{\al} \cap U_{\beta}$ there exists a unique non-degenerate matrix
$$
|| (\La_{\al \beta})^{a}_b || \in \Mat_d\big( \cO_X (U_{\al} \cap U_{\beta}) \big)
$$
such that 
\begin{equation}
\label{dx-alpha-beta}
d x^a_{\al} =  (\La_{\al \beta})^{a}_b \,  d x^b_{\beta}\,.
\end{equation}
In particular, $\La_{\beta \al} = (\La_{\al \beta})^{-1}$\,.

It is easy to see that the collection of tensor fields 
\begin{equation}
\label{Atiyah-cocycle}
(\La_{\beta \al})^{b_1}_a  \, d (\La_{\al \beta})^a_{b_2} \, 
 \pa_{x^{b_1}_{\beta}} \otimes d x^{b_2}_{\beta}  \in 
 \G(U_{\al \beta}, \cT^{1,2}_X)
\end{equation}
is a $1$-cocycle in the \v{C}ech complex 
\begin{equation}
\label{Cech-cT1-2}
\cCb(X, \cT^{1,2}_X)
\end{equation}
for the sheaf  $\cT^{1,2}_X$\,. Furthermore, the cohomology 
class of the cocycle \eqref{Atiyah-cocycle} does not depend on 
the choice of the systems of parameters on affine subsets $U_{\al}$\,.

According to \cite{Atiyah}, the cocycle \eqref{Atiyah-cocycle}
is trivial if and only if the tangent sheaf $\cT_X$ admits an algebraic 
connection. We refer to the cohomology class of  \eqref{Atiyah-cocycle}
as the {\it Atiyah class} of the tangent sheaf $\cT_X$ and denote
the cocycle \eqref{Atiyah-cocycle} by $\sfA$.

We claim that 
\begin{thm}
\label{thm:At-Fed}
Let $\om^a \pa_{t^a}$ be the global section of the sheaf
$$
 \Om^1(\cO_X^{\coord}) \otimes T^{1,0}(P)
$$
introduced in Theorem \ref{thm:om}. Then
\begin{equation}
\label{At-rep-ve}
\sfA_{\om} : = - \frac{\pa^2 \om^a}{\pa t^{b_1} \pa t^{b_2}} \,
\pa_{t^a} \otimes d t^{b_1}  \otimes d t^{b_2} 
\end{equation}
is a global $(d+ L_{\om})$-closed section of the sheaf 
\begin{equation}
\label{Om1T1-2-INV} 
\big( \Om^1(\cO_X^{\coord})  \otimes T^{1,2}(P) \big)^{[\mgl_d(\bbK)]}\,.
\end{equation}
Furthermore, $\sfA_{\om}$ is cohomologous to the cocycle 
$\tau(\sfA)$ in the \v{C}ech complex 
\begin{equation}
\label{Cech-Fed-resolv}
\cCb \Big( X,  \big( \Omb(\cO_X^{\coord})  \otimes T^{1,2}(P) \big)^{[\mgl_d(\bbK)]}   \Big)\,.
\end{equation}
\end{thm}
\begin{proof}
To prove the first statement we need to show that 
$\sfA_{\om}$ satisfies these conditions 
\begin{equation}
\label{i-bv-L-bv-At-rep}
i_{\bv} (\sfA_{\om}) = 0\,,  \qquad 
[(d+L_{\om}), i_{\bv}]   (\sfA_{\om}) = 0\,, 
 \qquad \forall~~ \mv \in \mgl_d(\bbK)\,,
\end{equation}
and
\begin{equation}
\label{cAom-d-om-flat}
(d+ L_{\om}) \sfA_{\om} = 0\,.
\end{equation}

Applying $\pa_{t^{b_1}} \pa_{t^{b_2}}$ to 
equation \eqref{om-is-flat1} we get 
\begin{multline*}
0 = \pa_{t^{b_1}} \pa_{t^{b_2}} 
(d \om^a + \om^{a'} \pa_{t^{a'}} (\om^a)) = 
d \, \frac{\pa^2 \om^a} {\pa t^{b_1} \pa t^{b_2}} +
\om^{a'}  \frac{\pa}{\pa t^{a'}} 
\left( \frac{\pa^2 \om^{a}}{\pa t^{b_1} \pa t^{b_2}} \right)
+\\
+ \frac{\pa^2 \om^{a'}}{\pa t^{b_1} \pa t^{b_2}} 
\frac{\pa \om^a}{ \pa t^{a'}}  + 
\frac{\pa  \om^{a'}}{\pa t^{b_1}}
\frac{ \pa^2 \om^a}{\pa t^{a'} \pa t^{b_2}} +
\frac{\pa  \om^{a'}}{\pa t^{b_2}}
\frac{ \pa^2 \om^a}{\pa t^{a'} \pa t^{b_1}}=
- \big( d \sfA_{\om} + L_{\om}(\sfA_{\om}) \big)^a_{b_1 b_2}\,.
\end{multline*}
Thus equation \eqref{cAom-d-om-flat} holds. 

Due to Corollary \ref{cor:gl-d-om}, we have 
$$
i_{\bv} \, \frac{\pa^2 \om^a}{\pa t^{b_1} \pa t^{b_2}} =
 \frac{\pa^2 }{\pa t^{b_1} \pa t^{b_2}} (i_{\bv}(\om^a)) =
- \frac{\pa^2 }{\pa t^{b_1} \pa t^{b_2}} (\mv^a_b t^b) = 0\,. 
$$
Thus the first equation in \eqref{i-bv-L-bv-At-rep} holds. 

The second equation in  \eqref{i-bv-L-bv-At-rep} follows 
from the first one and  \eqref{cAom-d-om-flat}.

Next, we observe that, due to equation \eqref{I-know},
\begin{equation}
\label{I-La-al-beta}
\tau(\La_{\beta \al})  = I(\La_{\beta \al}) =  J_{x_{\beta}} \,  J_{x_{\al}}^{-1}
\end{equation}
on every  non-empty intersection $U_{\al} \cap U_{\beta}$\,. 

Furthermore, since $\tau$ is compatible with the Schouten bracket, we 
have 
\begin{align*}
\tau \big(dx^b_{\beta} \pa_{x^b_{\beta}} (\La_{\al \beta}) \big) &=
dt^{b'} (J_{\beta})^b_{b'}\, \tau ( \pa_{x^b_{\beta}} \La_{\al \beta}) =
dt^{b'} (J_{\beta})^b_{b'}\, \tau ( \pa_{x^b_{\beta}}) \big( \tau(\La_{\al \beta}) \big)
\\
&= dt^{b'} (J_{\beta})^b_{b'}\, (J^{-1}_{\beta})^{b''}_b \frac{\pa}{\pa t^{b''}} \big( I(\La_{\al \beta}) \big) =
d t^{b'} \frac{\pa}{\pa t^{b'}} \big( I(\La_{\al \beta}) \big) 
\\&= d t^{b} \frac{\pa}{\pa t^{b}} (J_{x_{\al}}
J^{-1}_{x_{\beta}})\,.
\end{align*}

Therefore
\begin{equation}
\label{tau-cA}
\begin{aligned}
\tau(\sfA_{\al \beta}) &= \tau\Big( (\La_{\beta \al})^{b_1}_a  \, d (\La_{\al \beta})^a_{b_2} \, 
 \pa_{x^{b_1}_{\beta}} \otimes d x^{b_2}_{\beta} \Big) 
\\
  &=\left( J_{x_{\beta}} \,  J_{x_{\al}}^{-1} d t^{c} \frac{\pa}{\pa t^{c}} (J_{x_{\al}}
J^{-1}_{x_{\beta}})  \right)^{b_1}_{~b_2}  \,
  (J^{-1}_{x_{\beta}})^{b'_1}_{b_1}  (J_{x_{\beta}})^{b_2}_{b'_2}
  \pa_{t^{b'_1}} \otimes d t^{b'_2} 
\\
&= d t^{c}  \big(J_{x_{\al}}^{-1} \frac{\pa}{\pa t^{c}} (J_{x_{\al}}) \big)^{b_1}_{~b_2} 
  \pa_{t^{b_1}} \otimes d t^{b_2} -
 d t^{c}  \big(J_{x_{\beta}}^{-1} \frac{\pa}{\pa t^{c}} (J_{x_{\beta}}) \big)^{b_1}_{~b_2} 
  \pa_{t^{b_1}} \otimes d t^{b_2}\,.
\end{aligned}
\end{equation}
It is not hard to see that for every affine chart $U_{\al}$ the section 
$$
 d t^{c}  \big(J_{x_{\al}}^{-1} \frac{\pa}{\pa t^{c}} (J_{x_{\al}}) \big)^{b_1}_{~b_2} 
  \pa_{t^{b_1}} \otimes d t^{b_2}  
$$
is $\mgl_d(\bbK)$-invariant. 

Hence computation \eqref{tau-cA} implies that 
\begin{equation}
\label{tau-cA-Cech}
\tau(\sfA) = - \cpa(\sfA')\,,  
\end{equation}
where $\sfA'$ is the \v{C}ech $0$-cochain of
$$
\Big( \Om^0(\cO^{\coord}_X) \otimes T^{1,2}(P) \Big)^{\mgl_d(\bbK)} =  
\Big( \Om^0(\cO^{\coord}_X)  \otimes  T^{1,2}(P) \Big)^{[\mgl_d(\bbK)]}
$$ 
given by the equation
\begin{equation}
\label{cA-prime}
\sfA'_{\al} =  d t^{c}  \big(J_{x_{\al}}^{-1} \frac{\pa}{\pa t^{c}} (J_{x_{\al}}) \big)^{b_1}_{~b_2} 
  \pa_{t^{b_1}} \otimes d t^{b_2}\,.
\end{equation}

Thus, equation \eqref{tau-cA-Cech} implies that $\tau(\sfA)$ is cohomologous to the 
cocycle $\sfA''$ with 
\begin{equation}
\label{d-Lom-cA-pr}
\sfA''_{\al} = (d + L_{\om})  \sfA'_{\al}
\end{equation}

For the components of $ (d + L_{\om})  \sfA'_{\al} $ we have
\begin{multline}\label{At-comp0} 
\Big( (d + L_{\om})  \sfA'_{\al} \Big)^a_{b_1 b_2} = 
(J^{-1}_{\al})^a_{a'} \frac{\pa^2}{ \pa t^{b_1}  \pa t^{b_2}} 
\, d\, \wt{x}^{a'} 
+   \frac{\pa^2  \wt{x}^{a'}}{ \pa t^{b_1}  \pa t^{b_2}} \, d (J^{-1}_{\al})^a_{a'}
\\
+ \om^c \pa_{t^c} \left(   \frac{\pa^2  \wt{x}^{a'}}{ \pa t^{b_1}  \pa t^{b_2}} \, (J^{-1}_{\al})^a_{a'}
\right) 
-   \frac{\pa^2  \wt{x}^{a'}}{ \pa t^{b_1}  \pa t^{b_2}} \, (J^{-1}_{\al})^c_{a'} \, \pa_{t^c} \om^a
\\
+ \pa_{t^{b_1}} \om^c   \frac{\pa^2  \wt{x}^{a'}}{ \pa t^{c}  \pa t^{b_2}} \, (J^{-1}_{\al})^a_{a'} +
  \pa_{t^{b_2}} \om^c   \frac{\pa^2  \wt{x}^{a'}}{ \pa t^{b_1}  \pa t^{c}} \, (J^{-1}_{\al})^a_{a'}\,.
\end{multline}

Using the equation $(d +\om^c \pa_{t^c}) \, \wt{x}^{a'}  = 0$ and combining 
the first and the third term in the right hand side of  \eqref{At-comp0} we get 
\begin{equation}
\label{At-comp1}
\begin{aligned}
&(J^{-1}_{\al})^a_{a'} \frac{\pa^2}{ \pa t^{b_1}  \pa t^{b_2}}  \, d\, \wt{x}^{a'}  
+  \om^c \pa_{t^c} \left(   \frac{\pa^2  \wt{x}^{a'}}{ \pa t^{b_1}  \pa t^{b_2}} \, (J^{-1}_{\al})^a_{a'} \right)
\\
&= \om^c \pa_{t^c} \left(   \frac{\pa^2  \wt{x}^{a'}}{ \pa t^{b_1}  \pa t^{b_2}} \, (J^{-1}_{\al})^a_{a'} \right)  
 - (J^{-1}_{\al})^a_{a'} \frac{\pa^2}{ \pa t^{b_1}  \pa t^{b_2}} \left( \om^c 
\frac{\pa \wt{x}^{a'}}{\pa t^c} \right) 
\\
&= 
\begin{multlined}
- \frac{\pa^2  \om^a }{ \pa t^{b_1}  \pa t^{b_2}} - (\pa_{t^{b_1}} \om^c) 
 \frac{\pa^2  \wt{x}^{a'}}{ \pa t^{c}  \pa t^{b_2}} \, (J^{-1}_{\al})^a_{a'} 
 -  (\pa_{t^{b_2}} \om^c) 
  \frac{\pa^2  \wt{x}^{a'}}{ \pa t^{b_1}  \pa t^{c}} \, (J^{-1}_{\al})^a_{a'}
 \\
  + \frac{\pa^2  \wt{x}^{a'}}{ \pa t^{b_1}  \pa t^{b_2}} \,   \om^c \pa_{t^c} (J^{-1}_{\al})^a_{a'} \,.
  \end{multlined}
\end{aligned}
\end{equation}

Therefore, we can rewrite equation \eqref{At-comp0} as follows:

\begin{multline}
\label{At-comp11} 
\Big( (d + L_{\om})  \sfA'_{\al} \Big)^a_{b_1 b_2} = - \frac{\pa^2  \om^a }{ \pa t^{b_1}  \pa t^{b_2}} 
\\
+ \frac{\pa^2  \wt{x}^{a'}}{ \pa t^{b_1}  \pa t^{b_2}} \,
\Big( d (J^{-1}_{\al})^a_{a'} +     \om^c \pa_{t^c} (J^{-1}_{\al})^a_{a'} 
-  (J^{-1}_{\al})^c_{a'} \, \pa_{t^c} \om^a
\Big)\,.
\end{multline}

Let us observe that the expression $ d (J^{-1}_{\al})^a_{a'} +     \om^c \pa_{t^c} (J^{-1}_{\al})^a_{a'} 
-  (J^{-1}_{\al})^c_{a'} \, \pa_{t^c} \om^a$ is the $a$-th component of 
$$
(d+ L_{\om})\big( (J^{-1}_{\al})^a_{a'} \pa_{t^a} \big)
$$
which is zero due to \eqref{Jpa-t-is-flat}.

Thus 
$$
\Big( (d + L_{\om})  \sfA'_{\al} \Big)^a_{b_1 b_2} = - \frac{\pa^2  \om^a }{ \pa t^{b_1}  \pa t^{b_2}} 
$$
and Theorem \ref{thm:At-Fed} follows. 
\end{proof}

\section{Reminder of Kontsevich's graph complex $\GC$} 
\label{sec:fGC}

\subsection{The operad $\Gra$ and its action on $T_{\poly}(P)$}
We start by recalling the graded operad $\Gra$ \cite[Section 7]{notes}, \cite{Thomas}. 

For this purpose, we introduce an auxiliary set $\gra_n$.  
An element of $\gra_{n}$ is a labelled graph $\G$
with $n$ vertices and with the additional piece 
of data: the set of edges of $\G$ is equipped with a 
total order. An example of an element in $\gra_4$ is 
shown on figure \ref{fig:exam}. 
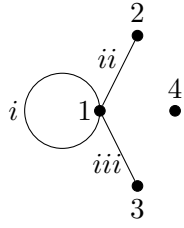
\begin{figure}[htp] 
\centering 
\begin{tikzpicture}[scale=0.5, >=stealth']
\tikzstyle{w}=[circle, draw, minimum size=4, inner sep=1]
\tikzstyle{b}=[circle, draw, fill, minimum size=4, inner sep=1]
\node [b] (b1) at (0,0) {};
\draw (-0.4,0) node[anchor=center] {{\small $1$}};
\node [b] (b2) at (1,2) {};
\draw (1,2.6) node[anchor=center] {{\small $2$}};
\node [b] (b3) at (1,-2) {};
\draw (1,-2.6) node[anchor=center] {{\small $3$}};
\node [b] (b4) at (2,0) {};
\draw (2,0.6) node[anchor=center] {{\small $4$}};
\draw (b1) edge (b2);
\draw (0.2,1.4) node[anchor=center] {{\small $ii$}};
\draw (b1) edge (b3);
\draw (0.2,-1.4) node[anchor=center] {{\small $iii$}};
\draw (-1,0) circle (1);
\draw (-2.3,0) node[anchor=center] {{\small $i$}};
\end{tikzpicture}
\caption{The Roman numerals 
indicate that we chose the total order on 
the set of edges $(1,1) < (1,2) < (1,3)$} \label{fig:exam}
\end{figure}
We will often use Roman numerals to specify total orders 
on sets of edges. Thus the Roman numerals on figure \ref{fig:exam} 
indicate that we chose the total order $(1,1) < (1,2) < (1,3)$\,.

The space  $\Gra(n)$ (for $n \ge 1$) is spanned by elements of 
$\gra_n$, modulo the relation $\G^{\si} = (-1)^{|\si|} \G$,
where the elements $\G^{\si}$ and $\G$ correspond to the same
labelled graph but differ only by permutation $\si$
of edges. We also declare that 
the degree of a graph $\G$ in $\Gra(n)$ equals 
$-e(\G)$, where $e(\G)$ is the number of edges in $\G$\,.
For example, the graph $\G$ on figure \ref{fig:exam} has 
$3$ edges. Thus its degree is $-3$\,. 

Finally, we set 
\begin{equation}
\label{Gra-0}
\Gra(0) = \bfzero\,.
\end{equation}

The symmetric group $S_n$ acts on $\Gra(n)$ in the obvious 
way by rearranging labels on vertices and 
elementary operadic insertions
$$
\circ_i : \Gra(n) \otimes \Gra(k) \to \Gra(n+k-1)
$$
are defined using natural operations with labeled graphs 
(we refer the reader for more details to \cite[Section 7]{notes}). 
  
To define an action of $\Gra$ on $T_{\poly}(P)$ (with $P=\bbK[[t^1, t^2, \dots, t^d]] $)
we identify $T_{\poly}(P)$ with the graded commutative algebra 
\begin{equation}
\label{T-poly-same}
P[\xi_1, \xi_2, \dots, \xi_d]\,,
\end{equation}
where $\xi_a$'s are degree $1$ auxiliary variables. 

Next, given an element $\G \in \gra_n$ and polyvectors 
$v_1, \dots, v_n \in P[\xi_1, \xi_2, \dots, \xi_d]$, we 
set\footnote{Note that, in computing the right hand side of 
equation \eqref{Gra-acts}, we use the Koszul rule of signs.} 
\begin{equation}
\label{Gra-acts}
\G(v_1, \dots, v_n) : = \mult_n  
\Big(
\Big[\prod_{(i,j) \in E(\G)} \Lap_{(i,j)} \Big]
\,(v_1  \otimes v_2 \otimes \dots \otimes v_n )\,
 \Big) \,,
\end{equation}
where   $\mult_n$ is the multiplication map 
$$
\mult_n : \big( T_{\poly}(P) \big)^{\otimes \, n}  \to 
T_{\poly}(P)\,,
$$ 
$E(\G)$ is the set of edges of $\G$,
\begin{equation}
\label{Lap}
\Lap_{(i,j)} = \sum_{a=1}^d 1 \otimes \dots \otimes 1 \otimes 
\underbrace{\pa_{\xi_a}}_{i\textrm{-th slot}} \otimes 1 \otimes \dots \otimes 1 
\otimes 
\underbrace{\pa_{x^a}}_{j\textrm{-th slot}}  \otimes 1  \otimes \dots  \otimes 1 +
\end{equation}
$$
\sum_{a=1}^d 1 \otimes \dots \otimes 1 \otimes 
\underbrace{\pa_{x^a}}_{i\textrm{-th slot}} \otimes 1 \otimes \dots \otimes 1 
\otimes 
\underbrace{\pa_{\xi_a}}_{j\textrm{-th slot}}  \otimes 1  \otimes \dots  \otimes 1
$$
and the order of operators $\Lap_{(i,j)}$ coincides with the total order 
on the set of edges of $\G$\,.

It is not hard to see that equation \eqref{Gra-acts} defines 
an action of the operad $\Gra$ on $T_{\poly}(P)$\,.

Let us also recall that the operad $\Gra$ receives a natural embedding 
\begin{equation}
\label{Ger-Gra}
\io : \Ger \to \Gra
\end{equation}
from the operad $\Ger$\,. Namely, the embedding $\io$ of $\Ger$ into 
$\Gra$ is defined on generators by the formulas: 
\begin{equation}
\label{iota-dfn}
\io(a_1 a_2) : = \G_{\bb}\,, \qquad 
\io(\{a_1, a_2\}) : = \G_{\ed}\,, 
\end{equation}
where 
\begin{equation}
\label{binary}
\G_{\bb} =   \begin{tikzpicture}[scale=0.5, >=stealth']
\tikzstyle{w}=[circle, draw, minimum size=4, inner sep=1]
\tikzstyle{b}=[circle, draw, fill, minimum size=4, inner sep=1]
\node [b] (b1) at (0,0) {};
\draw (0,0.6) node[anchor=center] {{\small $1$}};
\node [b] (b2) at (1.5,0) {};
\draw (1.5,0.6) node[anchor=center] {{\small $2$}};
\end{tikzpicture}
\qquad \quad
\G_{\ed} =   \begin{tikzpicture}[scale=0.5, >=stealth']
\tikzstyle{w}=[circle, draw, minimum size=4, inner sep=1]
\tikzstyle{b}=[circle, draw, fill, minimum size=4, inner sep=1]
\node [b] (b1) at (0,0) {};
\draw (0,0.6) node[anchor=center] {{\small $1$}};
\node [b] (b2) at (1.5,0) {};
\draw (1.5,0.6) node[anchor=center] {{\small $2$}};
\draw (b1) edge (b2);
\end{tikzpicture}
\end{equation}

The Gerstenhaber algebra structure on $T_{\poly}(P)$ is 
induced by the embedding \eqref{Ger-Gra}.

\begin{remark}
\label{rem:Gra-on-aff-space}
It is easy to see that equation \eqref{Gra-acts} defines 
an action of the operad $\Gra$ on polyvector fields 
on an affine space. Although equation \eqref{Gra-acts}
does not make sense for polyvector fields on an arbitrary (smooth) algebraic variety, 
the actions of $\G_{\bb}$ and $\G_{\ed}$ \eqref{binary} 
are well defined in this setting. 
\end{remark}

\subsection{The full graph complex $\fGC$}
Consider the convolution Lie algebra 
\begin{equation}
\label{coComm-Gra}
\Conv(\La^2 \coCom, \Gra) = \prod_{n \ge 1} 
\bs^{2n-2} \Big( \Gra(n) \Big)^{S_n}.
\end{equation}
The graph $\G_{\ed}$ of \eqref{binary} is symmetric under the interchange of the vertex labels and hence defines an element of $\Conv(\La^2 \coCom, \Gra)$, which we also denote by $\G_{\ed}$\,.  
One easily checks that $\G_{\ed}$ is a Maurer-Cartan element, i.e., it has degree $1$ in \eqref{coComm-Gra} and 
\begin{equation}
\label{MC-G-ed}
[\G_{\ed}, \G_{\ed} ] = 0\,.
\end{equation}

The {\it full graph complex} $\fGC$ is the cochain complex 
$\Conv(\La^2 \coCom, \Gra)$ with the differential
\begin{equation}
\label{diff-fGC}
\pa = [\G_{\ed}, ~]\,.
\end{equation}

In \cite[Section 5]{K-conj}, M. Kontsevich introduced the subcomplex
\begin{equation}
\label{GC}
\GC \subset \fGC
\end{equation}
which consists of infinite sums 
\begin{equation}
\label{vect-GC}
\sum_{n \ge 4} X_n\,, \qquad X_n \in \bs^{2n-2} \Big( \Gra(n) \Big)^{S_n} 
\end{equation}
where each graph $\G$ in the linear combination $X_n$ satisfies these 
properties 
\begin{itemize}

\item $\G$ is connected, 

\item $\G$ is 1-vertex irreducible\footnote{A connected graph $\G$ is called  
\emph{1-vertex irreducible} if the complement of each vertex of $\G$ is 
connected.}, and 

\item each vertex of $\G$ has valency $\ge 3$\,.

\end{itemize}

We call $\GC$ {\it Kontsevich's graph complex}.

\begin{example}
\label{exam:tetra}
Consider the tetrahedron in $\Gra(4)$ depicted 
on figure \ref{fig:tetra-here}. 
\begin{figure}[htp]
\centering 
\begin{tikzpicture}[scale=0.5, >=stealth']
\tikzstyle{ahz}=[circle, draw, fill=gray, minimum size=26, inner sep=1]
\tikzstyle{ext}=[circle, draw, minimum size=5, inner sep=1]
\tikzstyle{int}=[circle, draw, fill, minimum size=5, inner sep=1]
\node [int] (v1) at (2,0) {};
\draw (2,-0.6) node[anchor=center] {{\small $1$}};
\node [int] (v2) at (4,3) {};
\draw (4,3.6) node[anchor=center] {{\small $2$}};
\node [int] (v3) at (0,3) {};
\draw (0,3.6) node[anchor=center] {{\small $3$}};
\node [int] (v4) at (2,1.8) {};
\draw (2, 2.4) node[anchor=center] {{\small $4$}};
\draw (v1) edge (v2);
\draw (v1) edge (v3);
\draw (v1) edge (v4);
\draw (v2) edge (v3);
\draw (v2) edge (v4);
\draw (v3) edge (v4);
\end{tikzpicture}
~\\[0.3cm]
\caption{We may choose this order on the 
set of edges: $ (1,2) < (1,3) < (1,4) < (2,3)< (2,4) < (3,4)$ } \label{fig:tetra-here}
\end{figure}
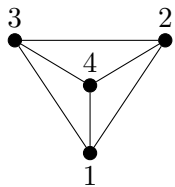 
This vector is invariant with 
respect to the action of $S_4$ and hence it can be 
viewed as vector in $\fGC$\,. It is not hard to see that 
this is a degree zero non-trivial cocycle in $\fGC$\,. Moreover 
the tetrahedron is connected, 1-vertex irreducible and
trivalent.  Thus this is an example of a non-trivial 
degree zero cocycle in $\GC$\,. 
\end{example}

\subsubsection{The Grothendieck-Teichm\"uller Lie algebra versus  $H^0(\GC)$}
\label{sec:grt}
In this small subsection, we give a very brief reminder 
of the Grothendieck-Teichm\"uller Lie algebra\footnote{Strictly speaking, $\grt$ 
is the pro-nilpotent part of the Lie algebra introduced by V. Drinfeld. In \cite{AT} and 
\cite{Drinfeld}, the Lie algebra we use here is denoted by $\grt_1$.} $\grt$\,. 
For more details we refer the reader to \cite{AT} or \cite{Drinfeld}.
We will conclude this subsection with a reminder of the necessary results 
from \cite{Thomas}. 

Let $m$ be an integer $\ge 2$. We denote by $\mt_m$ the Lie algebra  
generated by symbols $\{ t^{ij} = t^{ji} \}_{1 \le i \neq j \le m}$ subject to the following 
relations: 
\begin{multline}
\label{DK-relations} 
[t^{ij}, t^{ik} + t^{jk}] = 0 \qquad \textrm{for any triple of distinct indices } i,j,k\,,\\
[t^{ij}, t^{kl}] = 0 \qquad \textrm{for any quadruple of distinct indices } i,j,k,l\,.
\end{multline}
We also denote by $\hat{\mt}_m$ the degree completion of this Lie algebra. 

Let $\lie(x,y)$ be the degree completion of the free Lie algebra in two symbols $x$ and $y$\,.

As a graded vector space, $\grt$ consists of Lie series $\si(x,y)$ satisfying the 
following equations
\begin{equation}
\label{relations-grt}
\begin{array}{c}
\si(y,x) = - \si(x,y)\,, \\[0.4cm] 
\si(x,y) + \si(y, -x-y) + \si(-x-y, x) = 0\,, \\[0.4cm]
\si(t^{23}, t^{34}) - \si (t^{13} + t^{23}, t^{34}) +
\si(t^{12} + t^{13}, t^{24} + t^{34}) \\[0.2cm]
- \si(t^{12}, t^{23} + t^{24}) + \si(t^{12}, t^{23}) = 0\,.   
\end{array}
\end{equation}

The Lie bracket on $\grt$ is the Ihara bracket which is given by the formula: 
\begin{equation}
\label{Ihara}
[\si, \si']_{\Ih} := \de_{\si} (\si') - \de_{\si'}(\si) + [\si, \si']_{\lie(x,y)}\,,
\end{equation}
where $[~,~]_{\lie(x,y)}$ is the usual bracket on $\lie(x,y)$ and 
$\de_{\si}$ is the continuous derivation of $\lie(x,y)$ defined by 
the equations
$$
\de_{\si}(x) : = 0\,, \qquad \de_{\si}(y) : = [y, \si(x,y)]\,.
$$

Note that relations \eqref{relations-grt} are homogeneous in generators 
$x,y$. So it is sufficient to look only for homogeneous solutions. 

A direct computation shows that relations \eqref{relations-grt} have 
neither linear nor quadratic solutions, and the space of
solutions of degree $3$
is spanned by the element 
$$
\si_3(x,y) : = [x,[x,y]] -[y,[y,x]]\,.  
$$

More generally, we have 
\begin{prop}[Proposition 6.3, \cite{Drinfeld}]
\label{prop:DD-elements}
For every odd integer $n \ge 3$ there exists a non-zero vector 
$\si_n \in \grt$ of degree $n$ in symbols $x$ and $y$ such that 
\begin{equation}
\label{si-n-fix}
\si_n = \ad^{n-1}_{x}(y)  + \dots 
\end{equation}
where $\dots$ is a sum of Lie words of degrees $\ge 2$ in the symbol $y$\,.
\end{prop}
We call elements  $\{ \si_n \}_{n \textrm{ odd } \ge 3}$ {\it Deligne-Drinfeld elements} of $\grt$\,.
\begin{remark}
\label{rem:DD-elem}
There is no canonical choice of the elements $\{ \si_n \}_{n \textrm{ odd } \ge 3}$, i.~e., no canonical choice 
of the Lie words $\dots$ in \eqref{si-n-fix}. Nevertheless, there is an important conjecture \cite[Section 6]{Drinfeld}
(the Deligne-Drinfeld conjecture) which states that the Lie algebra $\grt$ 
is freely generated by elements  $\{ \si_n \}_{n \textrm{ odd } \ge 3}$ regardless of this choice. 
Recently, F. Brown \cite{Brown} proved that there exists a choice of Lie words $\dots$ in 
\eqref{si-n-fix} such that the elements  $\{ \si_n \}_{n \textrm{ odd } \ge 3}$ generate
a free Lie subalgebra $\mmu$ in $\grt$. However, it is still not known whether this subalgebra 
coincides with the full Lie algebra $\grt$. 
\end{remark}

For our purposes, we need the results from \cite{Thomas} which link 
the Lie algebra $\grt$ to $H^0(\GC)$. We assemble Theorem 1.1 
and Proposition 9.1 from \cite{Thomas} into the following 
statement:
\begin{thm}[\cite{Thomas}]
\label{thm:GC-grt}
Let $\GC$ be Kontsevich's graph complex and 
$\grt$ be the Grothendieck-Teichm\"uller  Lie algebra
defined above. Then we have an isomorphism of 
Lie algebras:
\begin{equation}
\label{H0GC-grt}
H^0(\GC) \cong \grt\,.
\end{equation}
Let $n$ be an odd integer $\ge 3$ and $\si_n$ be a homogeneous 
element of $\grt$ satisfying \eqref{si-n-fix}. 
If $\wt{\si}_n$ is the class in $H^0(\GC)$ corresponding to 
$\si_n \in \grt$, then each representative of $\wt{\si}_n$
has a non-zero coefficient in front of the graph shown on figure \ref{fig:wheel}. 
\end{thm}
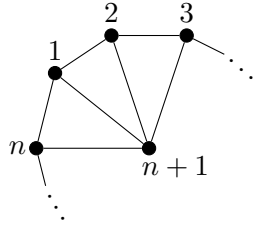
\begin{figure}[htp]
\centering 
\begin{tikzpicture}[scale=0.5, >=stealth']
\tikzstyle{ext}=[circle, draw, minimum size=5, inner sep=1]
\tikzstyle{int}=[circle, draw, fill, minimum size=5, inner sep=1]
\node [int] (v0) at (3,3) {};
\draw (3.7, 2.5) node[anchor=center] {{\small $n+1$}};
\node [int] (vn) at (0,3) {};
\draw (-0.5, 3) node[anchor=center] {{\small $n$}};
\node [int] (v1) at (0.5,5) {};
\draw (0.5, 5.6) node[anchor=center] {{\small $1$}};
\node [int] (v2) at (2,6) {};
\draw (2, 6.6) node[anchor=center] {{\small $2$}};
\node [int] (v3) at (4,6) {};
\draw (4, 6.6) node[anchor=center] {{\small $3$}};
\draw (5.4, 5.15) node[anchor=center, rotate=140] {{\small $\dots$}};
\draw (0.5, 1.5) node[anchor=center, rotate=120] {{\small $\dots$}};
\draw (vn) edge (0.25,2);
\draw (vn) edge (v1);
\draw (v1) edge (v2);
\draw (v2) edge (v3);
\draw (v3) edge (5,5.5);
\draw (vn) edge (v0);
\draw (v1) edge (v0);
\draw (v2) edge (v0);
\draw (v3) edge (v0);
\end{tikzpicture}
~\\[0.3cm]
\caption{Here $n$ is an odd integer $\ge 3$. (We do not specify the 
order on the set of edges.)} \label{fig:wheel}
\end{figure}

\subsection{The action of $\fGC$ and $\GC$ on $T_{\poly}(P)$} 

Let us view $T_{\poly}(P)$ as a $\La\Lie$-algebra. Then,
according to Appendix \ref{app:Def-comp}, the deformation complex of  
 $T_{\poly}(P)$ is 
\begin{equation}
\label{Def-Tpoly-P}
\Def_{\La\Lie}(T_{\poly}(P)) = \Conv(\La^2\coCom, \End_{T_{\poly}(P)})\,,
\end{equation}
where $\End_{T_{\poly}(P)}$ is the endomorphism operad of $T_{\poly}(P)$
and the differential is given by the adjoint action of the Maurer-Cartan element which 
corresponds to the composition 
$$
\Cobar(\La^2 \coCom) \to \La\Lie \to \End_{T_{\poly}(P)}\,. 
$$

The canonical operad morphism defined by \eqref{Gra-acts} 
\begin{equation}
\label{Gra-to-End-Tpoly}
\ma : \Gra \to \End_{T_{\poly}(P)}
\end{equation}
induces a morphism of graded Lie algebras: 
\begin{equation}
\label{fGC-to-DefTpoly}
\ma_* :  \Conv(\La^2\coCom, \Gra) \to 
 \Conv(\La^2\coCom, \End_{T_{\poly}(P)})\,.
\end{equation}

Furthermore, since for the generator $\{a_1, a_2\} \in \La\Lie(2)$, 
$\io(\{a_1, a_2\}) = \G_{\ed}$, the map $\ma_*$ sends the Maurer-Cartan element 
of $ \Conv(\La^2\coCom, \Gra)$ to the Maurer-Cartan element of 
$\Conv(\La^2\coCom,  \End_{T_{\poly}(P)} )$\,.
Hence $\ma_*$ is also a map of dg Lie algebras, and 
restricting $\ma_*$ to 
$$
\GC \subset   
\Conv(\La^2\coCom, \Gra ),
$$
we get a map (of dg Lie algebras) which we denote by the same letter $\ma_{*}$
\begin{equation}
\label{GC-to-DefTpoly}
\ma_{*} : \GC \to \Def_{\La\Lie}(T_{\poly}(P))\,. 
\end{equation}

\subsection{Dg Lie algebras related to $\fGC$}
This section is devoted to the auxiliary dg Lie algebras 
$\Conv(\Ger^{\vee}, \Ger)$ and $\Conv(\Ger^{\vee}, \Gra)$ 
which are used in proving a remarkable 
property of the map from Kontsevich's graph complex $\GC$ to 
the deformation complex of the sheaf of polyvector fields.

First, we recall that the cooperad $\Ger^{\vee}$ is obtained by 
taking the linear dual of the operad $\La^{-2}\Ger$\,. 
Hence, as graded vector spaces, 
\begin{equation}
\label{Conv-Ger-Ger}
\Conv(\Ger^{\vee}, \Ger)  \cong \prod_{n \ge 1} \Big( \Ger(n) \otimes \La^{-2} \Ger(n)  \Big)^{S_n}\,,
\end{equation}
and
\begin{equation}
\label{Conv-Ger-Gra}
\Conv(\Ger^{\vee}, \Gra)  \cong \prod_{n \ge 1} \Big( \Gra(n) \otimes \La^{-2} \Ger(n)  \Big)^{S_n}\,.
\end{equation}

Next,  we identify  $\Ger(n)$ with the subspace of the 
free Gerstenhaber algebra  $\Ger(a_1, \dots, a_n)$ spanned 
by  $\Ger$-monomials in which each symbol from 
the set $\{a_1, a_2, \dots, a_n\}$ appears exactly once.
We also identify  $\La^{-2}\Ger(n)$ with the subspace of the 
free  $\La^{-2}\Ger$-algebra  $\La^{-2}\Ger(b_1, \dots, b_n)$ spanned 
by  $\La^{-2}\Ger$-monomials in which each symbol from 
the set $\{b_1, b_2, \dots, b_n\}$ appears exactly once.  

Then vectors in \eqref{Conv-Ger-Ger} are infinite sums 
\begin{equation}
\label{sums-Ger-Ger}
\sum_{n\ge 1} Z_n\,,
\end{equation}
\begin{equation}
\label{Z-n}
Z_n = \sum_{j} Z_{n,j} \otimes w_{n,j} \in   \Big( \Ger(n) \otimes \La^{-2} \Ger(n)  \Big)^{S_n}\,,
\end{equation}
where $Z_{n,j} \in \Ger(n)$, 
$w_{n,j}$ is a $\La^{-2}\Ger$-monomial of the form
\begin{equation}
\label{Ger-monomials}
\vf_1(b_{i_{11}}, \dots, b_{i_{1 k_1}}) 
\vf_2(b_{i_{21}}, \dots, b_{i_{2 k_2}}) \dots
\vf_q(b_{i_{q1}}, \dots, b_{i_{q k_q}})\,, 
\end{equation}
$\vf_1, \dots, \vf_q$ are $\La^{-1}\Lie$-monomials, and each symbol
from the set $\{b_1, b_2, \dots, b_n\}$ appears in \eqref{Ger-monomials}
exactly once. 

Similarly, vectors in \eqref{Conv-Ger-Gra} are infinite sums 
\begin{equation}
\label{sums-Ger-Gra}
\sum_{n\ge 1} Y_n\,,
\end{equation}
\begin{equation}
\label{Y-n}
Y_n = \sum_{j} Y_{n,j} \otimes w_{n,j} \in   \Big( \Gra(n) \otimes \La^{-2} \Ger(n)  \Big)^{S_n}\,,
\end{equation}
where $Y_{n,j} \in \Gra(n)$ and $w_{n,j}$'s are as above. 

The canonical operad morphism 
\begin{equation}
\label{Cobar-Ger-Ger}
\Cobar(\Ger^{\vee}) \to \Ger
\end{equation}
corresponds to the Maurer-Cartan element 
\begin{equation}
\label{al-Ger}
\al_{\Ger} = \{a_1, a_2\} \otimes b_1 b_2 +  a_1 a_2 \otimes \{b_1, b_2\}\,,
\end{equation}
which allows us to equip the graded Lie algebra \eqref{Conv-Ger-Ger}
with the differential 
\begin{equation}
\label{diff-Ger-Ger}
[\al_{\Ger}, ~]\,.
\end{equation}
We recall \cite[Section 11]{notes} that the cochain complex  \eqref{Conv-Ger-Ger}
with the differential \eqref{diff-Ger-Ger} is called the {\it extended deformation 
complex} of the operad $\Ger$. 

Using the map 
$$
\io_* : \Conv(\Ger^{\vee}, \Ger) \to \Conv(\Ger^{\vee}, \Gra)
$$
induced by $\io$ \eqref{Ger-Gra} we get the following 
Maurer-Cartan element 
\begin{equation}
\label{al}
\al : = \io_*(\al_{\Ger}) =  \G_{\ed} \otimes b_1 b_2 + \G_{\bb} \otimes \{b_1, b_2\}
\end{equation}
in the graded Lie algebra $\Conv(\Ger^{\vee}, \Gra)$.

Furthermore, just as for $\Conv(\Ger^{\vee}, \Ger)$, we use $\al$ to 
equip the graded Lie algebra  $\Conv(\Ger^{\vee}, \Gra)$ with the differential:
\begin{equation}
\label{diff-Ger-Gra}
[\al, ~]\,.
\end{equation}

Let us observe that, using the map \eqref{Ger-Gra},  we can 
embed the vector spaces $\Ger(n) \otimes \Ger^{\vee}(n)$ and 
$\Gra(n) \otimes \Ger^{\vee}(n)$ in the vector space
\begin{equation}
\label{Gra-Gra}
\Gra(n) \otimes \La^{-2}\Gra(n)\,.
\end{equation}
Furthermore, it is convenient to represent vectors in (\ref{Gra-Gra}) by formal 
linear combinations of labeled graphs with two types 
of edges:  solid edges for left tensor factors and
dashed edges for right tensor factors. 

Using this interpretation, we introduce the following 
subspaces of  $\Conv(\Ger^{\vee}, \Ger)$ and  $\Conv(\Ger^{\vee}, \Gra)$, 
respectively: 
\begin{equation}
\label{Ger-Ger-conn}
\Conv(\Ger^{\vee}, \Ger)_{\conn} \subset  \Conv(\Ger^{\vee}, \Ger)\,,
\end{equation}
\begin{equation}
\label{Ger-Gra-conn}
\Conv(\Ger^{\vee}, \Gra)_{\conn} \subset \Conv(\Ger^{\vee}, \Gra)\,.
\end{equation}
Here $\Conv(\Ger^{\vee}, \Ger)_{\conn}$ (resp. $\Conv(\Ger^{\vee}, \Gra)_{\conn}$) 
consists of vectors \eqref{sums-Ger-Ger} (resp. \eqref{sums-Ger-Gra}) for which 
images of $Z_n$ (resp. $Y_n$) in \eqref{Gra-Gra} are sums of connected graphs. 
For example, the vectors $\al_{\Ger}$ and $\al$ belong to  $\Conv(\Ger^{\vee}, \Ger)_{\conn}$ 
and  $\Conv(\Ger^{\vee}, \Gra)_{\conn}$, respectively, while the vector 
$a_1 a_2 \otimes b_1 b_2  \in \big( \Ger(2) \otimes \La^{-2}\Ger(2) \big)^{S_2}$ 
does not belong to  $\Conv(\Ger^{\vee}, \Ger)_{\conn}$\,.    

It is easy to see that $\Conv(\Ger^{\vee}, \Ger)_{\conn}$ (resp. $\Conv(\Ger^{\vee}, \Gra)_{\conn}$) 
is a subcomplex of $\Conv(\Ger^{\vee}, \Ger)$ (resp. $\Conv(\Ger^{\vee}, \Gra)$).

For our purposes, we also need the subspace 
\begin{equation}
\label{Xi-conn}
\Xi_{\conn} \subset \Conv(\Ger^{\vee}, \Ger)_{\conn}\,.
\end{equation}
This subspace consists of sums  \eqref{sums-Ger-Ger} in 
$\Conv(\Ger^{\vee}, \Ger)_{\conn}$ for which each $\La^{-1}\Lie$-monomial 
in \eqref{Ger-monomials} has length $\ge 2$. 

Even though $\al_{\Ger} \notin \Xi_{\conn}$, the subspace
$\Xi_{\conn}$ is a subcomplex of $\Conv(\Ger^{\vee}, \Ger)_{\conn}$. 
This subcomplex plays an important role in establishing a link between 
the cohomology of the graph complex $\fGC$ and the cohomology of 
the extended deformation complex of $\Ger$.

Next, we consider the subspace of $\Conv(\Ger^{\vee}, \Gra)_{\conn}$ which 
consists of sums \eqref{sums-Ger-Gra} in $\Conv(\Ger^{\vee}, \Gra)_{\conn}$
satisfying the following property:
\begin{pty}
\label{pty:ge-3}
If a $\La^{-1} \Lie$-word $\varphi_r$ occurring in $w_{n,j}$ (as in \eqref{Ger-monomials}) has length $1$, i.e., if $\varphi_r=b_i$ for some $i$, then, in each graph of $Y_{n,j}$, the vertex with label $i$
is at least trivalent. 
\end{pty}

It is not hard to see that this subspace is a subcomplex of  $\Conv(\Ger^{\vee}, \Gra)_{\conn}$
and we denote it by 
\begin{equation}
\label{ge-3}
\Conv(\Ger^{\vee}, \Gra)_{\ge 3}
\end{equation}

We introduce the map $\mR$ of cochain complexes: 
$$
\mR : \Conv(\Ger^{\vee}, \Gra)_{\ge 3} \to \fGC
$$
\begin{equation}
\label{mR}
\mR(f) = f \Big|_{\La^2 \coCom}
\end{equation}
and observe that the subcomplex 
\begin{equation}
\label{ker-mR}
\ker(\mR) \subset \Conv(\Ger^{\vee}, \Gra)_{\ge 3}
\end{equation}
receives an obvious map $\psi$ from the cochain complex 
$\Xi_{\conn}$ \eqref{Xi-conn} 
which is induced by the map $\io$ \eqref{Ger-Gra}.

We claim that 
\begin{prop}
\label{prop:Xi-ker-mR}
The map 
\begin{equation}
\label{Xi-ker-mR}
\psi: \Xi_{\conn} \to \ker(\mR)
\end{equation}
induces an isomorphism on the level of cohomology.
\end{prop}
\begin{proof}
This statement is an analogue of Proposition 13.1 in \cite{notes}.  

To prove this statement we only need the following minor modification
of the proof of  \cite[Proposition 13.1]{notes}: we need to replace in {\it loc. cit.} the dg operad 
$\fgraphs$ by the dg operad $\graphs$ (see also \cite[Section 9.4]{notes}). 
Thus, since the embedding $\graphs \hookrightarrow \fgraphs$ is a quasi-isomorphism,
the modified proof goes through. 
\end{proof}

\section{The chain map $\Te : \GC  \to  \Def_{\La\Lie}(\FR)$}
\label{sec:GC-DefTpoly}

Let us recall that, due to Theorem \ref{thm:Fed-Tpoly}, 
the sheaf of dg  Gerstenhaber algebras
\begin{equation}
\label{FR}
\FR := \Big( \Omb(\cO_X^{\coord}) \otimes T_{\poly}(P) \Big)^{[\mgl_d(\bbK)]}
\end{equation}
is a resolution of the sheaf of  Gerstenhaber algebras $\cT_{\poly}$ on a 
smooth algebraic variety $X$. 

In this section, we view $\FR$ as the sheaf of $\La\Lie$-algebras
and construct a map of dg Lie algebras
\begin{equation}
\label{GC-Def-T-poly}
\Te : \GC \to \Def_{\La\Lie}(\FR)\,. 
\end{equation}

For this purpose, we denote by $\FR'$ the following sheaf of 
dg $\La\Lie$ algebras
\begin{equation}
\label{FR-pr}
\FR' =  \Omb(\cO_X^{\coord}) \otimes T_{\poly}(P) 
\end{equation}
with the de Rham differential $d$\,.

We recall that the sheaf of dg $\La\Lie$ algebras
$\FR$ \eqref{FR} is obtained from $\FR'$ in two steps. 
First, we twist\footnote{See Appendix \ref{app:twisting} for terminology and details 
of the twisting procedure.} the dg $\La\Lie$-algebra structure by 
the Maurer-Cartan element $\om$ introduced in Theorem \ref{thm:om}.
Under this procedure, the $\La\Lie$-bracket remains the same $\{~,~\}_{SN}$ 
and the differential $d$ gets replaced by 
\begin{equation}
\label{d-om}
d + \{\om, ~ \}_{SN}\,.
\end{equation}

Second, we apply trimming (see Section \ref{sec:trimming}) to the resulting 
sheaf of dg $\La\Lie$-algebras with respect to the set of degree $-1$ derivations 
$i_{\bv}$, $\mv \in \mgl_d(\bbK)$ and get \eqref{FR}. 

So to construct \eqref{GC-Def-T-poly}, we  define   
an auxiliary map of dg Lie algebras
\begin{equation}
\label{Te-pr}
\Te' : \GC \to \coDer \big( \La^2 \coCom (\FR') \big)\,.
\end{equation}
via extending the map $\ma_{*}$ \eqref{GC-to-DefTpoly} by 
linearity over $ \Omb(\cO_X^{\coord}) $\,.
Here the codomain 
$$
\coDer \big( \La^2 \coCom (\FR') \big)
$$
carries the differential 
\begin{equation}
\label{diff-d-Q}
[d+ Q, ~]\,,
\end{equation}
with $d$ being the de Rham differential and $Q$ being the 
coderivation of $ \La^2 \coCom (\FR')$ coming from the $\La\Lie$-bracket 
$\{~,~\}_{SN}$\,.

Since $\Te'$ is obtained via extending $\ma_*$ \eqref{GC-to-DefTpoly} by 
linearity over $ \Omb(\cO_X^{\coord})$ and $\ma_*$ is a chain map, we
conclude that $\Te'$ intertwines the differential 
$\pa$ \eqref{diff-fGC} with the differential $d+ Q$, i.e. 
\begin{equation}
\label{CE-diff}
[ d+ Q,   \Te'(\ga)  ] = 
\Te'(\pa \ga)\,, \qquad \forall ~~\ga \in \GC\,.
\end{equation}

For the map $\Te'$, we have the following statement: 
\begin{prop}
\label{prop:Te-pr-tw}
Let  $\om = \om^a \pa_{t^a}$ be the global section of the sheaf
$$
 \Om^1(\cO_X^{\coord})  \otimes T^{1,0}(P)
$$
introduced in Theorem \ref{thm:om}. Then the formula  
\begin{equation}
\label{Te-pr-tw-dfn}
(\Te')^{\om}(\ga)  = e^{-\bs^{-2}\, \om}\, \Te'(\ga)\, e^{\bs^{-2}\, \om}\,, 
\qquad \ga \in \GC
\end{equation}
defines a map of dg Lie algebras 
\begin{equation}
\label{Te-pr-tw}
(\Te')^{\om} :   \GC \to \coDer \big( \La^2 \coCom (\FR') \big)\,,
\end{equation}
where $ \coDer \big( \La^2 \coCom (\FR') \big)$ is considered 
with the differential 
$$
[d +  \{\om, ~\}_{SN} + Q, ~]\,. 
$$
Furthermore, $(\Te')^{\om}$ descends to a map of dg Lie algebras
\begin{equation}
\label{Te-dfn-almost}
\GC   ~ \to ~   \coDer \big( \La^2 \coCom (\FR) \big)\,,
\end{equation}
where 
$$
\FR = \Big( \Omb(\cO_X^{\coord})  \otimes  T_{\poly}(P) \Big)^{[\mgl_d(\bbK)]}\,.
$$
\end{prop}
\begin{proof}
We remark that, using the degree of exterior forms, we equip the 
sheaf $\FR'$ of dg $\La\Lie$-algebras with the natural decreasing 
filtration. This filtration is complete, and hence, we may apply to $\FR'$ 
the operation of twisting (see Appendix \ref{app:twisting}).

Let us denote by $p$ the canonical projection
\begin{equation}
\label{p-FR-pr}
p  :  \La^2 \coCom (\FR')  \to  \FR' \,,
\end{equation}
and prove that for all $n\ge 1$
\begin{equation}
\label{Te-pr-si-omomom}
p \circ \Te'(\ga) \big( \bs^2 (\bs^{-2}\, \om)^n \big) = 0\,.
\end{equation}

For this purpose,  we recall that an action of 
a graph $\G$ on a collection of polyvector fields is expressed in 
terms of the operators \eqref{Lap}. 

So we will keep track of terms involving the sum
$$
 \sum_{a=1}^d 1 \otimes \dots \otimes 1 \otimes 
\underbrace{\pa_{\xi_a}}_{i\textrm{-th slot}} \otimes 1 \otimes \dots \otimes 1 
\otimes 
\underbrace{\pa_{x^a}}_{j\textrm{-th slot}}  \otimes 1  \otimes \dots  \otimes 1 
$$
by choosing a direction on the edge $(i,j)$ from vertex $i$ to vertex $j$. 
Similarly, we will keep track of terms involving the sum
$$ 
\sum_{a=1}^d 1 \otimes \dots \otimes 1 \otimes 
\underbrace{\pa_{x^a}}_{i\textrm{-th slot}} \otimes 1 \otimes \dots \otimes 1 
\otimes 
\underbrace{\pa_{\xi_a}}_{j\textrm{-th slot}}  \otimes 1  \otimes \dots  \otimes 1
$$   
by choosing a direction on the edge $(i,j)$ from vertex $j$ to vertex $i$. 

Since $\om$ is vector-valued, equation \eqref{Te-pr-si-omomom}
is a consequence of the following simple combinatorial fact: 
\begin{claim}
\label{cl:graphs-ge-3}
If each vertex of a graph $\G$ has valency $\ge 3$ then edges of 
$\G$ cannot be oriented in such a way that all vertices have exactly one 
outgoing edge. \qed
\end{claim}

Thus \eqref{Te-pr-si-omomom} indeed holds. 

Now, it is easy to see that equations \eqref{CE-diff}, \eqref{Te-pr-si-omomom}, 
and Corollary \ref{cor:twisting} from Appendix \ref{app:twisting} imply that 
formula \eqref{Te-pr-tw-dfn} indeed defines a map from the graph complex 
$\GC$ to the cochain complex $ \coDer \big( \La^2 \coCom (\FR') \big)$ with
the differential 
$$
[d +  \{\om, ~\}_{SN} + Q, ~]\,. 
$$
The compatibility of this map with the Lie brackets is obvious. 

It remains to prove that for every cochain $\ga \in \GC$ and 
for any set $v_1, \dots, v_n$ of local sections of $\FR$ \eqref{FR} the 
section 
\begin{equation}
\label{Te-pr-om-gagaga}
v = p \circ (\Te')^{\om}(\ga)\big( \bs^2 (\bs^{-2}\, v_1\, 
\bs^{-2}\, v_2 \, \dots\, \bs^{-2}\, v_n) \big)  
\end{equation}
satisfies the conditions 
\begin{equation}
\label{i-bv-v}
i_{\bv} (v) = 0
\end{equation}
and 
\begin{equation}
\label{LLL-bv-v}
i_{\bv} (d v + \{\om, v \}_{SN}) = 0 \qquad \forall ~~ \mv \in \mgl_d(\bbK).
\end{equation}

Let 
$$
\ga = \sum_{\tau \in S_N}  \tau(\G)
$$
for an element $\G \in \gra_N$. Then    
$$
v = \frac{1}{r!}\, p \circ \Te'(\ga)\big( \bs^2 ((\bs^{-2} \om)^r\, \bs^{-2}\, v_1\, 
\bs^{-2}\, v_2 \, \dots\, \bs^{-2}\, v_n) \big)=
$$ 
$$
 \frac{1}{r!} \, \ga (\underbrace{\om, \dots, \om}_{r \textrm{ times}}, v_1, \dots, v_n )
$$
where $r = N - n$ and the action of $\ga$ on local sections of $\FR$ is 
obtained via extending  \eqref{Gra-acts} by linearity over $\Omb(\cO_X^{\coord})$\,.

Since for all  $\mv \in \mgl_d(\bbK)$ we have $i_{\bv} v_j = 0$, therefore  
$$
i_{\bv} \big( \ga (\underbrace{\om, \dots, \om}_{r \textrm{ times}}, v_1, \dots, v_n )
\big)  =  
\sum_{k=0}^{r-1} \ga (\underbrace{\om, \dots, \om}_{k \textrm{ times}},
i_{\bv} \om,  \underbrace{\om, \dots, \om}_{r-k-1 \textrm{ times}},
v_1, \dots, v_n )\,.
$$

On the other hand, by Corollary \ref{cor:gl-d-om}, 
$$
i_{\bv}\, \om = - \mv^a_b t^b \frac{\pa}{\pa t^a}\,.
$$
Hence equation \eqref{i-bv-v} holds simply because all 
vertices of $\G$ have valency $\ge 3$\,.

To prove \eqref{LLL-bv-v}, we recall that $(\Te')^{\om}$ is a chain map 
from the graph complex $\GC$ to the cochain complex 
$\coDer \big( \La^2 \coCom (\FR') \big)$ with
the differential 
$$
[d + \{\om, ~\}_{SN} + Q, ~]\,. 
$$

Therefore\footnote{In computation \eqref{playing}, we put $\pm$ 
in front of terms for which sign factors do not play an important role.} 
\begin{equation}	
\label{playing}
\begin{aligned}
 &(d  + \{\om, ~ \}_{SN}) \circ p \circ  (\Te')^{\om}(\ga)\big( \bs^2 (\bs^{-2}\, v_1\, 
\bs^{-2}\, v_2 \, \dots\, \bs^{-2}\, v_n) \big)
\\&=
p \circ \Big( (d  + \{\om, ~ \}_{SN} ) \circ  (\Te')^{\om}(\ga)\big( \bs^2 (\bs^{-2}\, v_1\, 
\bs^{-2}\, v_2 \, \dots\, \bs^{-2}\, v_n) \big)  \Big) 
\\&=
\begin{aligned}[t]
&p \circ \Big( (d  + \{\om, ~ \}_{SN} + Q) \circ  (\Te')^{\om}(\ga)\big( \bs^2 (\bs^{-2}\, v_1\, 
\bs^{-2}\, v_2 \, \dots\, \bs^{-2}\, v_n) \big)  \Big) 
\\&
- (-1)^{|\ga|} p \circ \Big(
(\Te')^{\om}(\ga) \circ  (d  + \{\om, ~ \}_{SN} + Q) 
\big( \bs^2 (\bs^{-2}\, v_1\,  \bs^{-2}\, v_2 \, \dots\, \bs^{-2}\, v_n) \big)  \Big) 
\\&
+ (-1)^{|\ga|} p \circ \Big(
(\Te')^{\om}(\ga) \circ  (d  + \{\om, ~ \}_{SN} + Q) 
\big( \bs^2 (\bs^{-2}\, v_1\,  \bs^{-2}\, v_2 \, \dots\, \bs^{-2}\, v_n) \big)  \Big) 
\\&
- p \circ \Big( Q \circ  (\Te')^{\om}(\ga)\big( \bs^2 (\bs^{-2}\, v_1\, 
\bs^{-2}\, v_2 \, \dots\, \bs^{-2}\, v_n) \big)  \Big) 
\end{aligned}
\\&= 
\begin{aligned}[t]
&p \circ  (\Te')^{\om}(\pa \ga) \big( \bs^2 (\bs^{-2}\, v_1\,  \bs^{-2}\, v_2 \, 
\dots\, \bs^{-2}\, v_n) \big) 
\\&
+ \sum_{k=1}^{n} \pm \frac{1}{r!}
 \ga ( \underbrace{\om, \dots, \om}_{r \textrm{ times}}, v_1, 
\dots, v_{k-1},  d v_k  + \{\om, v_k \}_{SN} , v_{k+1}, \dots, v_n) 
\\&
+ \sum_{1 \le j < k \le n}  \pm  \frac{1}{(r+1)!}
\ga(\underbrace{\om, \dots, \om}_{(r+1) \textrm{ times}}, \{v_j, v_k\}_{SN}, 
v_1 \dots, v_{j-1}, v_{j+1}, \dots, v_{k-1}, v_{k+1}, \dots, v_n)
\\&
+ \sum_{k=1}^n (-1)^{\ve_k}  \pm  \frac{1}{(r+1)!}
\{ \ga(\underbrace{\om, \dots, \om}_{(r+1) \textrm{ times}}, v_1, \dots, v_{k-1}, v_{k+1}, 
\dots, v_n),  v_k \}_{SN}\,.
\end{aligned}
\end{aligned}
\end{equation}

Hence we get 
\begin{equation}
\label{upshot}
\begin{aligned}
& (d  + \{\om, ~ \}_{SN}) \circ p \circ  (\Te')^{\om}(\ga)\big( \bs^2 (\bs^{-2}\, v_1\, 
\bs^{-2}\, v_2 \, \dots\, \bs^{-2}\, v_n) \big) 
\\&=
\begin{aligned}[t]
&\frac{1}{r!} (\pa \ga)(\underbrace{\om, \dots, \om}_{r \textrm{ times}}, v_1, \dots, v_n) 
\\&
+ \sum_{k=1}^{n} \pm  \frac{1}{r!}
 \ga(\underbrace{\om, \dots, \om}_{r \textrm{ times}}, v_1, 
\dots, v_{k-1},  d v_k  + \{\om, v_k \}_{SN} , v_{k+1}, \dots, v_n) 
\\&
+ \sum_{1 \le j < k \le n}  \pm  \frac{1}{(r+1)!}
\ga(\underbrace{\om, \dots, \om}_{(r+1) \textrm{ times}}, \{v_j, v_k\}_{SN}, 
v_1 \dots, v_{j-1}, v_{j+1}, \dots, v_{k-1}, v_{k+1}, \dots, v_n)
\\&
+ \sum_{k=1}^n (-1)^{\ve_k}  \pm  \frac{1}{(r+1)!}
\{ \ga(\underbrace{\om, \dots, \om}_{(r+1) \textrm{ times}}, v_1, \dots, v_{k-1}, v_{k+1}, 
\dots, v_n), v_k \}_{SN}\,.
\end{aligned}\end{aligned}
\end{equation}

Thus we see that \eqref{LLL-bv-v} follows from equations \eqref{i-bv-v}, 
$$
i_{\bv}  \big( d v_k  + \{\om, v_k \}_{SN}  \big) = 0
$$
and the fact that $i_{\bv}$ is a derivation of the bracket $\{~,~\}_{SN}$.

Proposition \ref{prop:Te-pr-tw} is proven. 
\end{proof}

Composing the map of dg Lie algebras \eqref{Te-dfn-almost} with the
canonical morphism (see Appendix \ref{app:Canon-homo})
$$ 
\coDer \big( \La^2 \coCom (\FR) \big) \to \Def_{\La\Lie}(\FR),  
$$
we get the desired map of dg Lie algebras 
\begin{equation}
\label{GC-Def-T-poly-here}
\Te : \GC \to \Def_{\La\Lie}(\FR)\,. 
\end{equation}

\subsection{Extending $\Te$ to $\Conv(\Ger^{\vee}, \Gra)_{\ge 3}$}
\label{sec:aux}

Although $\FR$ is also a sheaf of Gerstenhaber algebras,
there is no natural way of extending 
the map $\Te$ \eqref{GC-Def-T-poly-here} to
a map
\begin{equation}
\label{not-possible}
\Conv(\Ger^{\vee}, \Gra) \to \Def_{\Ger}(\FR)
\end{equation}
from the dg Lie algebra $\Conv(\Ger^{\vee}, \Gra)$ to the deformation 
complex $\Def_{\Ger}(\FR)$ of $\FR$ (viewed as the sheaf of Gerstenhaber algebras). 

However, it is possible to extend  $\Te$ \eqref{GC-Def-T-poly-here} 
to a map from a dg Lie subalgebra $\Conv(\Ger^{\vee}, \Gra)_{\ge 3}$ \eqref{ge-3} to
$\Def_{\Ger}(\FR)$\,. 

To construct this map, we extend $\ma$ \eqref{Gra-to-End-Tpoly} by linearity 
over  $\Omb(\cO_X^{\coord})$ to 
\begin{equation}
\label{Gra-End-FRpr}
\Gra(n)  \to \Hom\big( (\FR')^{\otimes\, n} , \FR'\big)\,,
\end{equation}
where $\FR'$ is the auxiliary sheaf of dg  Gerstenhaber algebras defined 
in \eqref{FR-pr}\,.

Next, using \eqref{Gra-End-FRpr}, we get an auxiliary map of dg Lie algebras 
\begin{equation}
\label{Te-pr-Ger}
\Te'_{\Ger} : \Conv(\Ger^{\vee}, \Gra) \to \coDer \big( \Ger^{\vee} (\FR') \big)\,,
\end{equation}
where the codomain 
$$
\coDer \big( \Ger^{\vee} (\FR') \big)
$$
carries the differential 
\begin{equation}
\label{diff-d-Q-Ger}
[d+ Q_{\Ger}, ~]\,,
\end{equation}
with $d$ being the de Rham differential and $Q_{\Ger}$ being the 
coderivation of $ \Ger^{\vee} (\FR')$ coming from the  Gerstenhaber algebra 
structure on $\FR'$. 

We now recall that the sheaf of  Gerstenhaber algebras $\FR$ is obtained from 
$\FR'$ \eqref{FR-pr} in two steps. First, we need to twist $\FR'$ by the Maurer-Cartan 
element $\om$ defined in Theorem \ref{thm:om}. Second, we apply trimming with 
respect to the derivations coming from the action of $\mgl_d(\bbK)$\,.   

We have the following analog of Proposition \ref{prop:Te-pr-tw}:
\begin{prop}
\label{prop:Te-pr-tw-Ger}
Let  $\om = \om^a \pa_{t^a}$ be the global section of the sheaf
$$
 \Om^1(\cO_X^{\coord})  \otimes  T^{1,0}(P)
$$
introduced in Theorem \ref{thm:om} and $\Conv(\Ger^{\vee}, \Gra)_{\ge 3}$ be
the dg Lie subalgebra of $\Conv(\Ger^{\vee}, \Gra)$ introduced in \eqref{ge-3}.
Then the formula  
\begin{equation}
\label{Te-pr-tw-Ger-dfn}
(\Te'_{\Ger})^{\om}(\ga)  = e^{-\bs^{-2}\, \om}\, \Te'_{\Ger}(\ga)\, e^{\bs^{-2}\, \om}\,, 
\qquad \ga \in  \Conv(\Ger^{\vee}, \Gra)_{\ge 3}
\end{equation}
defines a map of dg Lie algebras 
\begin{equation}
\label{Te-pr-tw-Ger}
(\Te')^{\om}_{\Ger} :  \Conv(\Ger^{\vee}, \Gra)_{\ge 3} \to \coDer \big( \Ger^{\vee} (\FR') \big)\,,
\end{equation}
where $ \coDer \big( \Ger^{\vee} (\FR') \big)$ is considered 
with the differential 
\begin{equation}
\label{d-om-Q-Ger}
[d +  \{\om, ~\}_{SN} + Q_{\Ger}, ~]\,. 
\end{equation}
Furthermore, $(\Te')^{\om}_{\Ger}$ descends to a map of dg Lie algebras
\begin{equation}
\label{Te-Ger-dfn-almost}
 \Conv(\Ger^{\vee}, \Gra)_{\ge 3}  ~ \to ~   \coDer \big( \Ger^{\vee} (\FR) \big)\,,
\end{equation}
where 
$$
\FR = \Big( \Omb(\cO_X^{\coord})  \otimes  T_{\poly}(P) \Big)^{[\mgl_d(\bbK)]}\,.
$$
\end{prop}
\begin{proof}
Just as in the proof of Proposition \ref{prop:Te-pr-tw}, 
Claim \ref{cl:graphs-ge-3} implies that for every vector 
$\ga \in  \Conv(\Ger^{\vee}, \Gra)_{\ge 3}$, we have 
\begin{equation}
\label{Te-pr-Ger-si-omomom}
p \circ \Te'_{\Ger}(\ga) \big( \bs^2 (\bs^{-2}\, \om)^n \big) = 0\,,
\end{equation}
where $p$ is the canonical projection 
\begin{equation}
\label{p-FR-pr-Ger}
p  :  \Ger^{\vee} (\FR')  \to  \FR' \,,
\end{equation}
and $ \bs^2 (\bs^{-2}\, \om)^n$ is a global section of 
$\La^{2}\coCom(\FR') \subset \Ger^{\vee}(\FR')$\,.

Hence Theorem \ref{thm:twisting-Ger} from Appendix \ref{app:tw-Ger}
implies that the assignment 
\begin{equation}
\label{ga-to-coder}
\ga ~\mapsto~  e^{-\bs^{-2}\, \om}\, \Te'_{\Ger}(\ga)\, e^{\bs^{-2}\, \om}
\end{equation}
is a map of dg Lie algebras from $ \Conv(\Ger^{\vee}, \Gra)_{\ge 3} $ to 
$ \coDer \big( \Ger^{\vee} (\FR') \big)$, where the codomain is considered with the 
differential \eqref{d-om-Q-Ger}.

It remains to prove that for every cochain $\ga \in  \Conv(\Ger^{\vee}, \Gra)_{\ge 3}$ 
and any local section $W \in \Ger^{\vee}(\FR)$ the 
section of $\FR'$ 
\begin{equation}
\label{Te-pr-Ger-om-gagaga}
v = p \circ (\Te'_{\Ger})^{\om}(\ga) \big( W \big)  
\end{equation}
satisfies the conditions 
\begin{equation}
\label{i-bv-v-Ger}
i_{\bv} (v) = 0
\end{equation}
and 
\begin{equation}
\label{LLL-bv-v-Ger}
i_{\bv} (d v + \{\om, v \}_{SN}) = 0 \qquad \forall ~~ \mv \in \mgl_d(\bbK).
\end{equation}

The latter can be shown by going through the corresponding steps {\it mutatis mutandis} 
in the proof of Proposition \ref{prop:Te-pr-tw}.
\end{proof}

\section{For every cocycle $\ga \in \GC$ the cocycle $\Te(\ga)$ induces a derivation of the Gerstenhaber algebra 
$H^{\bul}(X, \cT_{\poly} )$}
\label{sec:cD-ga-deriv}

Let us recall (see Appendix \ref{app:O-deriv-sheaves}) that for every degree $k$ cocycle $\ga \in \GC$
the cocycle $\Te(\ga) \in  \Def_{\La\Lie}(\FR)$ induces a degree $k$ derivation of the $\La\Lie$-algebra 
\begin{equation}
\label{H-Tpoly}
H^{\bul}(X, \cT_{\poly})\,.
\end{equation}
We will denote this derivation by $\sfD_{\ga}$\,.

More precisely, if $v$ is a cocycle in 
\begin{equation}
\label{Cech-FR}
\cCb(X, \FR)
\end{equation}
representing a class $[v]$ in \eqref{H-Tpoly} then $\sfD_{\ga}([v])$ 
is represented by 
\begin{equation}
\label{sfD-ga-dfn}
\sum_{n \ge 1} \frac{1}{n!} \ga (\, \underbrace{\om, \dots, \om}_{n \textrm{ times}},  v \,) 
\end{equation}
where the action of $\ga$ on local sections of $\FR$ is 
obtained via extending  \eqref{Gra-acts} by linearity over $\Omb(\cO_X^{\coord})$
and  $\om$ is the global section of the sheaf
$\Om^1(\cO_X^{\coord}) \otimes  T^{1,0}(P)$
introduced in Theorem \ref{thm:om}. 

Our goal is to prove that 
\begin{thm}
\label{thm:mB-ga-Ger-deriv}
For every cocycle $\ga \in \GC$, the map 
\begin{equation}
\label{sfD-ga}
\sfD_{\ga} : H^{\bul}(X, \cT_{\poly}) \to H^{\bul}(X, \cT_{\poly})
\end{equation}
defined by \eqref{sfD-ga-dfn} is a derivation of the Gerstenhaber 
algebra structure on $H^{\bul}(X, \cT_{\poly})$
\end{thm}
A proof of this theorem is given in Section \ref{sec:proof} below. 
It is based on a technical claim which we present now.

First, we recall that the dg Lie algebras $\fGC$, $\GC$, $\Conv(\Ger^{\vee}, \Ger)$, 
and  $\Conv(\Ger^{\vee}, \Gra)$  carry a natural decreasing filtration ``by arity'':
 \begin{equation}
\label{cF-m-fGC}
\cF_m \fGC  :=  \prod_{n \ge m+1} 
\bs^{2n-2} \Big( \Gra(n) \Big)^{S_n}\,,
\qquad 
\cF_m \GC = \GC ~ \cap ~ \cF_m \fGC\,,
\end{equation}
\begin{equation}
\label{filtr-Conv-Ger-Ger}
\cF_m \Conv(\Ger^{\vee}, \Ger)  :=
 \prod_{n \ge m+1} \Big( \Ger(n) \otimes \La^{-2} \Ger(n)  \Big)^{S_n}\,,
\end{equation}
and
\begin{equation}
\label{filtr-Conv-Ger-Gra}
\cF_m \Conv(\Ger^{\vee}, \Gra)  : = \prod_{n \ge m+1} \Big( \Gra(n) \otimes \La^{-2} \Ger(n)  \Big)^{S_n}\,.
\end{equation}

Second, we observe that every cochain $\ga \in \fGC$ may be extended 
``by zero'' to a cochain $\wt{\ga}$ in \eqref{Conv-Ger-Gra}. 
Indeed, the desired element $\wt{\ga}$ is 
defined by declaring that it vanishes on all 
vectors in $\Ger^{\vee}$ which involve at least one $\La\coLie$-monomial of length $\ge 2$, 
and setting
\begin{equation}
\label{wt-ga-ga}
\wt{\ga} \Big|_{\La^2\coCom} = \ga\,.
\end{equation}
It is obvious that for every cochain $\ga \in \GC$ we have 
$$
\wt{\ga} \in \Conv(\Ger^{\vee}, \Gra)_{\ge 3}\,.
$$

Finally, we formulate the technical statement which is 
used in the proof of Theorem \ref{thm:mB-ga-Ger-deriv}:
\begin{prop}
\label{prop:tech}
Let 
$$
\mR : \Conv(\Ger^{\vee}, \Gra)_{\ge 3} \to \fGC
$$
be the map of cochain complexes defined in \eqref{mR} and 
$\ga$ be a degree $q$ cocycle in $\cF_m \GC$\,.   
There exists a degree $q$ cochain $ \te \in \ker(\mR)$ and a degree $q+1$ cocycle
$ x \in \Xi_{\conn}  \cap \cF_{m+1} \, \Conv(\Ger^{\vee}, \Ger) $
such that 
\begin{equation}
\label{wt-ga-x-te}
\pa  \wt{\ga}  = \psi(x) + \pa \te\,,
\end{equation}
where the vector $\wt{\ga} \in \Conv(\Ger^{\vee}, \Gra)_{\ge 3}$ is obtained 
via extending $\ga$ ``by zero'' and $\psi$ is the embedding  \eqref{Xi-ker-mR}.
\end{prop}
\begin{proof} 
Since $\ga \in  \cF_m \GC$,  
\begin{equation}
\label{wtga-in-cF-m}
\wt{\ga} \in \cF_m  \Conv(\Ger^{\vee}, \Gra)_{\ge 3}\,.
\end{equation}
Furthermore, $\ga$ is a cocycle in $\GC$. Hence 
\begin{equation}
\label{diff-wtga-ker-mR}
\pa \, \wt{\ga} \in \ker(\mR)\,.
\end{equation}

Therefore, by Proposition \ref{prop:Xi-ker-mR}, there exists 
a degree $q$ cochain 
$ \te' \in \ker(\mR)$ and a degree $(q+1)$  cocycle
$x' \in \Xi_{\conn}  \cap \Conv(\Ger^{\vee}, \Ger) $ such that
\begin{equation}
\label{wt-ga-x-pr-te}
\pa \wt{\ga}  = \psi(x') + \pa \te'\,.
\end{equation}

Let us now observe that 
\begin{equation}
\label{diff-cF-m}
\pa \Big( \cF_m  \Conv(\Ger^{\vee}, \Gra) \Big) \subset  \cF_{m+1}  \Conv(\Ger^{\vee}, \Gra)\,.
\end{equation}
Hence, using inclusion \eqref{wtga-in-cF-m}, we deduce that
the restriction of $\psi(x')$ to $\Ger^{\vee}(n)$ for $n \le m+1$ gives 
us an exact cocycle in 
\begin{equation}
\label{quotient}
\prod_{n=2}^{m+1} \Big( \Gra(n) \otimes \La^{-2}\Ger(n) \Big)^{S_n}
\cap \ker(\mR)
\end{equation}

Therefore, applying  Proposition \ref{prop:Xi-ker-mR} again, we 
conclude that there exists a degree $q$ cochain 
$\te'' \in \ker(\mR)$ and a degree $q+1$ cocycle
$ x \in \Xi_{\conn}  \cap \cF_{m+1} \, \Conv(\Ger^{\vee}, \Ger) $
such that 
$$
\psi (x) = \psi(x') - \pa(\te'')\,.  
$$

Thus setting $\te = \te' + \te''$ we get the desired equation 
\eqref{wt-ga-x-te}.
\end{proof}

\subsection{Proof of Theorem \ref{thm:mB-ga-Ger-deriv}}
\label{sec:proof}

The map \eqref{sfD-ga} is a derivation of the $\La\Lie$-bracket on $ H^{\bul}(X, \cT_{\poly})$ 
since $\sfD_{\ga}$  comes from a cocycle in  $\Def_{\La\Lie}(\FR)$. Thus it remains to prove that 
$\sfD_{\ga}$ is a derivation for the commutative algebra structure on  $H^{\bul}(X, \cT_{\poly})$\,.

For this purpose we start with two cocycles in the \v{C}ech complex 
$$
v^1,v^2 \in \cCb(X, \FR)
$$
and consider the class
$$
\sfD_{\ga}([v^1] \cup [v^2]) - \sfD_{\ga}([v^1]) \cup [v^2] -
(-1)^{|v^1| |\ga|} [v^1] \cup \sfD_{\ga}([v^2]) ~\in~ H^{\bul}(X, \FR)
$$
where $\cup$ is the multiplication on $H^{\bul}(X, \FR)$ induced 
by the $\wedge$-product on $\FR$. 

This class is represented by the cocycle $v$ given 
by the equation
\begin{equation}
\label{v-Cech-dfn}
v_{\al_0 \dots \al_t} : = 
\begin{aligned}[t]
&\sum_{0 \le k \le t}
\sum_{n \ge 1} \frac{(-1)^{k |v^2|}}{n!} \ga (\, \underbrace{\om, \dots, \om}_{n \textrm{ times}},  
v^1_{\al_0 \dots \al_k} v^2_{\al_k \dots \al_t} \,) 
\\&
- \sum_{0 \le k \le t}
\sum_{n \ge 1} 
 \frac{(-1)^{k |v^2|}}{n!} \ga (\, \underbrace{\om, \dots, \om}_{n \textrm{ times}},  
v^1_{\al_0 \dots \al_k} ) v^2_{\al_k \dots \al_t} 
\\&
- (-1)^{|\ga| |v^1|} \sum_{0 \le k \le t}
\sum_{n \ge 1} \frac{(-1)^{k (|v^2| + |\ga|)}}{n!}
 v^1_{\al_0 \dots \al_k}
 \ga (\, \underbrace{\om, \dots, \om}_{n \textrm{ times}},  
 v^2_{\al_k \dots \al_t}  )\,. 
 \end{aligned}
\end{equation}
So, our goal is to show that the cocycle $v$ is exact. 

To prove this claim, we rewrite \eqref{v-Cech-dfn}  
using the fact that the element 
$$
\wt{\ga} \in \Conv(\Ger^{\vee}, \Gra)_{\ge 3}
$$
is obtained via extending $\ga \in \GC$ ``by zero'':
\begin{equation}
\label{v-Cech-new}
\begin{aligned}
v_{\al_0 \dots \al_t} &=
\begin{aligned}[t]
&\sum_{0 \le k \le t} (-1)^{k |v^2|}
p \circ  e^{-\bs^{-2} \om} \Te'_{\Ger}(\wt{\ga}) e^{\bs^{-2} \om} 
(v^1_{\al_0 \dots \al_k}  v^2_{\al_k \dots \al_t})
\\&
- \sum_{0 \le k \le t} (-1)^{k |v^2|}
\big( p \circ  e^{-\bs^{-2} \om} \Te'_{\Ger}(\wt{\ga}) e^{\bs^{-2} \om} 
(v^1_{\al_0 \dots \al_k}) \big)\, v^2_{\al_k \dots \al_t}
\\&
-(-1)^{|\wt{\ga}| |v^1|} \sum_{0 \le k \le t} (-1)^{k (|v^2|+|\wt{\ga}|)}
v^1_{\al_0 \dots \al_k} \,
\big( p \circ  e^{-\bs^{-2} \om} \Te'_{\Ger}(\wt{\ga}) e^{\bs^{-2} \om} 
( v^2_{\al_k \dots \al_t}) \big)
\end{aligned}
\\&=
\begin{aligned}[t]
&\sum_{0 \le k \le t} (-1)^{k |v^2|+ (|v^1|-k-1)}
p \circ  e^{-\bs^{-2} \om} \Te'_{\Ger}(\wt{\ga}) e^{\bs^{-2} \om} \circ 
\msQ^{\om}_{\Ger} \big( \L v^1_{\al_0 \dots \al_k},  v^2_{\al_k \dots \al_t} \R \big)
\\&
- \sum_{0 \le k \le t} (-1)^{k |v^2|+ (|v^1|-k-1)}
p \circ  e^{-\bs^{-2} \om} \Te'_{\Ger}(\wt{\ga}) e^{\bs^{-2} \om} 
 \big( \L d v^1_{\al_0 \dots \al_k} + \{\om, v^1_{\al_0 \dots \al_k}\}_{SN},  v^2_{\al_k \dots \al_t} \R \big)
\\&
- \sum_{0 \le k \le t} (-1)^{k (|v^2|+ 1)}
p \circ  e^{-\bs^{-2} \om} \Te'_{\Ger}(\wt{\ga}) e^{\bs^{-2} \om} 
 \big( \L v^1_{\al_0 \dots \al_k}  ,  d v^2_{\al_k \dots \al_m} + \{\om, v^2_{\al_k \dots \al_t}\}_{SN} \R \big)
\\&
+ \sum_{0 \le k \le t} (-1)^{k |v^2| + |v^1|-k + |\wt{\ga}|}
p \circ \msQ^{\om}_{\Ger} \circ  e^{-\bs^{-2} \om} \Te'_{\Ger}(\wt{\ga}) e^{\bs^{-2} \om} 
\L v^1_{\al_0 \dots \al_k} , v^2_{\al_k \dots \al_t} \R 
\end{aligned}
\\
&=
\begin{aligned}[t]
&- \sum_{0 \le k \le t} (-1)^{k |v^2|+ (|v^1|-k)}
p \circ \Big( e^{-\bs^{-2} \om} \Te'_{\Ger}(\wt{\ga}) e^{\bs^{-2} \om} 
\msQ^{\om}_{\Ger}  
\\&
- (-1)^{|\wt{\ga}|} \msQ^{\om}_{\Ger} 
\circ
 e^{-\bs^{-2} \om} \Te'_{\Ger}(\wt{\ga}) e^{\bs^{-2} \om} \Big)
\L v^1_{\al_0 \dots \al_k},  v^2_{\al_k \dots \al_t} \R \,,
\end{aligned}\end{aligned}
\end{equation}
where
$$
\L s_1, s_2 \R : = \{b_1,b_2\}^* \otimes s_1 \otimes s_2 \in \Ger^{\vee} (\FR) 
$$
for two local sections $s_1, s_2$ of $\FR$ and $\msQ^{\om}_{\Ger}$ is 
the codifferential of $\Ger^{\vee} (\FR)$ given by 
\begin{equation}
\label{msQ-om-Ger}
\msQ^{\om}_{\Ger} : =  e^{-\bs^{-2} \om} \circ (d+Q_{\Ger}) \circ  e^{\bs^{-2} \om}  = d + \{\om, ~\}_{SN}  + Q_{\Ger}\,.
\end{equation}

Thus we get 
\begin{multline}
\label{v-Cech-more}
v_{\al_0 \dots \al_t} =
\\
(-1)^{|\ga|} \sum_{0 \le k \le t} (-1)^{k |v^2|+ (|v^1|-k)}
p \circ  e^{-\bs^{-2} \om} \Te'_{\Ger}(\pa \wt{\ga}) e^{\bs^{-2} \om} 
\big( \L v^1_{\al_0 \dots \al_k},  v^2_{\al_k \dots \al_t} \R \big)\,. 
\end{multline}

Hence, due to Proposition \ref{prop:tech}, 
\begin{equation}
\label{v-Cech-x-te}
v_{\al_0 \dots \al_t} =
\begin{aligned}[t]
&
(-1)^{|\ga|} \sum_{0 \le k \le t} (-1)^{k |v^2|+ (|v^1|-k)}
p \circ  e^{-\bs^{-2} \om} \Te'_{\Ger}(\psi(x)) e^{\bs^{-2} \om} 
\big( \L v^1_{\al_0 \dots \al_k},  v^2_{\al_k \dots \al_t} \R \big) 
\\&
+(-1)^{|\ga|} \sum_{0 \le k \le t} (-1)^{k |v^2|+ (|v^1|-k)}
p \circ  e^{-\bs^{-2} \om} \Te'_{\Ger}(\pa \te) e^{\bs^{-2} \om} 
\big( \L v^1_{\al_0 \dots \al_k},  v^2_{\al_k \dots \al_t} \R \big)\,,
\end{aligned}\end{equation}
where $ x \in \Xi_{\conn}  \cap \cF_{m+1} \, \Conv(\Ger^{\vee}, \Ger) $
and $\te \in \ker(\mR)$\,.

Using the inclusion $x \in \Xi_{\conn}$ and the compatibility of $\Te'_{\Ger}$ with 
the differentials we conclude that
 \begin{equation}
\label{v-Cech-x-te1}
v_{\al_0 \dots \al_t} =
\begin{aligned}[t]
&
(-1)^{|\ga|} \sum_{0 \le k \le t} (-1)^{k |v^2|+ (|v^1|-k)}
p \circ  \Te'_{\Ger}(\psi(x))
\big( \L v^1_{\al_0 \dots \al_k},  v^2_{\al_k \dots \al_t} \R \big) 
\\&
+(-1)^{|\ga|} \sum_{0 \le k \le t} (-1)^{k |v^2|+ (|v^1|-k)}
p \circ  \msQ^{\om}_{\Ger} \circ e^{-\bs^{-2} \om} \Te'_{\Ger}(\te) e^{\bs^{-2} \om} 
\big( \L v^1_{\al_0 \dots \al_k},  v^2_{\al_k \dots \al_t} \R \big)
\\&
+ \sum_{0 \le k \le t} (-1)^{k |v^2|+ (|v^1|-k)} 
p \circ
e^{-\bs^{-2} \om} \Te'_{\Ger}(\te) e^{\bs^{-2} \om} \circ \hat{\msQ}^{\om}_{\Ger} 
\big( \L v^1_{\al_0 \dots \al_k},  v^2_{\al_k \dots \al_t} \R \big)
\end{aligned}\end{equation}

Since $\GC = \cF_2 \GC$ we have  $x \in \Xi_{\conn}  \cap \cF_{3} \, \Conv(\Ger^{\vee}, \Ger) $
and hence 
$$
  \Te'_{\Ger}(\psi(x))
\big( \L v^1_{\al_0 \dots \al_k},  v^2_{\al_k \dots \al_t} \R \big) =0\,.
$$

In addition, $\te \in \ker(\mR)$. Therefore,

\begin{equation}
\label{v-Cech-x-te11}
v_{\al_0 \dots \al_t} =
\begin{aligned}[t]
&
(-1)^{|\ga|} \sum_{0 \le k \le t} (-1)^{k |v^2|+ (|v^1|-k)}
(d + \{\om,~\}_{SN})
\Big( p \circ \Te'_{\Ger}(\te) e^{\bs^{-2} \om} 
\big( \L v^1_{\al_0 \dots \al_k},  v^2_{\al_k \dots \al_t} \R \big) \Big)
\\&
- \sum_{0 \le k \le t} (-1)^{k |v^2|+ (|v^1|-k)} 
p \circ \Te'_{\Ger}(\te) e^{\bs^{-2} \om} \circ 
\big( \L  (d v^1_{\al_0 \dots \al_k} + \{\om, v^1_{\al_0 \dots \al_k}\}_{SN}),  
v^2_{\al_k \dots \al_t} \R \big)
\\&
- \sum_{0 \le k \le t} (-1)^{k |v^2|} 
p \circ \Te'_{\Ger}(\te) e^{\bs^{-2} \om} \circ 
\big( \L  v^1_{\al_0 \dots \al_k} ,
(d v^2_{\al_k \dots \al_t} + \{\om, v^2_{\al_k \dots \al_t}\}_{SN}) \R \big)\,.\end{aligned}
\end{equation}

Since both $v^1$ and $v^2$ are cocycles in the \v{C}ech complex $\cCb(X, \FR)$, 
we have 
$$
d v^1_{\al_0 \dots \al_k} + \{\om, v^1_{\al_0 \dots \al_k}\}_{SN} = 
- (\cpa v^1)_{\al_0 \dots \al_k} 
$$
and 
$$
d v^2_{\al_k \dots \al_t} + \{\om, v^2_{\al_k \dots \al_t}\}_{SN} = 
- (\cpa v^2)_{\al_k \dots \al_t}\,. 
$$

~

Hence, equation \eqref{v-Cech-x-te11} implies that

\begin{equation}
\label{v-Cech-exactly}
\begin{aligned}
v_{\al_0 \dots \al_t} &=
\begin{aligned}[t]
&
(-1)^{|\ga|} \sum_{0 \le k \le t} (-1)^{k |v^2|+ (|v^1|-k)}
(d + \{\om,~\}_{SN})
\Big( p \circ \Te'_{\Ger}(\te) e^{\bs^{-2} \om} 
\big( \L v^1_{\al_0 \dots \al_k},  v^2_{\al_k \dots \al_t} \R \big) \Big)
\\&+
\sum_{0 \le k \le t} (-1)^{k |v^2|+ (|v^1|-k)} 
p \circ \Te'_{\Ger}(\te) e^{\bs^{-2} \om} \circ 
\big( \L  (\cpa v^1)_{\al_0 \dots \al_k} ,  
v^2_{\al_k \dots \al_t} \R \big)
\\&
+ \sum_{0 \le k \le t} (-1)^{k |v^2|} 
p \circ \Te'_{\Ger}(\te) e^{\bs^{-2} \om} \circ 
\big( \L  v^1_{\al_0 \dots \al_k} ,
(\cpa v^2)_{\al_k \dots \al_t} \R \big) 
\end{aligned}
\\&=
d u_{\al_0 \dots \al_t} + \{\om, u_{\al_0 \dots \al_t} \}_{SN} + 
(\cpa u)_{\al_0 \dots \al_t}\,,
\end{aligned}\end{equation}
where
\begin{equation}
\label{u-dfn}
 u_{\al_0 \dots \al_t} : =
(-1)^{|\ga|} \sum_{0 \le k \le t} (-1)^{k |v^2|+ (|v^1|-k)}
p \circ \Te'_{\Ger}(\te)  \big( e^{\bs^{-2} \om} 
 \L v^1_{\al_0 \dots \al_k},  v^2_{\al_k \dots \al_t} \R \big) \,.
\end{equation}

Thus the cocycle $v$ is indeed exact and 
Theorem \ref{thm:mB-ga-Ger-deriv} follows. 

\qed

\section{The action of Deligne-Drinfeld elements}
\label{sec:DD-action}

Let $\grt$ be the Grothendieck-Teichm\"uller  Lie algebra and let  $\si_n$ ($n$ is odd $\ge 3$) be 
Deligne-Drinfeld elements of $\grt$ (see Proposition \ref{prop:DD-elements} in Subsection \ref{sec:grt}). 
Let us denote by $\ga_n$ any cocycle which represents 
the cohomology class $\wt{\si}_n \in H^0(\GC)$ corresponding to $\si_n \in \grt$
via the isomorphism \eqref{H0GC-grt} in Theorem \ref{thm:GC-grt}.

According to Theorem \ref{thm:mB-ga-Ger-deriv}, $\Te(\ga_n)$ induces 
a degree zero derivation $\cD_n$ of the Gerstenhaber algebra 
\begin{equation}
\label{H-X-Tpoly}
H^{\bul}(X, \cT_{\poly})\,.
\end{equation}

The following theorem gives a natural geometric interpretation of 
this derivation:
\begin{thm}
\label{thm:main}
Let $n$ be an odd integer $\ge 3$\,. Then the action of $\cD_n$ on $H^{\bul}(X, \cT_{\poly})$ is a non-zero scalar multiple of 
the contraction with the $n$-th component of the Chern character\footnote{We consider the Chern character with values 
in $\bigoplus_{k} H^{k}(X, \Om^k_X)$\,.} 
of the tangent bundle of $X$, 
regardless of the choice of Deligne-Drinfeld elements, i.~e., of Lie words $\dots$ in \eqref{si-n-fix}. 
\end{thm}
Theorem \ref{thm:main} has the following obvious corollary: 
\begin{cor}
\label{cor:main}
Let $X$ be a smooth algebraic variety over $\bbK$ and $n$ be an 
odd integer $\ge 3$.
Then the contraction with the $n$-th component of the Chern character induces 
a derivation of the Gerstenhaber algebra $H^{\bul}(X, \cT_{\poly})$. \qed
\end{cor}
\begin{remark}
\label{rem:K-motives}
Corollary \ref{cor:main} implies that the contractions with odd components of 
the Chern character are derivations of the cup product on  $H^{\bul}(X, \cT_{\poly})$. 
This statement was formulated without a proof in 
\cite[Theorem 9]{K-motives}. 
\end{remark}

\begin{remark}
\label{rem:u-abelianization}
The proof of Theorem \ref{thm:main} is essentially based on the combinatorial Claim 
\ref{cl:wheels-only} given below. This claim implies that the derivation 
$\cD$ of the Gerstenhaber algebra $H^{\bul}(X, \cT_{\poly})$ corresponding 
to a cocycle $\ga \in \GC$ is non-zero if and only if $\ga$ involves 
the graph $\G^{wheel}_n$ shown on figure \ref{fig:wheel} for some 
odd integer $n \ge 3$.  
It is not hard to see that for any pair of vectors 
$\ga, \ga' \in \GC$ the Lie bracket $[\ga, \ga']$ does not involve graphs of 
the form $\G^{wheel}_n$, since $\G^{wheel}_n$ does not have a subgraph 
whose contraction would yield a non-zero vector in $\GC$. This observation implies 
that the action of $\grt$ on $H^{\bul}(X, \cT_{\poly})$ factors through its 
abelianization\footnote{This was conjectured by the anonymous referee.}  
$\grt^{ab} = \grt \big/ [\grt, \grt]$. Furthermore, note that the operations of contraction with the Chern characters pairwise commute, and hence it follows once again that the $\grt$ action is indeed an action.
\end{remark}

We will prove Theorem \ref{thm:main} in Section \ref{sec:proof-main} and 
now we will present a combinatorial fact which is used in the proof of 
this theorem.
\begin{claim}
\label{cl:wheels-only}
Let $\G$ be a connected, 1-vertex irreducible labeled graph with 
$n+1$ vertices, such that each vertex of $\G$ has valency $\ge 3$.
In addition, let $\G^{wheel}_n$ be the labeled graph depicted on figure \ref{fig:wheel}.
If $\G$ admits an orientation for which each vertex with label 
$\le n$ has at most one out-going edge, then
$\G = \si (\G^{wheel}_n)$ for 
some permutation\footnote{The group $S_n$ is tacitly identified with 
the stabilizer of $(n+1)$ in $S_{n+1}$.} $\si \in S_n$\,. Furthermore, $\G^{wheel}_n$
has exactly two orientations which satisfy the above condition. These orientations are shown 
on figures \ref{fig:clock}  and \ref{fig:counterclock}. 
\end{claim}
\begin{figure}[htp]
\centering 
\begin{minipage}[t]{0.45\linewidth} 
\centering 
\begin{tikzpicture}[scale=0.5, >=stealth']
\tikzstyle{ext}=[circle, draw, minimum size=5, inner sep=1]
\tikzstyle{int}=[circle, draw, fill, minimum size=5, inner sep=1]
\node [int] (v0) at (3,3) {};
\draw (3.7, 2.5) node[anchor=center] {{\small $n+1$}};
\node [int] (vn) at (0,3) {};
\draw (-0.5, 3) node[anchor=center] {{\small $n$}};
\node [int] (v1) at (0.5,5) {};
\draw (0.5, 5.6) node[anchor=center] {{\small $1$}};
\node [int] (v2) at (2,6) {};
\draw (2, 6.6) node[anchor=center] {{\small $2$}};
\node [int] (v3) at (4,6) {};
\draw (4, 6.6) node[anchor=center] {{\small $3$}};
\draw (5.4, 5.15) node[anchor=center, rotate=140] {{\small $\dots$}};
\draw (0.5, 1.5) node[anchor=center, rotate=120] {{\small $\dots$}};
\draw  [<-] (vn) edge (0.25,2);
\draw [->] (vn) edge (v1);
\draw [->] (v1) edge (v2);
\draw  [->] (v2) edge (v3);
\draw  (v3) edge (5,5.5);
\draw [<-] (vn) edge (v0);
\draw [<-] (v1) edge (v0);
\draw [<-] (v2) edge (v0);
\draw [<-] (v3) edge (v0);
\end{tikzpicture}
~\\[0.3cm]
\caption{Each vertex with label $\le n$ has at most $1$ out-going edge} \label{fig:clock}
\end{minipage} ~
\begin{minipage}[t]{0.45\linewidth} 
\centering 
\begin{tikzpicture}[scale=0.5, >=stealth']
\tikzstyle{ext}=[circle, draw, minimum size=5, inner sep=1]
\tikzstyle{int}=[circle, draw, fill, minimum size=5, inner sep=1]
\node [int] (v0) at (3,3) {};
\draw (3.7, 2.5) node[anchor=center] {{\small $n+1$}};
\node [int] (vn) at (0,3) {};
\draw (-0.5, 3) node[anchor=center] {{\small $n$}};
\node [int] (v1) at (0.5,5) {};
\draw (0.5, 5.6) node[anchor=center] {{\small $1$}};
\node [int] (v2) at (2,6) {};
\draw (2, 6.6) node[anchor=center] {{\small $2$}};
\node [int] (v3) at (4,6) {};
\draw (4, 6.6) node[anchor=center] {{\small $3$}};
\draw (5.4, 5.15) node[anchor=center, rotate=140] {{\small $\dots$}};
\draw (0.5, 1.5) node[anchor=center, rotate=120] {{\small $\dots$}};
\draw (vn) edge (0.25,2);
\draw  [<-] (vn) edge (v1);
\draw [<-] (v1) edge (v2);
\draw [<-] (v2) edge (v3);
\draw [<-] (v3) edge (5,5.5);
\draw [<-] (vn) edge (v0);
\draw [<-] (v1) edge (v0);
\draw [<-] (v2) edge (v0);
\draw [<-] (v3) edge (v0);
\end{tikzpicture}
~\\[0.3cm]
\caption{Each vertex with label $\le n$ has at most $1$ out-going edge} \label{fig:counterclock}
\end{minipage}
\end{figure} 
\begin{proof} {\it The vertex with label $(n+1)$ has no incoming edges.}

Let us  assume that the vertex of $\G$ with label $(n+1)$ has 
an incoming edge which originates, say, at vertex $i_1$. Since vertex $i_1$ 
is at least trivalent and it has at most one outgoing edge, this vertex has at least two 
incoming edges. One of these edges originates at, say, vertex $i_2$. 
Since $\G$ does not have double edges, $i_2 \le n$ so we may apply the same 
argument to vertex $i_2$ and choose an edge which terminates at vertex $i_2$ 
and originates, say, at vertex $i_3$. 

Continuing this process we will get an oriented path which returns to vertex 
$(n+1)$. Indeed, since the graph is finite and each vertex with label $\le n$ 
cannot have more than one out-going edge, the path must come back to vertex $(n+1)$. 
See figure \ref{fig:cycle1}.
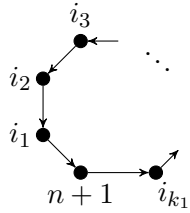
\begin{figure}[htp]
\centering 
\begin{tikzpicture}[scale=0.5, >=stealth']
\tikzstyle{ext}=[circle, draw, minimum size=5, inner sep=1]
\tikzstyle{int}=[circle, draw, fill, minimum size=5, inner sep=1]
\node [int] (vn1) at (0,0) {};
\draw (0, -0.6) node[anchor=center] {{\small $n+1$}};
\node [int] (vi1) at (-1,1) {};
\draw (-1.6, 1) node[anchor=center] {{\small $i_1$}};
\node [int] (vi2) at (-1,2.5) {};
\draw (-1.6, 2.5) node[anchor=center] {{\small $i_2$}};
\node [int] (vi3) at (0,3.5) {};
\draw (0, 4.2) node[anchor=center] {{\small $i_3$}};
\node [int] (vik1) at (2,0) {};
\draw (2.5, -0.6) node[anchor=center] {{\small $i_{k_1}$}};
\draw  [<-] (vn1) edge (vi1);
\draw  [<-] (vi1) edge (vi2);
\draw  [<-] (vi2) edge (vi3);
\draw  [<-] (vi3) edge (1,3.5);
\draw (2, 3) node[anchor=center,  rotate=140] {{\small $\dots$}};
\draw  [->] (vn1) edge (vik1);
\draw  [->] (vik1) edge (2.6,0.6);
\end{tikzpicture}
\caption{The oriented path returns to the vertex with label $(n+1)$} \label{fig:cycle1}
\end{figure}

Since vertex $i_{k_1}$ on figure \eqref{fig:cycle1} is at least trivalent and has at most 
one out-going edge, it has at least one more incoming edge which originates, say, at vertex 
$j_1$. Since $\G$ does not have double edges, $j_1 \le n$. Hence, we can pick an edge which 
terminates at vertex $j_1$ and originates, say, at vertex $j_2$. We continue this process and get another 
oriented path which is shown on figure \ref{fig:cycle-path}. 
\begin{figure}[htp]
\centering 
\begin{tikzpicture}[scale=0.5, >=stealth']
\tikzstyle{ext}=[circle, draw, minimum size=5, inner sep=1]
\tikzstyle{int}=[circle, draw, fill, minimum size=5, inner sep=1]
\node [int] (vn1) at (0,0) {};
\draw (0, -0.6) node[anchor=center] {{\small $n+1$}};
\node [int] (vi1) at (-1,1) {};
\draw (-1.6, 1) node[anchor=center] {{\small $i_1$}};
\node [int] (vi2) at (-1,2.5) {};
\draw (-1.6, 2.5) node[anchor=center] {{\small $i_2$}};
\node [int] (vi3) at (0,3.5) {};
\draw (0, 4.2) node[anchor=center] {{\small $i_3$}};
\node [int] (vik1) at (2,0) {};
\draw (1.8, 0.7) node[anchor=center] {{\small $i_{k_1}$}};
\node [int] (vj1) at (3,-1) {};
\draw (3.6, -1) node[anchor=center] {{\small $j_1$}};
\node [int] (vj2) at (3,-2.5) {};
\draw (3.6, -2.5) node[anchor=center] {{\small $j_2$}};
\draw  [<-] (vn1) edge (vi1);
\draw  [<-] (vi1) edge (vi2);
\draw  [<-] (vi2) edge (vi3);
\draw  [<-] (vi3) edge (1,3.5);
\draw (2, 3) node[anchor=center,  rotate=140] {{\small $\dots$}};
\draw  [->] (vn1) edge (vik1);
\draw  [->] (vik1) edge (2.6,0.6);
\draw  [<-] (vik1) edge (vj1);
\draw  [<-] (vj1) edge (vj2);
\draw  [<-] (vj2) edge (2, -2.5);
\draw (1, -2.5) node[anchor=center] {{\small $\dots$}};
\end{tikzpicture}
\caption{The oriented path returns to the vertex with label $(n+1)$} \label{fig:cycle-path}
\end{figure}
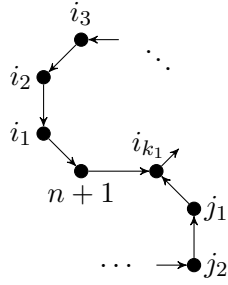 

The path
\begin{equation}
\label{path2}
i_{k_1} \leftarrow j_1 \leftarrow j_2 \leftarrow \dots
\end{equation}
 cannot arrive at any of vertices $i_1, i_2, \dots, i_{k_1}$ because vertices 
with labels $\le n$ have at most one out-going edge. For the same reason, it cannot return 
to any of the vertices $j_1, j_2, \dots $. Hence path \eqref{path2} will eventually return to 
the vertex with label $(n+1)$ and we get another oriented path from vertex  $(n+1)$
to vertex $i_{k_1}$\,.
This path is shown on figure \ref{fig:cycle2}.
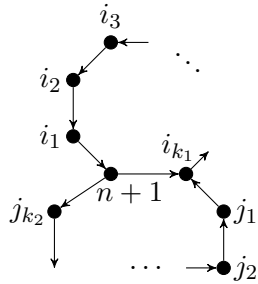
\begin{figure}[htp]
\centering 
\begin{tikzpicture}[scale=0.5, >=stealth']
\tikzstyle{ext}=[circle, draw, minimum size=5, inner sep=1]
\tikzstyle{int}=[circle, draw, fill, minimum size=5, inner sep=1]
\node [int] (vn1) at (0,0) {};
\draw (0.5, -0.5) node[anchor=center] {{\small $n+1$}};
\node [int] (vjk2) at (-1.5,-1) {};
\draw (-2.2, -1) node[anchor=center] {{\small $j_{k_2}$}};

\node [int] (vi1) at (-1,1) {};
\draw (-1.6, 1) node[anchor=center] {{\small $i_1$}};
\node [int] (vi2) at (-1,2.5) {};
\draw (-1.6, 2.5) node[anchor=center] {{\small $i_2$}};
\node [int] (vi3) at (0,3.5) {};
\draw (0, 4.2) node[anchor=center] {{\small $i_3$}};
\node [int] (vik1) at (2,0) {};
\draw (1.8, 0.7) node[anchor=center] {{\small $i_{k_1}$}};
\node [int] (vj1) at (3,-1) {};
\draw (3.6, -1) node[anchor=center] {{\small $j_1$}};
\node [int] (vj2) at (3,-2.5) {};
\draw (3.6, -2.5) node[anchor=center] {{\small $j_2$}};
\draw  [<-] (vn1) edge (vi1);
\draw  [->] (vn1) edge (vjk2);
\draw  [->] (vjk2) edge (-1.5,-2.5);

\draw  [<-] (vi1) edge (vi2);
\draw  [<-] (vi2) edge (vi3);
\draw  [<-] (vi3) edge (1,3.5);
\draw (2, 3) node[anchor=center,  rotate=140] {{\small $\dots$}};
\draw  [->] (vn1) edge (vik1);
\draw  [->] (vik1) edge (2.6,0.6);
\draw  [<-] (vik1) edge (vj1);
\draw  [<-] (vj1) edge (vj2);
\draw  [<-] (vj2) edge (2, -2.5);
\draw (1, -2.5) node[anchor=center] {{\small $\dots$}};
\end{tikzpicture}
\caption{We found another oriented path from vertex $i_{k_1}$ to vertex $(n+1)$} \label{fig:cycle2}
\end{figure} 

Let us now observe that $j_{k_2} \neq i_{k_1}$. Hence, applying the above argument once again 
we construct another oriented path which starts at vertex $(n+1)$, terminates at
vertex $j_{k_2}$, and has length $>1$. 

This process of building oriented paths will not terminate and this contradicts to the fact 
that the graph $\G$ is finite. 

Thus the vertex with label $(n+1)$ does not have incoming edges. 

~\\
{\it A vertex with label $i \le n$ cannot have two incoming edges which originate 
at vertices with labels $\le n$.}

~\\

Indeed, let us consider an edge which originates, say, at vertex $i_1 \le n$ and 
terminates at $i$. Then vertex $i_1$ has at least two incoming edges. 
At least one of these edges originates at vertex $i_2 \le n$. We pick this edge and 
find an edge which terminates at $i_2$ and originates at a vertex with label $i_3 \le n$. 

Continuing this process we find an oriented path which goes only through vertices with 
labels $\le n$ and terminates at $i$. Since the graph $\G$ is finite, we can complete
the path
$$
i \leftarrow i_1 \leftarrow i_2 \leftarrow i_3 \leftarrow \dots
$$
to the cycle:
\begin{equation}
\label{cycle-i}
i \leftarrow i_1 \leftarrow i_2 \leftarrow i_3 \leftarrow \dots \leftarrow i_{k} \leftarrow i
\end{equation}
with $i,i_1, i_2, \dots, i_k \le n$\,.

If there is another edge which terminates at vertex $i$ and originates, say, at vertex $j_1 \le n$
then we may repeat the same process and find another oriented path which terminates at $i$ 
and goes only through vertices with labels $\le n$: 
\begin{equation}
\label{path-j}
i \leftarrow j_1 \leftarrow j_2 \leftarrow j_3 \leftarrow \dots 
\end{equation} 

Since each vertex with label $\le n$ has at most one out-going edge, the set 
of vertices
$$
\{ j_1, j_2, j_3, \dots \}
$$
must have the empty intersection with the set of vertices in the cycle \eqref{cycle-i}. 
In addition, the path \eqref{path-j} cannot return to any of the vertices 
$ j_1, j_2, j_3, \dots $. This observation contradicts to the fact that the graph $\G$ is 
finite. 

Thus, a vertex with label $i \le n$ cannot have two incoming edges which 
originate at vertices with labels $\le n$. 

On the hand, every vertex with label $i \le n$ has at least two incoming edges
and at most one out-going edge. Therefore every vertex with label $i \le n$ has 
valency $3$. It has exactly one out-going edge which terminates at a vertex 
with label $\le n$; it has exactly one incoming edge which originates at a vertex with label 
$\le n$; and it has exactly one incoming edge which originates at the vertex with 
label $(n+1)$. 

We conclude that the graph $\G$ is a ``join of wheels'' shown 
on figure \ref{fig:join}.

\begin{figure}[htp]
\centering 
\begin{tikzpicture}[ext/.style={draw, circle,minimum size=5, inner sep=0},
 int/.style={draw, circle, fill,minimum size=5, inner sep=0}, scale=.7]
\node [int] (v0) at (0,0) {};
\draw (0, -0.6) node[anchor=center] {{\small $n+1$}};

\begin{scope}[shift={(-3,1)}, scale=.7]
\node [int] (v5) at (-144:1.5) {};
\node [int] (v3) at (0:1.5) {};
\node [int] (v2) at (72:1.5) {};
\node [int] (v1) at (144:1.5) {};
\node [int] (v4) at (-72:1.5) {};
\end{scope}
\draw (v1) edge (v2);
\draw (v2) edge (v3);
\draw (v3) edge (v4);
\draw (v4) edge (v5);
\draw (v5) edge (v1);
\draw (v1) edge (v0);
\draw (v3) edge (v0);
\draw (v2) edge (v0);
\draw (v0) edge (v4);
\draw (v5) edge (v0);

\begin{scope}[shift={(2,1)}, scale=0.7]
\node [int] (v15) at (-120:1.5) {};
\node [int] (v13) at (0:1.5) {};
\node [int] (v12) at (120:1.5) {};
\end{scope}

\draw (v0) edge (v12);
\draw (v15) edge (v13);
\draw (v13) edge (v12);
\draw (v15) edge (v12);
\draw (v15) edge (v0);
\node at (-0.2984,2.0464) {$\cdots$};
\end{tikzpicture}
\caption{All unlabeled vertices on the picture should carry labels $\le n$} \label{fig:join}
\end{figure}
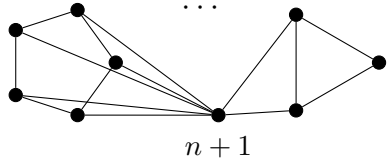

Since $\G$ is 1-vertex irreducible, we conclude immediately that $\G= \si(\G^{wheel}_n)$
for some permutation $\si \in S_n$.

It is obvious that the orientations shown 
on figures \ref{fig:clock}  and \ref{fig:counterclock} are 
the only possible orientations of $\G^{wheel}_n$ satisfying the 
condition stated in the claim. 

Claim \ref{cl:wheels-only} is proven. 
\end{proof}

\subsection{The proof of Theorem \ref{thm:main}}
\label{sec:proof-main}
Let $\G$ be an element of $\gra_{n+1}$ whose underlying graph is 
connected, 1-vertex irreducible, and each vertex of $\G$ has valency 
$\ge 3$. 

Then, for a local section $v$ of the sheaf $\FR'$ \eqref{FR-pr} we consider the local section
\begin{equation}
\label{G-omomom-ga}
\G (\underbrace{\om, \dots, \om}_{n \textrm{ times}}, v) 
\end{equation}
of $\FR'$, where $\om$ is the global section of the sheaf
$\Om^1(\cO_X^{\coord})  \otimes  T^{1,0}(P)$
introduced in Theorem \ref{thm:om} and
the action of $\Gra(n)$ on $\FR'$ is obtained via extending 
$\ma$  \eqref{Gra-to-End-Tpoly} by linearity over  $\Omb(\cO_X^{\coord})$\,.

Let us prove that 
\begin{claim}
\label{cl:Gomom-v-zero}
The section \eqref{G-omomom-ga} is zero 
unless there exists a permutation $\si \in S_n$ such 
that the underlying labeled graph for $\G$ equals 
$\si(\G^{wheel}_n)$, where $\G^{wheel}_n$ is shown on 
figure \ref{fig:wheel}.  
\end{claim}
\begin{proof}
Recall that an action of $\G$ on a collection of polyvector fields is expressed in 
terms of the operators \eqref{Lap}. 

So we will keep track of terms involving the sum
$$
 \sum_{a=1}^d 1 \otimes \dots \otimes 1 \otimes 
\underbrace{\pa_{\xi_a}}_{i\textrm{-th slot}} \otimes 1 \otimes \dots \otimes 1 
\otimes 
\underbrace{\pa_{x^a}}_{j\textrm{-th slot}}  \otimes 1  \otimes \dots  \otimes 1 
$$
by choosing a direction on the edge $(i,j)$ from vertex $i$ to vertex $j$. 
Similarly, we will keep track of terms involving the sum
$$ 
\sum_{a=1}^d 1 \otimes \dots \otimes 1 \otimes 
\underbrace{\pa_{x^a}}_{i\textrm{-th slot}} \otimes 1 \otimes \dots \otimes 1 
\otimes 
\underbrace{\pa_{\xi_a}}_{j\textrm{-th slot}}  \otimes 1  \otimes \dots  \otimes 1
$$   
by choosing a direction on the edge $(i,j)$ from vertex $j$ to vertex $i$. 

We see that $\G (\underbrace{\om, \dots, \om}_{n \textrm{ times}}, v)$ splits into the 
sum over all possible orientations of $\G$ and, since $\om$ is vector-valued, 
a summand corresponding a given orientation of $\G$ is zero if this orientation 
does not satisfy this property: {\it each vertex of $\G$ with label $\le n$ has 
at most one out-going edge.}

Thus Claim \ref{cl:wheels-only} implies the desired statement. 
\end{proof}

Let us now recall that the class $\wt{\si}_n \in H^0(\GC)$ can be represented 
by a cocycle $\ga_n$ of the form 
\begin{equation}
\label{ga-n}
\ga_n = \la \sum_{\si \in S_{n+1}} \si(\G^{wheel}_{n}) + \sum _{i=1}^k \la_i \G_i \in \bs^{2n} \big( \Gra(n+1) \big)^{S_{n+1}}\,,
\end{equation}
where $\la, \la_i$ are non-zero scalars, $\G^{wheel}_n$ is the graph shown 
on figure \ref{fig:wheel} and for each index $i$ the underlying unlabeled 
graph $\G_i$ is not isomorphic to $\G^{wheel}_n$\,.
In addition, every graph $\G_i$ is connected, 1-vertex irreducible and each vertex 
of $\G_i$ has valency $\ge 3$\,. 

Let $v$ be a cocycle in $\cCb(X, \FR)$ which represents 
a cohomology class in \eqref{H-X-Tpoly}. 

Then the cohomology class $\sfD_{\ga_n} ([v])$ is represented by 
the \v{C}ech cocycle $w^n$ with 
\begin{equation}
\label{w-n}
w^n_{\al_0 \dots \al_m} : =  \frac{\la}{n!} \sum_{\si \in S_{n+1}} \si(\G^{wheel}_{n}) (\underbrace{\om, \dots, \om}_{n \textrm{ times}}, 
v_{\al_0 \dots \al_m} ) ~ + ~
 \sum _{i=1}^k \frac{\la_i}{n!}  \G_i (\underbrace{\om, \dots, \om}_{n \textrm{ times}}, 
v_{\al_0 \dots \al_m} )\,. 
\end{equation}

Claim \ref{cl:Gomom-v-zero} implies that  
$$
 \G_i (\underbrace{\om, \dots, \om}_{n \textrm{ times}}, 
v_{\al_0 \dots \al_m} ) = 0
$$
for all $i$ and 
$$
 \si(\G^{wheel}_{n}) (\underbrace{\om, \dots, \om}_{n \textrm{ times}}, 
v_{\al_0 \dots \al_m} ) = 0
$$
unless $\si(n+1) = n+1$\,.

Thus, 
\begin{equation}
\label{w-n-better}
w^n_{\al_0 \dots \al_m}  =  \frac{\la}{n!} \sum_{\si \in S_{n}} \si(\G^{wheel}_{n}) (\underbrace{\om, \dots, \om}_{n \textrm{ times}}, 
v_{\al_0 \dots \al_m} )
\end{equation}
or equivalently, 
\begin{equation}
\label{w-n-even-better}
w^n_{\al_0 \dots \al_m}  = 
\la\, \G^{wheel}_{n} (\underbrace{\om, \dots, \om}_{n \textrm{ times}}, 
v_{\al_0 \dots \al_m})\,.
\end{equation}

Claim \ref{cl:wheels-only} implies that an orientation of $\G^{wheel}_n$
gives the zero contribution to the right hand side of \eqref{w-n-even-better}, 
unless this is the orientation shown on figure \ref{fig:clock} or the orientation 
shown on figure \ref{fig:counterclock}.

Using this observation, it is not hard to see that $w^n_{\al_0 \dots \al_m} $
is obtained by contracting $v_{\al_0 \dots \al_m}$ with the global section 
\begin{equation}
\label{la-pr-Chern-n}
\la' ~ \sum_{1 \le a_1,\dots, a_{n} \le d}~\sum_{1 \le b_1,\dots, b_{n} \le d} ~
\frac{\pa^2 \om^{a_1}}{\pa t^{a_2} \pa t^{b_1}} 
\frac{\pa^2 \om^{a_2}}{\pa t^{a_3} \pa t^{b_2}}
\cdots
\frac{\pa^2 \om^{a_{n}}}{\pa t^{a_{1}} \pa t^{b_{n}}}
dt^{b_1}\cdots dt^{b_{n}}
\end{equation}
of the sheaf $\Omb(\cO_X^{\coord})\,  \otimes\,  \Om^n_{\bbK}(P)$\,, where 
$\la'$ is a non-zero scalar. 

By Theorem  \ref{thm:At-Fed}, the global section  
\begin{equation}
\label{Chern-n}
- \frac{1}{n!} \sum_{1 \le a_1,\dots, a_{n} \le d}~\sum_{1 \le b_1,\dots, b_{n} \le d} ~
\frac{\pa^2 \om^{a_1}}{\pa t^{a_2} \pa t^{b_1}} 
\frac{\pa^2 \om^{a_2}}{\pa t^{a_3} \pa t^{b_2}}
\cdots
\frac{\pa^2 \om^{a_{n}}}{\pa t^{a_{1}} \pa t^{b_{n}}}
dt^{b_1}\cdots dt^{b_{n}}
\end{equation}
represents the $n$-th component of the Chern character on $X$. 

Thus, using Theorem \ref{thm:Fedosov}, we conclude that the contraction 
of the class $[v]$ of $v$ with the $n$-th component of  the Chern character on $X$
is indeed proportional to  $\sfD_{\ga_n} ([v])$ (with a non-zero coefficient).

Theorem \ref{thm:main} is proven. \qed

\section{Application: Isomorphisms between harmonic and Hochschild structures}
\label{sec:harm-Hoch}

Let $X$ be a smooth algebraic variety (over 
an algebraically closed field of characteristic zero). Let, as above, $\cT_{\poly}$ be 
the sheaf of polyvector fields on $X$ and $\Cbu(\cO_X)$ be
the sheaf of polydifferential operators on $X$\,.

Following \cite{Cald}, we call the Gerstenhaber algebras 
\begin{equation}
\label{H-T-poly}
H^{\bul}(X,\cT_{\poly})
\end{equation}
and
\begin{equation}
\label{H-C-OX}
H^{\bul}(X, \Cbu(\cO_X))
\end{equation}
the {\it harmonic structure} and the {\it Hochschild structure} of 
$X$, respectively. 

Due to the Hochschild-Kostant-Rosenberg (HKR) theorem \cite{HKR}, 
the canonical embedding 
\begin{equation}
\label{HKR-map}
\cT_{\poly} \hookrightarrow \Cbu(\cO_X)
\end{equation}
induces an isomorphism of Gerstenhaber algebras from
\eqref{H-T-poly} to \eqref{H-C-OX}, provided $X$ is affine.

In general, the HKR map \eqref{HKR-map} does not induce 
an  isomorphism of Gerstenhaber algebras from
\eqref{H-T-poly} to \eqref{H-C-OX}. However, according to 
\cite[Theorem 1.3]{Damien},  the correction of the HKR map by the ``square root of'' the Todd class 
of $X$ does induce an isomorphism\footnote{The existence of such an isomorphism 
also follows from the results of \cite{DTT}.}
of Gerstenhaber algebras  \eqref{H-T-poly} and \eqref{H-C-OX}.

Let us recall that the ``square root of''  the Todd class of $X$ is given by the formula: 
\begin{equation}
\label{Todd-root}
\Td^{1/2}(X) = \det(\wt{q} ([\sfA]))\,,
\end{equation}
where 
\begin{equation}
\label{q-Todd-root}
\wt{q}(t) = \left( \frac{t}{1-e^{-t}} \right)^{1/2}\,,
\end{equation}
and $[\sfA]$ denotes the Atiyah class of $X$\,. In equation \eqref{Todd-root},  
the expression $\det(\wt{q} ([\sfA]))$ is defined by the formula 
\begin{equation}
\label{det-smthng}
\det(\wt{q} ([\sfA])) = \exp \Big( \tr \log \big( \wt{q} ([\sfA]) \big) \Big)\,, 
\end{equation}
where $[\sfA]$ is considered as the element of the algebra 
$$
\bigoplus_{n \ge 0} H^n(X, \Om^n_X \otimes_{\cO_X} \End(\cT_X))
$$
and $\tr$ is the natural trace map 
$$
\tr : H^n(X, \Om^n_X \otimes_{\cO_X} \End(\cT_X)) \to H^n(X, \Om^n_X)\,.
$$
  
It was proven in \cite{Damien} that the correction of the  HKR map \eqref{HKR-map} 
by the $\hat{A}$-genus 
\begin{equation}
\label{A-genus}
\hat{A}(X) = \det(q ([\sfA]))\,, \qquad 
q(t) = \left( \frac{t}{e^{t/2} - e^{-t/2}} \right)^{1/2}
\end{equation}
also induces an isomorphism of  Gerstenhaber algebras  \eqref{H-T-poly} 
and \eqref{H-C-OX}\,.

Let us observe that the function $q(t)$ is even and 
the function $\wt{q}(t)$ is related to $q(t)$ by the formula 
\begin{equation}
\label{q-wt-q}
\log(q(t)) = \frac{1}{2} \big( \log(\wt{q}(t)) + \log(\wt{q}(-t)) \big)\,.
\end{equation}

This observation motivates us to introduce the notion of generalized $\hat{A}$-genus:  
\begin{defi}
\label{dfn:gen-A-genus}
Let $f(t)$ be a formal power series in $1 + t \bbK[[t]]$ for which the 
even part of $\log(f(t))$ coincides with the Taylor series of the function
\begin{equation}
\label{the-even-part}
\frac{1}{2} \log  \left( \frac{t}{e^{t/2} - e^{-t/2}} \right) 
\end{equation}
at $t=0$. Then the \emph{generalized $\hat{A}$-genus} $\hat{A}_f(X)$ 
of $X$ corresponding to the series $f$ is defined by the equation 
\begin{equation}
\label{gen-A-genus}
\hat{A}_f(X) = \det(f ([\sfA]))\,,
\end{equation}
where $[\sfA]$ is the Atiyah class of $X$\,.
\end{defi}

Theorem \ref{thm:main} implies the following remarkable statement:
\begin{thm}
\label{thm:A-hat}
Let $X$ be a smooth algebraic variety over an algebraically closed field $\bbK$ of 
characteristic zero and let  $f(t)$ be a formal power series in $1 + t \bbK[[t]]$ for which the 
even part of $\log(f(t))$ coincides with the Taylor expansion of \eqref{the-even-part}. 
Then the correction of the HKR map by the generalized $\hat{A}$-genus $\hat{A}_f(X)$
induces an isomorphism of Gerstenhaber algebras  \eqref{H-T-poly} 
and \eqref{H-C-OX}\,.
\end{thm}
\begin{proof}
Let us recall that the $n$-th component  $c_n(X)$  of the Chern 
character of $X$ is represented by 
$$
\frac{1}{n!} \tr\, \sfA^n\,,
$$
where $\sfA$ is any representative of the Atiyah class. 
To prove Theorem \ref{thm:A-hat}, it suffices to show that for every odd $n \ge 1$ 
$$
\exp \left( \frac{1}{n!} \tr\, \sfA^n  \right)
$$
induces an automorphism of the  Gerstenhaber algebra  \eqref{H-T-poly}. 

For $n=1$, this fact is proven in \cite[Section 10.3]{Damien} and the 
remaining cases $n$ odd $\ge 3$ are covered by Theorems \ref{thm:mB-ga-Ger-deriv} and \ref{thm:main}.
\end{proof}

\begin{remark}
\label{rem:relation-to-Kapranov}
This statement is very similar to Proposition 6.2 in \cite{AT}. 
We suspect, based on this, that there may exist a deep link between 
solutions of the generalized Kashiwara-Vergne problem 
\cite{AT} and the above isomorphisms between harmonic 
and Hochschild structures on an algebraic variety.
It is likely that this link can be found using 
the ideas developed in \cite{Calaque}, \cite{CCT}, \cite{ChenXuStienon},  \cite{Kapranov}, 
and \cite{Markarian}.
\end{remark}

\section{Examples}
\label{sec:examples}
It is not easy to find examples of varieties $X$ for which 
the sheaf cohomology of the sheaf of polyvector fields has been computed explicitly in the literature.
Furthermore, the authors do not know any general tools to determine the Gerstenhaber algebra structure on 
$H^{\bul}(X, \cT_\poly)$. This is unfortunate, since this Gerstenhaber algebra structure is an invariant of the variety, 
potentially containing valuable information.

Corollary \ref{cor:main} shall be seen as a first step towards the determination of this Gerstenhaber structure. It gives a very general constraint on the possible products and brackets.

Here we give some examples, for which at least the dimensions of $H^{q}(X, \cT^k_\poly)$ 
may be computed explicitly. 
In this section, we assume that $\bbK= \bbC$ and we use freely tools of 
complex algebraic geometry.
Our goal is to compute
\begin{equation}
\label{equ:Hqtpoly}
H^q(X, \cT_\poly^k) \cong H^q(X, \Omega_X^{d-k}\otimes K_X^{-1})
\end{equation}
where $K_X$ is the canonical bundle of $X$ and $d = \dimens_{\bbC} X$\,. 

\begin{enumerate}
\item For $\mathbb{P}^d$, Grassmannians, and some simple enough flag manifolds the sheaf cohomology of the polyvector fields can be deduced from the Borel-Weil-Bott theorem. Unfortunately, by the same theorem the cohomology is concentrated in degree $q=0$ and the statement of Corollary \ref{cor:main} is trivial.

\item For Calabi-Yau varieties the canonical bundle $K_X$ is trivial and the numbers\\ $\dimens  H^q(X, \cT_\poly^k) = h^{d-k,q}$ agree with the Hodge numbers of $X$. 

\item For complete intersections in $\mathbb{P}^N$ there are explicit formulas for all the twisted Hodge numbers  $h^{p,q}_j = \dimens H^q(X, \Omega_X^{p}(j))$. They have been computed by P. Br\"uckmann \cite[Satz 3]{B3} (see also \cite{B1, B2}). Together with the adjunction formula it follows that $H^q(X, \cT_\poly^k)$ can be computed explicitly in this case.
\end{enumerate}

We will focus on the latter class of examples. Since these results seem not so well known, let us sketch 
a possible way to compute the twisted Hodge diamond (i. e., the numbers $h^{p,q}_j$) for smooth complete intersections 
$X=Y_1\cap \cdots\cap Y_r \subset \mathbb{P}^{d+r}$. The twisted Hodge diamond has the following general form 
(see \cite[Folgerung 2]{B3}):
\usetikzlibrary{matrix,arrows}
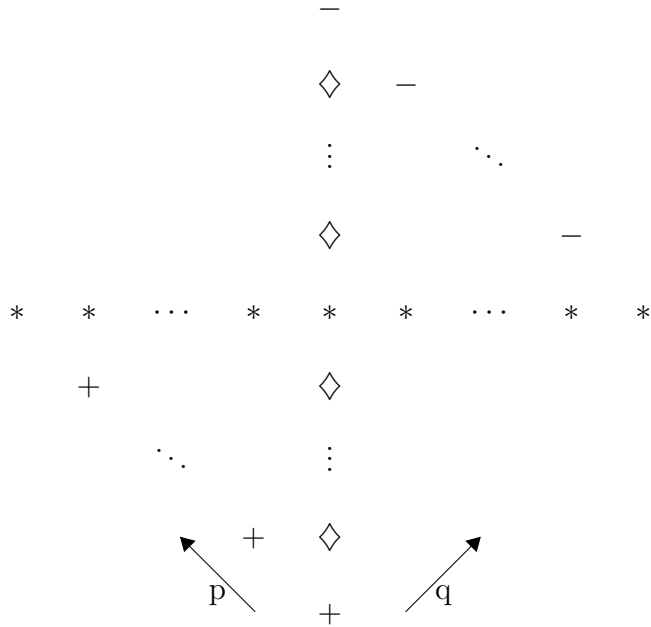
\[
\begin{tikzpicture}[description/.style={fill=white,inner sep=2pt}]
\matrix (m) [matrix of math nodes, row sep=1em,
column sep=1em, text height=1.5ex, text depth=0.25ex]
{
& & & & - & & & &\\
& & & & \ds & - & & &\\
& & & & \vdots & & \ddots & &\\
& & & & \ds & &  &  - &\\
* & * & \cdots& *& * &* &\cdots  &* &* \\
& +& & & \ds & &  &  &\\
& & \ddots & & \vdots & &  &  &\\
& & & + & \ds & &  &  &\\
& & & &+  & &  &  &\\
 };
\draw (-1,-4) edge[-triangle 60] node[below] {p}  (-2,-3);
\draw (1,-4) edge[-triangle 60] node[below] {q}  (2,-3);
\end{tikzpicture}
\]
The numbers not shown are all zero. The symbols $*$, $+$, $-$ stand for some possibly non-zero numbers.\footnote{In particular not all $*$'s are the same in general etc.} The numbers $\ds$ are only present for $j=0$. The numbers $+$ are zero for $j<0$ and the numbers $-$ are zero for $j>0$. In fact, by the weak Lefschetz Theorem, for $j=0$ all the $+$ and $-$ are zero, except for the ones in the central column, which are 1.  There are explicit formulas for all numbers $+$, $-$, $*$. To get them, one may proceed as follows. For $j=0$ the explicit formula was given by Hirzebruch \cite[Section 22]{H}. It is a classical application of the Riemann-Roch-Hirzebruch Theorem. For $j<0$ one may use the Serre duality to restrict to the case $j>0$. 
For $j>0$ the numbers $-$ vanish. The numbers $+$ are the numbers of global holomorphic sections, which are not difficult to compute. The remaining numbers $*$ can be computed from the explicit formula for the Euler characteristic of 
$H^\bullet(X, \Omega_X^{p}(j))$ given by Hirzebruch, see \cite{H}, eqn. (2) on page 160. 
The resulting expressions for the $h_j^{p,q}$ are explicit, but neither pretty nor relevant here, 
so we refer to \cite[Satz 3]{B3} and \cite{B1, B2} instead. 

Now let us consider polyvector fields on a complete intersection $X$. 
By the adjunction formula, $\deg K_X = -d-r-1+\sum_j d_j$, where $d_j$  is the degree of $Y_j$. 
Hence, using \eqref{equ:Hqtpoly}, we see that 
\[
\dimens H^q(X, \cT_\poly^k)= h^{d-k,q}_{d+r+1-\sum_j d_j}.
\]

\subsection{The case of Calabi-Yau complete intersections}
A smooth complete intersection $X$ is Calabi-Yau if and only if 
\begin{equation}
\label{CY-cond}
d+r+1-\sum_j d_j = 0\,.
\end{equation}
In this case, the only nontrivial entries of the Hodge diamond are the numbers $*$, and the numbers $\ds$. 
The corresponding classes of polyvector fields live in $H^k(X, \Tpoly^{k})$ (corresponding to the $*$ entries) and $H^q(X, \Tpoly^{d-q})$ (corresponding to the $\ds$ entries). Since, in the Calabi-Yau case, the Lie bracket on $ H^{\bul}(X, \Tpoly^{\bul}) $
is identically zero (see \cite[Section 2.1]{BK}), we only consider the wedge products. One has the following potentially non-trivial components.
\begin{align*}
\wedge \colon H^k(X, \Tpoly^{k}) \times H^{k'}(X, \Tpoly^{k'}) &\to  H^{k+k'}(X, \Tpoly^{k+k'}) \\
\wedge \colon H^q(X, \Tpoly^{d-q}) \times H^{q'}(X, \Tpoly^{d-q'}) &\to  H^{d}(X, \Tpoly^{d})\cong \bbC & \text{for $q+q'=d$}.
\end{align*}
Here we consider the product with $1\in H^0(X, \Tpoly^{0})$ as a trivial operation.

Let us now show that among Calabi-Yau complete intersections we have  
plenty of examples $X$ for which the odd components $ch_{2l+1}(X)$, $l \ge 1$ of the 
Chern character are non-zero and they act non-trivially on the cohomology
\begin{equation}
\label{H-X-Tpoly-here}
H^{\bul}(X, \cT_\poly)\,.
\end{equation}
 
For this purpose, we observe that the  $n$-th component $ch_n(X)$ of the Chern character $ch(X)$
can be expressed in terms of the first Chern class $h$ of the hyperplane bundle. Namely, 
\begin{equation}
\label{ch-n}
ch_{n}(X) = (d+r+1-\sum_j d^n_j)\frac{h^n}{n!} ~\in~  H^{n}(X, \Omega^{n})\,.
\end{equation}
If at least one degree $d_j$ is $> 1$ and $n >1$ then
the coefficient  $(d+r+1-\sum_j d^n_j)$ is non-zero in virtue 
of \eqref{CY-cond}.

Next, we remark that the classes $h^n$ come from algebraic cycles on 
$X$ and vectors in $H^q(X, \Tpoly^{d-q})$, $q=0,1,\dots$ correspond 
to classes in $H^{q}(X, \Omega^{q})$\,.

On the other hand, if classes $\eta \in  H^{n}(X, \Omega^{n})$ and $v \in H^q(X, \Tpoly^{d-q})$
come from algebraic cycles $\eta'$ and $v'$ on $X$ then the contraction of $\eta$ with $v$ gives us a class 
in $H^{q+n}(X, \Tpoly^{d-q-n})$ which corresponds intersection of the cycles $\eta'$ 
and $v'$. Hence this class is non-trivial provided $n+q \le d$\,.  

Thus we see that, for a large set of tuples $d,r,d_1,\dots, d_r$, the components $ch_n(X)$, $n > 1$
of the Chern character of $X$ act non-trivially on  \eqref{H-X-Tpoly-here}.

Combining these considerations with Theorem \ref{thm:main}, we see that 
Calabi-Yau complete intersections provide us with a large supply of non-trivial representations 
of the Grothendieck-Teichm\"uller Lie algebra $\grt$.

\subsection{The case of Fano complete intersections}
\label{sec:Fano}

We now consider the case $d+r+1-\sum_j d_j > 0$, i.e., $X$ being a Fano variety\footnote{The case of 
$d+r+1-\sum_j d_j<0$ can also be considered, of course, but it adds nothing new to the discussion.}. 
In this case the numbers $-$ are all zero, while the numbers $+$ and $*$ in general are not. 
The Gerstenhaber structure on $H^{\bul}(X, \cT_\poly)$ reduces to the following data: 
First, we have the commutative sub-algebra 
\begin{equation*}
A = \bigoplus_k H^{k}(X, \cT_\poly^k).
\end{equation*}
Second, holomorphic polyvector fields 
form the Gerstenhaber sub-algebra
\[
B := \bigoplus_k H^0(X, \cT_\poly^k)\,.
\] 
The only potentially nontrivial Gerstenhaber operations between elements of $A$ and $B$ 
(counting the product with $1$ as trivial) are the Lie brackets of elements of $H^0(X, \cT_\poly^1)$ 
with elements of $A$.

Unfortunately, in this case, contractions with the odd components $ch_{2l+1}(X)$, $l\geq 1$,
of the Chern character are trivial.

\appendix

\section{Homotopy $O$-algebras. Deformation complex of an $O$-algebra}
\label{app:Def-comp}

Let $O$ be an operad (possibly with a non-zero differential). 
We assume that $O$ admits a cobar resolution
\begin{equation}
\label{Cobar-C-O}
\vf_{O} : \Cobar(C) \stackrel{\sim}{\longrightarrow} O\,,
\end{equation}
where $C$ is a coaugmented dg cooperad satisfying the following 
technical condition: {\it the cokernel $C_{\c}$ of the coaugmentation 
carries an ascending filtration}
\begin{equation}
\label{C-circ-filtr}
\bfzero  = \cF^0 C_{\c} \subset  \cF^1 C_{\c} \subset 
  \cF^2 C_{\c} \subset \dots 
\end{equation}
{\it which is compatible with the pseudo-operad structure on $C_{\c}$, and 
\begin{equation}
\label{C-cocomplete}
C_{\c}(n)  = \bigcup_m \cF^m C_{\c} (n)\,, \qquad \forall ~~ n \ge 0\,. 
\end{equation}
}

For example, if the dg cooperad $C$ has the properties
\begin{equation}
\label{C-conditions}
C(1) \cong \bbK, \qquad \qquad C(0) = \bfzero
\end{equation} 
then the filtration ``by arity'' on $C_{\c}$ satisfies the above technical condition. 

In this paper, we mostly use $O = \Ger$ or $O= \La\Lie$\,. In the former 
case, $C$ is the linear dual to $\La^{-2} \Ger$ and in the latter case 
$C = \La^2 \coCom$\,. It is clear that, in both cases, condition \eqref{C-conditions}
is satisfied.  

Recall that, for a cochain complex $\cV$, we denote by 
\begin{equation}
\label{app-C-cV}
C(\cV) := \bigoplus_{n \ge 1} \Big( C(n) \otimes \cV^{\otimes\, n}  \Big)^{S_n}
\end{equation}
the ``cofree'' $C$-coalgebra co-generated by $\cV$\,.

We also denote by 
\begin{equation}
\label{coDer-C-cV}
\coDer \big(C(\cV)\big)
\end{equation}
the cochain complex of coderivations of the $C$-coalgebra $C(\cV)$\,.

In other words, $\coDer \big(C(\cV)\big)$
consists of $\bbK$-linear maps 
\begin{equation}
\label{cD-coder-C-cV}
\cD : C(\cV) \to C(\cV)
\end{equation}
which are compatible with the $C$-coalgebra structure on $C(\cV)$ in the 
following sense: 
\begin{equation}
\label{coder-axiom}
\D_n \circ \cD  = \sum_{i=1}^n \big( \id_C \otimes \id_{\cV}^{\otimes (i-1)} \otimes \cD \otimes 
\id_{\cV}^{n-i} \big) \circ \D_n
\end{equation}
where $\D_n$ is the comultiplication map 
$$
\D_n :  C(\cV) \to \Big( C(n) \otimes  \big( C(\cV) \big)^{\otimes n} \Big)^{S_n}\,.
$$ 
The $\bbZ$-graded vector space \eqref{coDer-C-cV} carries the natural differential
$\pa$ which comes from those on $C$ and $\cV$\,.

Since the commutator of two coderivations is again a coderivation, the 
cochain complex \eqref{coDer-C-cV} is naturally a dg Lie algebra.

We denote by 
\begin{equation}
\label{coDer-pr-C-cV}
\coDer' \big(C(\cV)\big)
\end{equation}
the dg Lie subalgebra of coderivations $\cD \in \coDer \big(C(\cV)\big)$
satisfying the additional technical condition
\begin{equation}
\label{pr-cond}
\cD \Big|_{\cV} = 0\,.
\end{equation}

Recall that, since the $C$-coalgebra $C(\cV)$ is cofree, every
coderivation $\cD : C(\cV) \to C(\cV) $ is uniquely determined by 
its composition $p_{\cV} \circ \cD$ with the canonical projection:
\begin{equation}
\label{p-cV}
p_{\cV} : C(\cV) \to \cV\,.
\end{equation}

It is not hard to see that the map 
$$
\cD \mapsto p_{\cV} \circ \cD
$$
induces isomorphisms of dg Lie algebras
\begin{equation}
\label{coDer-Conv}
\coDer \big(C(\cV) \big) \cong \Conv(C, \End_{\cV})\,,
\end{equation}
and 
\begin{equation}
\label{coDer-pr-Conv}
\coDer' \big(C(\cV) \big) \cong \Conv(C_{\circ}, \End_{\cV})\,,
\end{equation}
where the differential $\pa$ on $\Conv(C, \End_{\cV})$
and $\Conv(C_{\circ}, \End_{\cV})$ comes solely 
from the differential on $C$ and $\cV$\,. 

Recall that \cite[Proposition 5.2]{notes} 
$\Cobar(C)$-algebra structures on a cochain complex $\cV$
are in bijection with degree $1$ coderivations 
\begin{equation}
\label{Q-C-cV}
Q \in  \coDer' \big(C(\cV) \big)
\end{equation}
satisfying the Maurer-Cartan equation 
\begin{equation}
\label{MC-Q}
\pa Q + \frac{1}{2}[Q, Q] = 0\,.
\end{equation}

Hence, given a $\Cobar(C)$-algebra structure on $\cV$, we may consider 
the dg Lie algebras  \eqref{coDer-Conv}, \eqref{coDer-pr-Conv} and the $C$-coalgebra $C(\cV)$
with the new differentials 
\begin{equation}
\label{diff-tw-Q}
\pa + [Q, ~]\,,
\end{equation}
and
\begin{equation}
\label{diff-tw-Q1}
\pa + Q\,,
\end{equation}
respectively.

Any dg $O$-algebra $\cV$ is naturally a $\Cobar(C)$-algebra. 
Thus, any $O$-algebra structure on $\cV$ gives us a Maurer-Cartan element
\eqref{Q-C-cV} and hence the new differential \eqref{diff-tw-Q} on 
\begin{equation}
\label{Def-comp}
\coDer \big(C(\cV) \big)  \cong \Conv(C, \End_{\cV})\,.
\end{equation}

\begin{defi}
\label{dfn:Def-comp}
The cochain complex \eqref{Def-comp} with the differential \eqref{diff-tw-Q}
is called the \emph{deformation complex} of the $O$-algebra $\cV$. We denote this 
complex by $\Def_{O}(\cV)$ or simply $\Def(\cV)$ when the operad $O$
is clear from the context.
\end{defi}
For more details about the deformation complex and its properties we refer the 
reader to papers \cite{DefCompHi} and \cite{MV}.

For example, if $O= \La\Lie$ then, we may choose $C = \La^2 \coCom$ and, in 
this case,  $\Def_{O}(\cV)$ is the truncated version of the Chevalley-Eilenberg cochain 
complex of  $\cV$ with coefficients in $\cV$. 

It turns out that the deformation complex is a homotopy invariant of 
an $O$-algebra. More precisely,  Theorem 3.1 in \cite{DefCompHi} 
implies that 
\begin{thm}
\label{thm:DefComp-Hi}
If dg $O$-algebras $\cA$ and $\cB$ are quasi-isomorphic then the dg Lie 
algebra $\Def_{O}(\cA)$ is quasi-isomorphic to the 
dg Lie algebra $\Def_{O}(\cB)$\,.  \qed
\end{thm}
\begin{remark}
\label{rem:replace-C}
The construction of the deformation complex $\Def_{O}(\cV)$ depends on
the choice of the cooperad $C$ in \eqref{Cobar-C-O}. However, using homological properties 
\cite[Section 4.4]{notes}  of the bi-functor $\Conv$, it is not hard to prove that, if dg cooperads 
$C$ and $\wt{C}$ are quasi-isomorphic, then the dg Lie algebras $\Conv(C, \End_{\cV})$
and  $\Conv(\wt{C}, \End_{\cV})$ are also quasi-isomorphic. Here, the dg Lie algebras 
$\Conv(C, \End_{\cV})$ and  $\Conv(\wt{C}, \End_{\cV})$ are considered with the differentials 
coming from the $O$-algebra structure on $\cV$\,.
\end{remark}

\subsection{A cocycle in $\Def_{O}(\cV)$ induces a derivation of the $O$-algebra $H^{\bul}(\cV)$}
\label{app:O-deriv}

Let $O$ be an augmented operad in the category of graded vector spaces and $\cV$
be a dg $O$-algebra. 
In this subsection, we show that any cocycle in the deformation complex $\Def_{O}(\cV)$ 
induces a derivation of the $O$-algebra $H^{\bul}(\cV)$\,. 

Let $\cD$ be a cochain in the deformation complex $\Def_O(\cV)$
and $v$ be a cochain in $\cV$\,.

We claim that 
\begin{prop}
\label{prop:cD-de}
The equation 
\begin{equation}
\label{cD-de}
\mB_{\cD} (v) := \cD(v)
\end{equation}
defines a chain map 
\begin{equation}
\label{mB-dfn}
\mB : \Def_O(\cV) \to \Hom (\cV, \cV)\,.
\end{equation}
For any cocycle $\cD$ in $\Def_O(\cV) $ the induced 
map 
\begin{equation}
\label{deriv-of-H-cV}
H^{\bul}(\cV) \to H^{\bul}(\cV) 
\end{equation}
is a derivation of the $O$-algebra $H^{\bul}(\cV)$\,.
\end{prop}

\begin{proof}
The compatibility of $\mB$ with the differentials follows directly from definitions.

To prove that \eqref{deriv-of-H-cV} is a derivation, we recall that $O$ receives a quasi-isomorphism
$\vf_{O}$  \eqref{Cobar-C-O} from $\Cobar(C)$.

Due to Remark \ref{rem:replace-C}, we have a freedom of choosing a convenient 
$\Cobar$-resolution of $O$.  So we choose $C= \Bar(O)$ and observe that 
every vector $\beta$ in $O_{\c}(n)$ gives us a cocycle 
\begin{equation}
\label{beta-pr}
\beta' : =\bs\,(\bsi \beta) \in \Cobar \big( \Bar(O) \big)
\end{equation}
which satisfies the property 
\begin{equation}
\label{beta-pr-P}
\vf_O(\beta') = \beta\,.
\end{equation}

Next, for every $n$-tuple of cocycles 
$$
v_1, v_2, \dots, v_n \in \cV
$$
we consider the cocycle 
\begin{equation}
\label{s-beta-vvv}
(\bsi \beta; v_1, v_2, \dots, v_n)
\end{equation}
in the ``cofree'' $\Bar(O)$-coalgebra $\Bar(O)(\cV)$\,.

Since the coderivation $\cD$ is closed with respect to the 
differential $\pa + [Q, ~]$, we conclude that 
\begin{equation}
\label{cocyc-cond-cD}
\pa \circ \cD (\bsi \beta; v_1, v_2, \dots, v_n) + Q\circ  \cD  (\bsi \beta; v_1, v_2, \dots, v_n)
\end{equation}
$$
- (-1)^{|\cD|}  \cD \circ Q  (\bsi \beta; v_1, v_2, \dots, v_n) = 0\,.
$$

Using the fact that for every elementary co-insertion 
$$
\D_{i,q} : \Bar(O)(n) \to  \Bar(O)(n-q+1) \otimes  \Bar(O)(q)  
$$
$$
\D_{i,q} (\bsi \beta) = 0\,,
$$
we deduce that
\begin{equation}
\label{cD-acts}
\cD  (\bsi \beta; v_1, v_2, \dots, v_n) = p_{\cV} \circ \cD  (\bsi \beta; v_1, v_2, \dots, v_n) 
\end{equation}
$$
+  \sum_{i=1}^n (-1)^{|\cD|(|\beta|-1+|v_1| + \dots +|v_{i-1}|)} (\bsi \beta; v_1, \dots, v_{i-1}, 
\cD(v_i), v_{i+1}, \dots,  v_n)\,,
$$
and 
\begin{equation}
\label{Q-acts}
Q (\bsi \beta; v_1, v_2, \dots, v_n) = p_{\cV} \circ Q  (\bsi \beta; v_1, v_2, \dots, v_n) = 
\beta (v_1, v_2, \dots, v_n)\,. 
\end{equation}

Thus, applying the projection $p_{\cV}$ to both sides of \eqref{cocyc-cond-cD} and 
moving terms around, we get 
\begin{equation}
\label{H-cD-deriv}
 \cD \circ \beta(v_1, v_2, \dots, v_n) - 
  \sum_{i=1}^n (-1)^{|\cD|(|\beta|+ |v_1| + \dots +|v_{i-1}|)} \beta( v_1, \dots, v_{i-1}, 
\cD(v_i), v_{i+1}, \dots,  v_n) =  
\end{equation}
$$
(-1)^{|\cD|}\pa  \big( p_{\cV} \circ \cD (\bsi \beta; v_1, v_2, \dots, v_n) \big)\,.
$$

Proposition \ref{prop:cD-de} is proven.
\end{proof}

\section{Sheaves of algebras over an operad}
\label{app:sheaves}

\subsection{Reminder of the Thom-Sullivan normalization}
\label{app:TS}
Let $\sfDel$ be the simplicial category. In other words, objects 
of $\sfDel$ are ordered sets $[n] : = \{0 < 1 < \dots < n\}$, $n \ge 0$
and morphisms are non-decreasing functions
$$
\vf : \{0 < 1 < \dots < k\} \to \{0 < 1 < \dots < n\}\,.
$$

For the geometric $n$-simplex 
\begin{equation}
\label{Delta-n}
\bDel_n = \{(u_0, u_1, \dots, u_n) \in \bbR^{n+1} \quad | \quad u_i \ge 0, \quad u_0 + u_1 + \dots + u_n = 1\}
\end{equation}
we denote by $C^{\bul}_{simp}(\bDel_n)$ the normalized simplicial cochain complex  and by 
$\Om^{\bul}_{\poly}(\bDel_n)$ the dg commutative algebra of polynomial exterior 
forms on $\bDel_n$\,. Both with coefficients in $\bbK$\,.
It is easy to see that the collection 
\begin{equation}
\label{C-simp-Del-n}
C^{\bul}_{simp}(\bDel_n)
\end{equation}
is a simplicial object in the category of dg associative algebras over $\bbK$ and 
\begin{equation}
\label{Om-bul-Del-n}
\Om^{\bul}_{\poly}(\bDel_n)
\end{equation}
is a simplicial object in the category of dg commutative algebras over $\bbK$. 

The formal integration of polynomial exterior forms gives us a 
natural map of cosimplicial objects 
\begin{equation}
\label{mI-star}
\mI_* : \Om^{\bul}_{\poly}(\bDel_*) \to C^{\bul}_{simp}(\bDel_*)
\end{equation}
and the Stokes theorem implies that this map is compatible with the 
differentials. Furthermore, \cite[Proposition 3.3]{BG} implies that there exists 
a sequence of maps $(n \ge 0)$
\begin{equation}
\label{chi-n}
\chi_n :  \Om^{\bul}_{\poly}(\bDel_n) \otimes  \Om^{\bul}_{\poly}(\bDel_n) \to 
 C^{\bul}_{simp}(\bDel_n)
\end{equation}
of degree $-1$ which are compatible with faces and degeneracies, and 
\begin{equation}
\label{mI-chi}
\mI_n(\eta_1 \eta_2)- \mI_n(\eta_1) \, \mI_n (\eta_2) = d \chi_n(\eta_1, \eta_2) + 
   \chi_n(d \eta_1, \eta_2) + (-1)^{|\eta_1|} \chi_n(\eta_1, d \eta_2)
\end{equation}
for all $\eta_1, \eta_2 \in  \Om^{\bul}_{\poly}(\bDel_n)$\,.
In particular, it means that the multiplication on \eqref{C-simp-Del-n} is 
commutative up to homotopy.

Let us now consider a topological space $X$ with a fixed locally finite 
cover $\{U_{\al}\}_{\al \in \cI}$  by open subsets. 
Then to any dg sheaf $\cV$ on $X$, we may assign the 
cosimplicial cochain complex $\cC(\cV)$ whose $n$-th level is 
$$
\cC(\cV)_n : = \prod_{\al_0, \dots, \al_n} \G(U_{\al_0} \cap \dots \cap U_{\al_n}, \cV)\,.
$$
The $i$-th co-face comes from omitting $U_{\al_i}$ and the $i$-th degeneracy 
comes from repeating $U_{\al_i}$ twice. 

It is not hard to see that $\cC$ is a functor 
from the category of dg sheaves on $X$ to the category of cosimplicial 
cochain complexes. 

The normalized \v{C}ech complex of a dg sheaf $\cV$ can be defined as 
\begin{equation}
\label{Cech-cV-dfn}
\cCb(\cV) : = C^{\bul}_{simp}(\bDel_*) \, \otimes_{\sfDel}\, \cC(\cV)_*\,.
\end{equation}
In other words, $\cCb(\cV)$ is the subspace of vectors 
$$
\sum_{n\ge 0} \ka_n \otimes w_n ~ \in ~  \prod_{n \ge 0}  C^{\bul}_{simp}(\bDel_n) \, \otimes \, \cC(\cV)_n
$$
satisfying the condition 
$$
\ka_k \otimes  \vf_*(w_k) = \vf^*(\ka_n) \otimes w_n \qquad \textrm{in} \qquad 
 C^{\bul}_{simp}(\bDel_k) \otimes \cC(\cV)_n
$$ 
for every morphism $\vf : [k] \to [n]$ in the category $\sfDel$. 

It is not hard to show that $\cCb(\cV)$ is isomorphic to 
\begin{equation}
\label{Cech-cV-conven}
\cCb(\cV) :=  \prod_{\al_0, \dots, \al_n} \bs^n \G(U_{\al_0} \cap \dots \cap U_{\al_n}, \cV)\,,
\end{equation}
where the product is taken over all $n$-tuples $(\al_0, \dots, \al_n) \in \cI$ for which
$\al_i \neq \al_j$ if $i \neq j$\,. 

The differential on \eqref{Cech-cV-conven} is the sum of the differential 
$\pa_{\cV}$ coming from $\cV$ and the \v{C}ech differential 
\begin{equation}
\label{Cech-diff}
\cpa(f)_{\al_0 \al_1 \dots \al_n} = 
\sum_{j=0}^n (-1)^{|f|+j} f_{\al_0 \dots \wh{\al_j} \dots \al_n}\,. 
\end{equation}

\begin{remark}
\label{rem:covers}
Here we tacitly assume that the chosen cover of $X$ is acyclic for the 
class of dg sheaves we consider. So we can always use the Cech resolution 
for computing the sheaf cohomology.  In this article, $X$ is usually a smooth algebraic 
variety considered with the Zariski topology. For our purposes, any cover of $X$ by open 
affine subsets, each of which admits a global system of parameters, suffices. 
\end{remark}

Let $\cV_1, \cV_2$ be a pair of dg sheaves.
The structure of associative algebra on \eqref{C-simp-Del-n} gives 
us the map of cochain complexes: 
\begin{equation}
\label{AW} 
\AW : \cCb(\cV_1) \otimes  \cCb(\cV_2)  \to \cCb(\cV_1 \otimes \cV_2) 
\end{equation}
which is given by the formula: 
\begin{equation}
\label{AW-dfn-here}
\AW(f^1, f^2)_{\al_0 \dots \al_m} : = 
\sum_{0 \le k \le m}
(-1)^{|f^2| k} f^1_{\al_0 \dots \al_k} \otimes f^2_{\al_k \dots \al_m}\,.
\end{equation}
We call $\AW$ the {\it Alexander-Whitney map}.

It is not hard to see that $\AW$ equips $\cCb$ with a natural structure of 
a monoidal functor from the category of dg sheaves to the category of cochain 
complexes. Unfortunately, the transformation $\AW$ is compatible with the 
braiding only up to homotopy. So, in general, $\AW$ cannot be used to pull an algebraic 
structure from a dg sheaf $\cV$ to its \v{C}ech complex $\cCb(\cV)$\,.

It is the Thom-Sullivan normalization $\cN^{\TS}$ \cite{BG}, \cite[Section 1]{H-Yekutieli}, 
\cite[Appendix A]{VdB}, which repairs this defect. 
Namely,  $\cN^{\TS}$ is a functor from the category of cosimplicial
cochain complexes to the category of cochain complexes defined by the formula 
\begin{equation}
\label{N-TS}
\cN^{\TS}(\mS) :=  \Om^{\bul}_{\poly}(\bDel_*) \, \otimes_{\sfDel}\, \mS_*\,,
\end{equation}
where $\mS$ is a cosimplicial cochain complex. (For example, $\mS = \cC(\cV)$ 
for a dg sheaf $\cV$ on $X$). 

Composing $\cN^{\TS}$ with $\cC$, we get a functor from the category of 
dg sheaves on $X$ to the category of cochain complexes. 
This composition $\cN^{\TS} \circ \cC$  satisfies the following remarkable properties:
\begin{pty}
\label{P:tensor-functor} 
The natural transformation 
\begin{equation}
\label{mu-TS}
\mu^{\TS} : \cN^{\TS} \circ \cC(\cV_1) \otimes  \cN^{\TS} \circ \cC(\cV_2) \to 
 \cN^{\TS} \circ \cC(\cV_1 \otimes \cV_2) 
\end{equation}
coming from the multiplication of exterior forms on geometric simplices 
equips $\cN^{\TS} \circ \cC$ with a structure a monoidal functor 
from the category of dg sheaves to the category cochain complexes. 
This functor respects the braidings ``on the nose''. 
\end{pty}
\begin{pty}
\label{P:q-iso-Cech}
If $\cV$ is a cochain complex of sheaves with degree bounded 
below then the map
$$
\mI^{\cC}_{\cV} :  \cN^{\TS} \circ \cC(\cV) \to \cCb(\cV)
$$
induced by the natural transformation 
\begin{equation}
\label{N-TS-to-Cech}
\mI^{\cC} : \cN^{\TS} \circ \cC \to \cCb
\end{equation}
is a quasi-isomorphism of cochain complexes.  
\end{pty}
\begin{pty}
\label{P:preserve-q-iso} 
The functor $\cN^{\TS} \circ \cC$ preserves quasi-isomorphisms. 
\end{pty}
Property \ref{P:tensor-functor} follows immediately from the construction and 
Property \ref{P:q-iso-Cech} is a consequence of \cite[Theorem 2.2]{BG}. Finally, 
Property \ref{P:preserve-q-iso} follows from Property \ref{P:q-iso-Cech} and the 
fact that the functor $\cCb$ sends quasi-isomorphisms of sheaves to quasi-isomorphisms 
of cochain complexes. 
 
In addition, we remark that the existence of maps \eqref{chi-n} satisfying 
\eqref{mI-chi} implies that 
\begin{prop}
\label{prop:BG}
For every pair of dg sheaves $\cV_1$, $\cV_2$ on $X$ the diagram 
\begin{equation}
\label{diag-BG}
 \begin{tikzpicture}
\matrix (m) [matrix of math nodes, column sep=1.6em, row sep=1.6em]
{\cN^{\TS}\circ \cC(\cV_1) \otimes \cN^{\TS}\circ \cC(\cV_2) &  \cN^{\TS}\circ \cC(\cV_1 \otimes \cV_2)  \\
\cCb(\cV_1) \otimes  \cCb(\cV_2)  &  \cCb(\cV_1 \otimes \cV_2)\\};
\path[->,font=\scriptsize]
(m-1-1) edge node[auto] {$\mu^{\TS}$} (m-1-2)   edge node[auto] {$\mI_{\cV_1} \otimes \mI_{\cV_2}$} (m-2-1)  
(m-2-1)  edge node[auto] {$\AW$} (m-2-2)   (m-1-2)  edge node[auto] {$\mI_{\cV_1 \otimes \cV_2}$} (m-2-2);
\end{tikzpicture} 
\end{equation}
commutes up to homotopy. $\Box$
\end{prop}

\subsection{Derived global sections of a dg sheaf of $O$-algebras}
\label{app:RGamma}

Let $\cA$ be a dg sheaf on a topological space $X$. 
For $\cA$ we form a dg operad $\End_{\cA}$. As a graded 
vector space,
\begin{equation}
\label{End-cA}
\End_{\cA}(n) : = \bigoplus_m \Hom^{(m)} (\cA^{\otimes n}, \cA)\,,
\end{equation}
where $ \Hom^{(m)} (\cA^{\otimes n}, \cA)$ consists of $\bbK$-linear 
maps of sheaves of degree $m$, and the differential 
$$
\pa :  \Hom^{(m)} (\cA^{\otimes n}, \cA) \to  \Hom^{(m+1)} (\cA^{\otimes n}, \cA) 
$$
comes naturally from the differential on $\cA$.

Let $O$ be a dg operad. We recall that an $O$-algebra structure 
on a dg sheaf $\cA$ is a map of dg operads
\begin{equation}
\label{O-alg-on-cA}
O \to \End_{\cA}\,. 
\end{equation}

The monoidal structure on the functor $\cN^{\TS} \circ \cC$ gives 
us a canonical map of dg operads 
\begin{equation}
\label{End-cA-End-N-TS}
\End_{\cA} \to \End_{\cN^{\TS} \circ \cC(\cA)}\,. 
\end{equation}
Hence, for every dg sheaf $\cA$ of $O$-algebras the cochain complex 
\begin{equation}
\label{N-TS-cA}
\cN^{\TS} \circ \cC(\cA)
\end{equation}
is naturally an algebra over $O$\,.

We call the $O$-algebra \eqref{N-TS-cA} the {\it algebra of derived 
global sections of} $\cA$\,.

\subsection{Deformation complex of a sheaf of $O$-algebras}
\label{app:Def-comp-sheaf}
Deformation complex of an $O$-algebra admits a generalization 
to the setting of sheaves. We briefly describe this generalization here 
and refer the reader for more details to \cite[Section 4]{DefCompHi}. 

Let $\cA$ be a dg sheaf on a topological space $X$ equipped with an 
algebra structure over $O$\,. According to Subsection \ref{app:RGamma}, the 
cochain complex \eqref{N-TS-cA}
is naturally an algebra over the operad $O$ and hence an 
algebra over the operad $\Cobar(C)$\,.

Using this $\Cobar(C)$-algebra structure on \eqref{N-TS-cA}, we get 
a degree $1$ coderivation
\begin{equation}
\label{Q-for-N-TS}
Q^{\TS} \in \coDer' \Big( C \big(\, \cN^{\TS} \circ \cC(\cA)\, \big) \Big)
\end{equation}
of the ``cofree'' $C$-coalgebra 
$$
C \big(\, \cN^{\TS} \circ \cC(\cA)\, \big)
$$
satisfying the Maurer-Cartan equation.  

Using this coderivation $Q^{\TS}$, we equip the graded Lie algebra 
\begin{equation}
\label{coDer-N-TS-cA}
 \coDer \Big( C \big(\, \cN^{\TS} \circ \cC(\cA)\, \big) \Big)
\end{equation}
with the differential $\pa + [Q^{\TS}, ~]$\,. 
\begin{defi}
\label{dfn:DefComp-sheaves}
For a dg sheaf of $O$-algebras $\cA$ we call the cochain 
complex \eqref{coDer-N-TS-cA} with the differential $\pa + [Q^{\TS},~]$
the \emph{deformation complex} of $\cA$.  We denote this 
complex by $\Def_{O}(\cA)$ or simply $\Def(\cA)$ when the operad $O$
is clear from the context.
\end{defi}

According to \cite{DefCompHi} the deformation complex $\Def_O(\cA)$ is 
a homotopy invariant of $\cA$. More precisely, 
\begin{thm}{\bf \cite[Theorem 4.11]{DefCompHi}}
\label{thm:DefComp-Hi-sh}
Let $\cA$, $\cB$ be dg sheaves of $O$-algebras. 
If there exists a sequence of quasi-isomorphisms of sheaves 
of $O$-algebras connecting $\cA$ to $\cB$, then the dg Lie algebras
 $\Def_{O}(\cA)$ and $\Def_{O}(\cB)$ are 
quasi-isomorphic.  \qed
\end{thm}

\subsection{A canonical homomorphism from $\coDer(C(\cA))$ to $\Def_O(\cA)$}
\label{app:Canon-homo}

For a dg sheaf $\cA$ we denote by $C(\cA)$ the dg sheaf of 
$C$-coalgebras ``cofreely'' cogenerated by $\cA$. We also denote by
\begin{equation}
\label{coDer-C-cA}
\coDer \big(C(\cA)\big)
\end{equation}
the cochain complex of coderivations of $C(\cA)$\,. 
In other words, $\coDer \big(C(\cA)\big)$
consists of maps of sheaves 
\begin{equation}
\label{cD-coder}
\cD : C(\cA) \to C(\cA)
\end{equation}
which are compatible with the $C$-coalgebra structure on $C(\cA)$ in the 
following sense: 
\begin{equation}
\label{coder-axiom-sheaves}
\D_n \circ \cD  = \sum_{i=1}^n \big( \id_C \otimes \id_{\cA}^{\otimes (i-1)} \otimes \cD \otimes 
\id_{\cA}^{n-i} \big) \circ \D_n
\end{equation}
where $\D_n$ is the comultiplication map 
$$
\D_n :  C(\cA) \to \Big( C(n) \otimes  \big( C(\cA) \big)^{\otimes n} \Big)^{S_n}\,.
$$ 
The graded vector space \eqref{coDer-C-cA} carries the obvious Lie bracket 
and the natural differential $\pa$ which comes from those on $C$ and $\cA$\,.

We denote by 
\begin{equation}
\label{coDer-pr}
\coDer' \big(C(\cA)\big)
\end{equation}
the dg Lie subalgebra of \eqref{coDer-C-cA} which consists of coderivations $\cD$
satisfying the additional technical condition 
$$
\cD \Big|_{\cA} = 0\,.
$$

Let us denote by $p_{\cA}$ the canonical projection
$$
p_{\cA} :   C(\cA) \to \cA\,.
$$  

It is not hard to see that the assignment 
$$
\cD \mapsto p_{\cA} \circ \cD
$$
gives us isomorphisms of dg Lie algebras 
\begin{equation}
\label{coDer-Conv-sheaf}
\coDer \big(C(\cA)\big) \cong \Conv(C, \End_{\cA})\,,
\end{equation}
and
\begin{equation}
\label{coDer-pr-Conv-sheaf}
\coDer' \big(C(\cA)\big) \cong \Conv(C_{\c}, \End_{\cA})\,,
\end{equation}
where all the dg Lie algebras are considered with the differentials 
coming solely from the ones on $C$ and $\cA$\,.

Using the map of dg operads \eqref{End-cA-End-N-TS}, we get 
the canonical morphism of dg Lie algebras 
\begin{equation}
\label{the-one}
\Psi^{\hs} : \Conv(C, \End_{\cA}) \to  \Conv \big( C,  \End_{\cN^{\TS} \circ \cC(\cA)} \big)\,,
\end{equation}
where, again, the differentials come solely from the ones on $C$, $\cA$, and
$\cN^{\TS} \circ \cC$\,.

If $\cA$ is, in addition, a sheaf of $O$-algebras, then $\cA$ is also a sheaf 
of $\Cobar(C)$-algebras and $\cN^{\TS} \circ \cC(\cA)$ is a $\Cobar(C)$-algebra.
Hence, due to \cite[Proposition 5.2]{notes}, we get the Maurer-Cartan element
\begin{equation}
\label{cQ}
\cQ \in  \Conv(C_{\c}, \End_{\cA})
\end{equation}
and hence a new differential $\pa + [\cQ, ~]$ on \eqref{coDer-Conv-sheaf}.

Let us recall that the graded Lie algebras
\begin{equation}
\label{Conv-N-TS}  
\Conv(C, \End_{\cN^{\TS} \circ \cC(\cA)})
\end{equation}
and 
$$
\coDer \Big( C \big( \End_{\cN^{\TS} \circ \cC(\cA)} \big) \Big)
$$
are isomorphic. Furthermore, it is not hard to see that, 
the Maurer-Cartan element $Q^{\TS}$ \eqref{Q-for-N-TS}
is related to $\cQ$ via the equation
\begin{equation}
\label{Q-TS-cQ}
Q^{\TS} = \Psi^{\hs} (\cQ)\,. 
\end{equation}

Therefore, the same map $\Psi^{\hs}$ \eqref{the-one} gives 
us a morphism of dg Lie algebras 
\begin{equation}
\label{Psi-heart}
\Psi^{\hs} : \Big( \Conv(C, \End_{\cA}), \pa + [\cQ, ~] \Big)
\to  \Big( \Conv \big( C,  \End_{\cN^{\TS} \circ \cC(\cA)} \big), \pa + [Q^{\TS}, ~] \Big)
\end{equation}
and, since the target is canonically isomorphic to the deformation complex 
$\Def_{O}(\cA)$, $\Psi^{\hs}$ can be viewed as a map of dg Lie algebras
\begin{equation}
\label{heart}
\Psi^{\hs} :  \coDer(C(\cA))
\to  \Def_{O}(\cA)\,.
\end{equation}

\subsection{A cocycle in $\coDer(C(\cA))$ induces a derivation of the $O$-algebra $H^{\bul}(X, \cA)$}
\label{app:O-deriv-sheaves}

Let us now adapt the construction of Subsection \ref{app:O-deriv} to the setting of sheaves. 
Just as in Subsection \ref{app:O-deriv}, we assume that the operad $O$ carries the zero differential.

Let $\cA$ be a dg sheaf of $O$-algebras and  $\cD$ be a cocycle in $\coDer(C(\cA))$, where 
$\coDer(C(\cA))$ is considered with the differential $\pa + [\cQ, ~]$\,. According to the previous 
subsection, 
\begin{equation}
\label{Psi-hs-cD}
\Psi^{\hs} (\cD)
\end{equation}
is a cocycle in the deformation complex  
\begin{equation}
\label{Def-cA-here}
 \Def_{O}(\cA) \cong   \Big( \Conv \big( C,  \End_{\cN^{\TS} \circ \cC(\cA)} \big), \pa + [Q^{\TS}, ~] \Big)\,.
\end{equation}
 
Therefore, by Proposition \ref{prop:cD-de}, the map 
\begin{equation}
\label{mB-Psi-cD}
\mB_{\Psi^{\hs}(\cD)} : \cN^{\TS} \circ \cC(\cA) \to \cN^{\TS} \circ \cC(\cA)
\end{equation}
induces a derivation on the $O$-algebra 
\begin{equation}
\label{H-N-TS-cA}
H^{\bul} \Big( \cN^{\TS} \circ \cC(\cA) \Big)\,.
\end{equation}

On the other hand,  the cochain complex $\cN^{\TS} \circ \cC(\cA)$ computes the sheaf 
cohomology $H^{\bul}(X, \cA)$\,. Hence, the map \eqref{mB-Psi-cD} induces a derivation of the 
$O$-algebra  $H^{\bul}(X, \cA)$\,.

For our purposes, we need an explicit way of computing this derivation in terms of conventional 
\v{C}ech cochains. This is given by the following proposition. 
\begin{prop}
\label{prop:O-deriv}
Let  $\cA$ be a dg sheaf of $O$-algebras, $\cD$ be a cocycle in $\coDer(C(\cA))$, and 
$v$ be a cochain in the \v{C}ech complex $\cCb(\cA)$. The formula 
\begin{equation}
\label{mB-Cech}
\cmB_{\cD}(v)_{\al_0 \al_1 \dots \al_m} : = \cD(v_{\al_0 \al_1 \dots \al_m})
\end{equation}
defines degree $|\cD|$ map 
$$
\cmB_{\cD} : \cCb(\cA) \to \cCb(\cA)
$$
which intertwines the differentials and such that the corresponding 
map 
$$
H^{\bul}(X, \cA) \to H^{\bul}(X, \cA)
$$
coincides with the derivation induced by \eqref{mB-Psi-cD}.  
\end{prop}
\begin{proof}
The compatibility of $\cmB_{\cD}$ with the differentials follows easily from the fact 
that $\cD$ is a cocycle in  $\coDer(C(\cA))$.

Next, unfolding the definitions, we see that the diagram 
\begin{equation}
\label{mB-mB-Cech}
 \begin{tikzpicture}
\matrix (m) [matrix of math nodes, column sep=3.6em, row sep=1.6em]
{ \cN^{\TS} \circ \cC(\cA) & \cN^{\TS} \circ \cC(\cA)\\
 \cCb(\cA) &  \cCb(\cA)\\};
\path[->,font=\scriptsize]
(m-1-1) edge node[auto] {$\mB_{\Psi^{\hs}(\cD)}$} (m-1-2)   edge node[auto] {$\mI_{\cA}$} (m-2-1)  
(m-2-1)  edge node[auto] {$\cmB_{\cD}$} (m-2-2)   (m-1-2)  edge node[auto] {$\mI_{\cA}$} (m-2-2);
\end{tikzpicture} 
\end{equation}
commutes. So the second claim of the proposition follows as well.
\end{proof}

\section{Operations of twisting}
\label{app:twisting}

In this section, we recall twisting of $\La\Lie$-algebras and  Gerstenhaber algebras 
by Maurer-Cartan elements. We also extend the twisting operation to a subspace 
of cochains in the deformation complex of a $\La\Lie$-algebra and 
a Gerstenhaber algebra. For more details about the twisting procedure we 
refer the reader to \cite{DeligneTwist}.

\subsection{Twisting operation for Chevalley-Eilenberg cochains}
\label{app:tw-LaLie}

Let $\cV$ be a dg $\La\Lie$-algebra equipped with a complete 
descending filtration
\begin{equation}
\label{filtr-cV}
\cV \supset \dots \supset \cF_0 \cV \supset \cF_1 \cV \supset  \cF_2 \cV \supset \dots\,,
\end{equation}
\begin{equation}
\label{cV-complete}
\cV = \lim_k \cV ~ \big/  ~\cF_k \cV\,,
\end{equation}
which is compatible with the dg $\La\Lie$-structure. 

We say that a degree $2$ vector $\al \in \cV$ is a {\it Maurer-Cartan 
element}\footnote{Condition \eqref{al-filtr-1} is sometimes omitted.} if 
\begin{equation}
\label{al-filtr-1}
\al \in \cF_1 \cV
\end{equation}
and 
\begin{equation}
\label{MC-al}
\pa \al + \frac{1}{2}\{\al, \al\} = 0\,.
\end{equation}

For every Maurer-Cartan element $\al$ of a $\La\Lie$-algebra $\cV$,
the equations
\begin{equation}
\label{tw-diff}
\pa^{\al}_{\cV} : = \pa_{\cV} + \{\al, ~\}
\end{equation}
and
\begin{equation}
\label{tw-bracket}
\{~,~\}^{\al} = \{~,~\}
\end{equation}
define a new dg $\La\Lie$-structure on $\cV$\,. 

We denote this new  $\La\Lie$-algebra by $\cV^{\al}$ and say 
that $\cV^{\al}$ is obtained from $\cV$ via {\it twisting} by the Maurer-Cartan element 
$\al$\,. 

Now, for a given  Maurer-Cartan element $\al$ of $\cV$, we consider
the following element of the completion of $\La^2 \coCom(\cV)$ 
\begin{equation}
\label{the-elem}
\bs^2\, (e^{\bs^{-2}\, \al} -1) = \sum_{n=1}^{\infty} \frac{1}{n!} 
 \bs^2\,(\bs^{-2}\, \al)^{n}
\end{equation}
and define the subspace 
of coderivations $\cD \in  \coDer \big( \La^2 \coCom(\cV) \big)$ 
satisfying the additional condition\footnote{We tacitly assume that our 
coderivations are compatible with the filtration on $ \La^2 \coCom(\cV)$
coming from  \eqref{filtr-cV}.}
\begin{equation}
\label{do-not-move}
\cD\, \bs^2 \, ( e^{\bs^{-2}\, \al} -1) = 0\,.
\end{equation}
This subspace is obviously closed with respect to the commutator. 
Furthermore, we have the following theorem: 
\begin{thm}
\label{thm:twisting}
Let $\cV$ be a filtered $\La\Lie$-algebra, $\al$ be a Maurer-Cartan 
element of $\cV$\,, $p_{\cV} : \La^2 \coCom(\cV) \to \cV$ be the 
canonical projection,
and $\msQ$ (resp. $\msQ^{\al}$) be the codifferential 
on  $\La^2 \coCom(\cV)$   (resp.  $\La^2 \coCom(\cV^{\al})$) corresponding 
to the dg $\La\Lie$-structures on $\cV$ (resp. $\cV^{\al}$). 
Let us also denote by 
\begin{equation}
\label{coder-cond}
\coDer \big( \La^2 \coCom(\cV) \big)_{\al} 
\end{equation}
the subspace of coderivations of $\La^2 \coCom(\cV)$ satisfying 
condition \eqref{do-not-move}. Then  

\begin{itemize}

\item[{\bf i)}] Condition \eqref{do-not-move} on coderivations is equivalent to 
\begin{equation}
\label{do-not-move1}
\sum_{n=1}^{\infty}\frac{1}{n!}\,
p_{\cV} \circ \cD \,  \big( \bs^2\, (\bs^{-2}\, \al)^{n} \big) = 0\,.
\end{equation}

\item[{\bf ii)}]  The codifferential $\msQ$
satisfies Condition  \eqref{do-not-move}.

\item[{\bf iii)}] For every coderivation $\cD$ in \eqref{coder-cond}  
the operation 
\begin{equation}
\label{cD-tw}
e^{ - \bs^{-2}\, \al} \, \cD \,  e^{\bs^{-2}\, \al} :  
 \La^2 \coCom(\cV) \to  \La^2 \coCom(\cV)
\end{equation}
is a coderivation of  $\La^2 \coCom(\cV)$.

\item[{\bf iv)}]  The codifferential $\msQ^{\al}$ is related 
to $\msQ$ by the formula: 
\begin{equation}
\label{msQ-tw}
\msQ^{\al} ~ = ~ 
e^{ - \bs^{-2}\, \al} \, \msQ \,  e^{\bs^{-2}\, \al}\,. 
\end{equation}

\item[{\bf v)}]  The subspace \eqref{coder-cond} is a subcomplex
of the deformation complex for $\cV$\,.
Furthermore, the assignment 
\begin{equation}
\label{CE-twisting}
\cD ~ \mapsto ~ \cD^{\al} =
e^{ - \bs^{-2}\, \al} \, \cD \,  e^{\bs^{-2}\, \al}  
\end{equation}
defines a map of cochain complexes 
\begin{equation}
\label{CE-twisting-map}
\coDer \big( \La^2 \coCom(\cV) \big)_{\al} ~ \to ~
  \coDer \big( \La^2 \coCom(\cV^{\al}) \big)
\end{equation}
from \eqref{coder-cond} to the deformation complex 
$\coDer \big( \La^2 \coCom(\cV^{\al}) \big)$ of the $\La\Lie$-algebra 
$\cV^{\al}$. 

\end{itemize}
\end{thm}
\begin{proof}
Using the compatibility of $\cD$ with the comultiplication on 
$\La^2 \coCom(\cV)$, it is not hard to show that  
\begin{equation}
\label{cD-exp}
\cD\,   \bs^2 \, ( e^{\bs^{-2}\, \al} -1) = 
 e^{\bs^{-2}\, \al}\, p_{\cV} \circ \cD\big( \bs^2 \, ( e^{\bs^{-2}\, \al} -1) \big)\,.
\end{equation}
Hence, \eqref{do-not-move} is equivalent to \eqref{do-not-move1} 
and claim {\bf i)} follows. 

Since $\al$ satisfies the Maurer-Cartan equation \eqref{MC-al}, we have 
$$
p_{\cV} \circ \msQ \,  \bs^2 \, ( e^{\bs^{-2}\, \al} -1) = 0\,.
$$
Thus,  claim {\bf i)} implies claim  claim {\bf ii)}. 

To prove  claim {\bf iii)}, we recall that the comultiplication $\D$ on 
$$ 
\La^2 \coCom(\cV)  = \bs^2 \und{S}(\bs^{-2} \, \cV)
$$
is given by the formula: 
\begin{equation}
\label{Delta-coCom}
\D (\bs^2\, w_1 w_2 \dots w_n ) =
\end{equation}
$$
 \sum_{p=1}^{n-1} \sum_{\tau \in \Sh_{p, n-p}}
 (-1)^{\ve(\tau, w_1, \dots, w_n)} 
(\bs^2\,  w_{\tau(1)} \dots  w_{\tau(p)}) ~\otimes~
(\bs^2\,  w_{\tau(p+1)} \dots  w_{\tau(n)})\,,
$$
$$
 w_1, w_2, \dots, w_n \in \bs^{-2} \, \cV\,,
$$
where the sign factor $ (-1)^{\ve(\tau, w_1, \dots, w_n)} $ is determined by 
the usual Koszul rule of signs. 

Let us extend $\cD$ to the space of the full symmetric algebra 
\begin{equation}
\label{full-symm}
 \bs^2  S(\bs^{-2} \, \cV)
\end{equation}
by requiring that 
\begin{equation}
\label{cD-1}
\cD(\bs^2 \,1) = 0\,. 
\end{equation}

Then $\cD$ respects the following comultiplication on \eqref{full-symm} 
$$
\wt{\D} (\bs^2 \, 1) = \bs^2 \, 1 ~ \otimes ~\bs^2 \, 1 \,,
$$
\begin{equation}
\label{Delta-full-symm}
\wt{\D} (\bs^2\, w_1 \dots w_n ) =
\end{equation}
$$
\bs^2\, 1 \otimes (\bs^2\, w_1 \dots w_n ) +
 \sum_{p=1}^{n-1} \sum_{\tau \in \Sh_{p, n-p}}
 (-1)^{\ve(\tau, w_1, \dots, w_n)} 
(\bs^2\,  w_{\tau(1)} \dots  w_{\tau(p)}) ~\otimes~
(\bs^2\,  w_{\tau(p+1)} \dots  w_{\tau(n)}) 
$$
$$
 +  (\bs^2\, w_1 \dots w_n ) \otimes  \bs^2 \, 1\,,   \qquad 
 w_1, \dots, w_n \in \bs^{-2} \, \cV\,, \qquad n \ge 1
$$
in the sense of the identity
\begin{equation}
\label{cD-wtDelta}
\wt{\D} \circ \cD =  ( \cD \otimes \id ~ + ~  \id \otimes  \cD )
\circ \wt{\D}\,.
\end{equation}

A direct computation shows that 
\begin{equation}
\label{wtDelta-exp}
\wt{\D} \big(  \bs^2 \,e^{\bs^{-2}\, \al} \big) =   \bs^2 \,e^{\bs^{-2}\, \al}~ \otimes~  \bs^2 \,e^{\bs^{-2}\, \al}\,,
\end{equation}
and the operation 
\begin{equation}
\label{mult-by-exp}
W ~ \mapsto ~ e^{\bs^{-2}\, \al}\, W ~ : ~  \bs^2  S(\bs^{-2} \, \cV) \to  \bs^2  \hat{S}(\bs^{-2} \, \cV)
\end{equation}
satisfies the identity
\begin{equation}
\label{wtDelta-exp-iden}
\wt{\D}  \big( e^{\bs^{-2}\, \al} \, W \big) = 
(e^{\bs^{-2}\, \al} \otimes e^{\bs^{-2}\, \al}) \,\wt{\D} (W) \,. 
\end{equation}

Hence the operation 
\begin{equation}
\label{tw-coder-full}
e^{- \bs^{-2}\, \al} \, \cD \, e^{\bs^{-2}\, \al} 
~:~ \bs^2  S(\bs^{-2} \, \cV) \to  \bs^2  \hat{S}(\bs^{-2} \, \cV)
\end{equation}
satisfy the identity
\begin{equation}
\label{wtDelta-tw-coder}
\wt{\D} \circ 
\big( e^{- \bs^{-2}\, \al} \, \cD \, e^{\bs^{-2}\, \al} \big) ~ = ~
\Big(
\big( e^{- \bs^{-2}\, \al} \, \cD \, e^{\bs^{-2}\, \al} \big) ~\otimes~  \id ~~ + ~~
 \id ~ \otimes ~ \big( e^{- \bs^{-2}\, \al} \, \cD \, e^{\bs^{-2}\, \al} \big)  \Big)
\circ
\wt{\D}\,. 
\end{equation}

On the other hand, $\D$ is related to $\wt{\D}$ by the formula
\begin{equation}
\label{wtD-D}
\D(W) = \wt{D}(W)  -  \bs^2\, 1 \otimes W - W \otimes \bs^2\, 1\,, 
\qquad \forall ~~ W \in \bs^2 \und{S}(\bs^{-2}\, \cV)\,.
\end{equation}

Therefore, using \eqref{cD-1}, \eqref{wtDelta-tw-coder}, and \eqref{wtD-D}, we get 
$$
\D \circ \big( e^{- \bs^{-2}\, \al} \, \cD \, e^{\bs^{-2}\, \al} \big) W = 
\wt{\D} \circ \big( e^{- \bs^{-2}\, \al} \, \cD \, e^{\bs^{-2}\, \al} \big) W 
$$
$$
- \bs^{2} 1 \otimes \big( e^{- \bs^{-2}\, \al} \, \cD \, e^{\bs^{-2}\, \al} \big) W
-  \big( e^{- \bs^{-2}\, \al} \, \cD \, e^{\bs^{-2}\, \al} \big) W \otimes \bs^{2} 1=
$$
$$
\Big(
\big( e^{- \bs^{-2}\, \al} \, \cD \, e^{\bs^{-2}\, \al} \big) ~\otimes~ \id ~~ + ~~
\id ~ \otimes ~ \big( e^{- \bs^{-2}\, \al} \, \cD \, e^{\bs^{-2}\, \al} \big)  \Big) \circ
\wt{\D} (W)
$$
$$
- \bs^{2} 1 \otimes \big( e^{- \bs^{-2}\, \al} \, \cD \, e^{\bs^{-2}\, \al} \big) W
-  \big( e^{- \bs^{-2}\, \al} \, \cD \, e^{\bs^{-2}\, \al} \big) W \otimes \bs^{2} 1=
$$
$$
\Big(
\big( e^{- \bs^{-2}\, \al} \, \cD \, e^{\bs^{-2}\, \al} \big) ~\otimes~ \id ~~ + ~~
\id ~ \otimes ~ \big( e^{- \bs^{-2}\, \al} \, \cD \, e^{\bs^{-2}\, \al} \big)  \Big) \circ
\D (W)
$$
$$ +
\Big(
\big( e^{- \bs^{-2}\, \al} \, \cD \, e^{\bs^{-2}\, \al} \big) ~\otimes~ \id ~~ + ~~
 \id ~ \otimes ~ \big( e^{- \bs^{-2}\, \al} \, \cD \, e^{\bs^{-2}\, \al} \big)  \Big)
(W \otimes \bs^{2} \, 1 +  \bs^{2} \, 1 \otimes W) 
$$
$$
- \bs^{2} 1 \otimes \big( e^{- \bs^{-2}\, \al} \, \cD \, e^{\bs^{-2}\, \al} \big) W
-  \big( e^{- \bs^{-2}\, \al} \, \cD \, e^{\bs^{-2}\, \al} \big) W \otimes \bs^{2} 1=
$$
$$
\Big(
\big( e^{- \bs^{-2}\, \al} \, \cD \, e^{\bs^{-2}\, \al} \big) ~\otimes~ \id ~~ + ~~
 \id ~ \otimes ~ \big( e^{- \bs^{-2}\, \al} \, \cD \, e^{\bs^{-2}\, \al} \big)  \Big) \circ
\D (W)
$$
$$ 
+ \, W ~ \otimes ~
e^{- \bs^{-2}\, \al} \, \cD \, \big( \bs^2 (e^{\bs^{-2}\, \al}  -1) \big)  ~~ + ~~
e^{- \bs^{-2}\, \al} \, \cD \, \big( \bs^2 (e^{\bs^{-2}\, \al}  -1) \big) ~\otimes~
W\,.
$$
Thus condition \eqref{do-not-move} implies that the operation 
\eqref{cD-tw} is indeed a derivation of $\La^2 \coCom(\cV)$\,.

Let us now prove claim  {\bf iv)}. Due to claim {\bf iii)} the operation
$$
e^{ - \bs^{-2}\, \al} \, \msQ \,  e^{\bs^{-2}\, \al}
$$
is a coderivation of  $\La^2 \coCom(\cV)$. Hence, it suffices to show that
\begin{equation}
\label{p-circ}
p_{\cV} \circ  \big( \msQ \big) \,  e^{\bs^{-2}\, \al} = 
p_{\cV} \circ  \msQ^{\al}\,.
\end{equation}

Equation \eqref{p-circ} directly follows from \eqref{tw-diff} and \eqref{tw-bracket}. 
Thus  claim  {\bf iv)} follows.

Claim {\bf v)} is now a straightforward consequence of claims {\bf ii)} - {\bf iv)}.

Theorem \ref{thm:twisting} is proven.  
\end{proof}

We say that the cochain $\cD^{\al}$ in  
\eqref{CE-twisting} is obtained from $\cD$ via {\it twisting} by the 
Maurer-Cartan element $\al$. 

Theorem \ref{thm:twisting} has the following corollary
\begin{cor}
\label{cor:twisting}
Let $\cV$ be a filtered $\La\Lie$-algebra, $\al$ be a Maurer-Cartan 
element of $\cV$, and $\cV^{\al}$ be the   $\La\Lie$-algebra which 
is obtained from $\cV$ via twisting by $\al$.
If $\cD$ is a cochain in the deformation complex for $\cV$ satisfying the condition 
\begin{equation}
\label{cond-weaker}
\cD\big( \bs^2 ( \bs^{-2}\, \al )^n \big) = 0 \qquad \forall~~ n \ge 1
\end{equation}
then
\begin{itemize}

\item the operator 
\begin{equation}
\label{tw-cD}
\cD^{\al} : =
e^{ - \bs^{-2}\, \al }\, \cD\, e^{ \bs^{-2}\, \al } : \La^2 \coCom(\cV)  
\to  \La^2 \coCom(\cV) 
\end{equation}
is a cochain in the deformation complex 
\begin{equation}
\label{CE-tw}
\Big( \coDer \big( \La^2 \coCom(\cV^{\al})  \big),~ \msQ^{\al} \Big)
\end{equation}
for $\cV^{\al}$;

\item we have 
\begin{equation}
\label{tw-chain-map}
[ \msQ, \cD]^{\al} = [ \msQ^{\al}, \cD^{\al}]\,;
\end{equation}

\item finally, for every $\La\Lie_{\infty}$-derivation\footnote{Recall that degree zero 
cocycles in the deformation complex of a $\La\Lie$-algebra $\cV$ are called 
$\La\Lie_{\infty}$-derivations of $\cV$\,.} 
$\cD$ of $\cV$ satisfying
\eqref{cond-weaker},
the cochain $\cD^{\al}$ \eqref{tw-cD} is a $\La\Lie_{\infty}$-derivation 
of $\cV^{\al}$\,.

\end{itemize}

$\qed$
\end{cor}

\subsection{Twisting operation for cochains in the deformation complex of a 
Gerstenhaber algebra}
\label{app:tw-Ger}

We now assume that $\cV$ is a dg  Gerstenhaber algebra equipped with a complete 
descending filtration
\begin{equation}
\label{filtr-cV-Ger}
\cV \supset \dots \supset \cF_0 \cV \supset \cF_1 \cV \supset  \cF_2 \cV \supset \dots\,,
\end{equation}
\begin{equation}
\label{Ger-cV-complete}
\cV = \lim_k \cV ~ \big/  ~\cF_k \cV\,,
\end{equation}
which is compatible with the differential $\pa$, the $\La\Lie$-bracket $\{~,~\}$ and 
the multiplication $\cdot$ on $\cV$\,.

Since every  Gerstenhaber algebra is also a $\La\Lie$-algebra, we have the 
notion of Maurer-Cartan elements in $\cV$\,. 
Furthermore, given a Maurer-Cartan element $\al$ of $\cV$, we denote 
by $\cV^{\al}$ the dg  Gerstenhaber algebra which is obtained from $\cV$ via twisting 
by $\al$\,. In other words, $\cV^{\al} = \cV$ as the graded vector space, the differential 
$\pa^{\al}$ on $\cV^{\al}$ is given by equation \eqref{tw-diff}. Finally $\cV^{\al}$ and $\cV$ 
share the same $\La\Lie$-bracket $\{~,~\}$ and the same multiplication $\cdot$\,.

Just as for $\La\Lie$-algebras,  we consider
the following element of the completion of $\Ger^{\vee}(\cV)$ 
\begin{equation}
\label{the-elem-Ger}
\bs^2\, (e^{\bs^{-2}\, \al} -1) = \sum_{n=1}^{\infty} \frac{1}{n!} 
 \bs^2\,(\bs^{-2}\, \al)^{n}
\end{equation}
and define the subspace 
of coderivations $\cD \in  \coDer \big( \Ger^{\vee}(\cV) \big)$ 
satisfying the additional condition\footnote{We tacitly assume that our 
coderivations are compatible with the filtration on $\Ger^{\vee}(\cV)$
coming from  \eqref{filtr-cV-Ger}.}
\begin{equation}
\label{do-not-move-Ger}
\cD\, \bs^2 \, ( e^{\bs^{-2}\, \al} -1) = 0\,.
\end{equation}
This subspace is obviously closed with respect to the commutator.

We now present the following analogue of Theorem \ref{thm:twisting}
\begin{thm}
\label{thm:twisting-Ger}
Let $\cV$ be a filtered  Gerstenhaber algebra, $\al$ be a Maurer-Cartan 
element of $\cV$\,, $p_{\cV} : \Ger^{\vee}(\cV) \to \cV$ be the 
canonical projection,
and $\msQ$ (resp. $\msQ^{\al}$) be the codifferential 
on  $\Ger^{\vee}(\cV)$   (resp.  $\Ger^{\vee}(\cV^{\al})$) corresponding 
to the dg $\Ger$-structures on $\cV$ (resp. $\cV^{\al}$). 
Let us also denote by 
\begin{equation}
\label{coder-cond-Ger}
\coDer \big( \Ger^{\vee}(\cV) \big)_{\al} 
\end{equation}
the subspace of coderivations of $\Ger^{\vee}(\cV)$ satisfying 
condition \eqref{do-not-move-Ger}. Then  

\begin{itemize}

\item[{\bf i)}] Condition \eqref{do-not-move-Ger} on coderivations is equivalent to 
\begin{equation}
\label{do-not-move1-Ger}
\sum_{n=1}^{\infty}\frac{1}{n!}\,
p_{\cV} \circ \cD \,  \big( \bs^2\, (\bs^{-2}\, \al)^{n} \big) = 0\,.
\end{equation}

\item[{\bf ii)}]  The codifferential $\msQ$
satisfies Condition  \eqref{do-not-move-Ger}.

\item[{\bf iii)}] For every coderivation $\cD$ in \eqref{coder-cond-Ger}  
the operation 
\begin{equation}
\label{cD-tw-Ger}
e^{ - \bs^{-2}\, \al} \, \cD \,  e^{\bs^{-2}\, \al} :  
 \Ger^{\vee}(\cV) \to  \Ger^{\vee}(\cV)
\end{equation}
is a coderivation of  $\Ger^{\vee}(\cV)$.

\item[{\bf iv)}]  The codifferential $\msQ^{\al}$ is related 
to $\msQ$ by the formula: 
\begin{equation}
\label{msQ-tw-Ger}
\msQ^{\al} ~ = ~ 
e^{ - \bs^{-2}\, \al} \, \msQ \,  e^{\bs^{-2}\, \al}\,. 
\end{equation}

\item[{\bf v)}]  The subspace \eqref{coder-cond-Ger} is a subcomplex
of the deformation complex for the  Gerstenhaber algebra $\cV$\,.
Furthermore, the assignment 
\begin{equation}
\label{Ger-twisting}
\cD ~ \mapsto ~ \cD^{\al} =
e^{ - \bs^{-2}\, \al} \, \cD \,  e^{\bs^{-2}\, \al}  
\end{equation}
defines a map of cochain complexes 
\begin{equation}
\label{Ger-twisting-map}
\coDer \big( \Ger^{\vee} (\cV) \big)_{\al} ~ \to ~
  \coDer \big( \Ger^{\vee} (\cV^{\al}) \big)
\end{equation}
from \eqref{coder-cond-Ger} to the deformation complex 
$\coDer \big( \Ger^{\vee}(\cV^{\al}) \big)$ of the Gerstenhaber algebra 
$\cV^{\al}$. 

\end{itemize}
\end{thm}
\begin{proof}
Proofs of all these statements are obtained by incorporating only
minor modifications in the corresponding proof of Theorem \ref{thm:twisting}. 

The only exception is probably claim {\bf iii)}. In this case, we 
also have to prove that
$$
\cD^{\al} = e^{ - \bs^{-2}\, \al} \, \cD \,  e^{\bs^{-2}\, \al} 
$$
respects the cobracket
$$
\D_{\{~,~\}} :  \Ger^{\vee}(\cV)  \to  \Ger^{\vee}(\cV) \otimes  \Ger^{\vee}(\cV) 
$$
on $ \Ger^{\vee}(\cV)$ in the sense of the equation
\begin{equation}
\label{cDtw-cobrack}
\D_{\{~,~\}} \circ \cD^{\al} =  (-1)^{|\cD|} (\cD^{\al} \otimes \id  + \id \otimes  \cD^{\al}) \circ \D_{\{~,~\}}\,.
\end{equation}

To prove this fact we observe that for any degree $2$ vector $\al \in \cV$
the operation 
$$
W \mapsto \bs^{-2} \al \, W :  \Ger^{\vee}(\cV) \to \Ger^{\vee}(\cV) 
$$
is a degree $0$ coderivation with respect to the cobracket $\D_{\{~,~\}}$\,. Hence for 
every vector $W$ in the completion $ \Ger^{\vee}(\cV)\,\hat{} $ of $\Ger^{\vee}(\cV)$
we have 
\begin{equation}
\label{cobrack-exp}
\D_{\{~,~\}}  \big( e^{\bs^{-2} \al} \, W \big) =  e^{\bs^{-2} \al} \otimes  e^{\bs^{-2} \al}\, \big( \D_{\{~,~\}}   W \big)\,.
\end{equation}

Thus \eqref{cDtw-cobrack} indeed holds. 

The compatibility of  $\cD^{\al}$ with the comultiplication on  $\Ger^{\vee}(\cV)$
is proven in the same way as for the case of $\La\Lie$-algebras. 
\end{proof}

~\\

\noindent\textsc{Department of Mathematics,
Temple University, \\
Wachman Hall Rm. 638\\
1805 N. Broad St.,\\
Philadelphia PA, 19122 USA \\
\emph{E-mail address:} {\bf vald@temple.edu}}

~\\

\noindent\textsc{Institut f\"ur Mathematik und Informatik\\
Universit\"at Greifswald\\
Walther-Rathenau-Strasse 47\\
17487 Greifswald, Germany
\emph{E-mail address:} {\bf rogersc@uni-greifswald.de} }

~\\

\noindent\textsc{University of Z\"urich, \\
Institute of Mathematics, \\
Winterthurerstrasse 190,\\
8057 Z\"urich, Switzerland  \\
\emph{E-mail address:} {\bf thomas.willwacher@math.uzh.ch}}

\end{document}